\documentclass[10pt]{article}
\usepackage{amsmath,amsfonts,amssymb,latexsym,epic,eepic}
\usepackage[overload]{empheq}
\usepackage[dvips]{graphicx,graphics,epsfig,color}
\usepackage{ifthen}
\usepackage{geometry}
\usepackage{amsthm}
\usepackage{mathrsfs}
\usepackage{tikz,pstricks,fp}
\usepackage{multirow}
\usepackage{array}
\usepackage{dsfont}
\newtheorem{thrm}{Theorem}[section]
\newtheorem{lem}[thrm]{Lemma}

\newtheorem{prop}[thrm]{Proposition}
\newtheorem{defi}{Definition}[section]
\newtheorem{rmrk}{Remark}[section]
\usepackage[colorlinks=true, pdfstartview=FitV, linkcolor=blue, citecolor=red, urlcolor=blue]{hyperref}
\def\dsp{\displaystyle}

\newcommand{\mcal}[1]{\mathcal{#1}}
\newcommand{\ds}{\displaystyle}

\newcommand{\weak}{\rightharpoonup}

\definecolor{vertf}{rgb}{0,0.55,0.1}
\definecolor{bclairf}{rgb}{0.40,0.65,0.89}
\definecolor{or}{rgb}{0.98,0.6,.1}

\newcommand{\disc}{\mathcal{D}}
\newcommand{\edge}{\sigma}

\newcommand{\edges}{{\mathcal E}}
\newcommand{\edgesd}{\tilde {\edges}}
\newcommand{\edgesint}{{\mathcal E}_{{\rm int}}}
\newcommand{\edgesext}{{\mathcal E}_{{\rm ext}}}
\newcommand{\edgesdint}{{\edgesd}_{{\rm int}}}
\newcommand{\edgesdext}{{\edgesd}_{{\rm ext}}}

\newcommand{\mesh}{{\mathcal M}}
\newcommand{\edged}{\epsilon}
\newcommand{\edgeedgeprime}{D_\edge|D_{\edge'}}
\newcommand{\fluxK}{F_{K,\edge}}

\newcommand{\fluxd}{F_{\edge,\edged}}
\newcommand{\rhomin}{\rho_{\rm min}}
\newcommand{\rhomax}{\rho_{\rm max}}

\newcommand{\eps}{\varepsilon}
\def\xR{{\mathbb R}}
\def\xC{{\rm C}}
\def\xN{{\mathbb N}}
\def\xH{{\rm H}}
\def\xW{{\rm W}}
\def\xL{{\rm L}} 
 
\newcommand{\bfn}{{\boldsymbol n}}
\newcommand{\bfu}{{\boldsymbol u}}
\newcommand{\bfa}{{\boldsymbol a}}
\newcommand{\bfb}{{\boldsymbol b}}
\newcommand{\bfphi}{{\boldsymbol \phi}}
\newcommand{\bfv}{{\boldsymbol v}}
\newcommand{\bfw}{{\boldsymbol w}}
\newcommand{\bfx}{{\boldsymbol x}}
\newcommand{\bfy}{{\boldsymbol y}}

\newcommand{\bfD}{{\boldsymbol D}}
\newcommand{\dv}{\partial}
\newcommand{\partd}{\eth}
\newcommand{\gradi}{{\boldsymbol \nabla}}
\newcommand{\lapi}{{\boldsymbol \Delta}}
\newcommand{\dive}{{\rm div}}
\newcommand{\divv}{\boldsymbol{\rm div}}
\newcommand{\norm}[1]{{\lVert #1 \rVert}}
\newcommand{\Ent}[1]{{\lfloor #1 \rfloor}}
\newcommand{\snorm}[1]{{| #1 |}}

\newcommand{\fv}{{\edges,{\rm FV}}}

\newcommand{\brok}{{\edges,{\rm b}}}

\newcommand{\m}{^{(m)}}
\newcommand{\dx}{\,{\rm d}\bfx}
\newcommand{\dedge}{\,{\rm d}\gamma}

\newcommand{\dt}{\,{\rm d}t}

\newcommand{\eg}{\emph{e.g.}}
\newcommand{\ie}{\emph{i.e.}}

\newcommand{\Ds}{{\scalebox{0.6}{$D_\edge$}}}
\newcommand{\Dsp}{{\scalebox{0.6}{$D_{\edge'}$}}}

\title
{\textbf{A Convergent Staggered Scheme for the Variable Density Incompressible Navier-Stokes Equations}}

\author{J.C. Latch\'e\footnote{Institut de Radioprotection et de S\^{u}ret\'{e} Nucl\'{e}aire (IRSN). Email: jean-claude.latche@irsn.fr} \hspace{2ex} K. Saleh\footnote{Universit\'e de Lyon, CNRS UMR 5208, Universit\'e Lyon 1, Institut Camille Jordan, 43 bd 11 novembre 1918; F-69622 Villeurbanne cedex, France. Email: saleh@math.univ-lyon1.fr}}

\date{}

\begin{document}

\maketitle

\begin{abstract}
In this paper, we analyze a scheme for the time-dependent variable density Navier-Stokes equations.
The algorithm is implicit in time, and the space approximation is based on a low-order staggered non-conforming finite element, the so-called Rannacher-Turek element.
The convection term in the momentum balance equation is discretized by a finite volume technique, in such a way that a solution obeys a discrete kinetic energy balance, and the mass balance is approximated by an upwind finite volume method.
We first show that the scheme preserves the stability properties of the continuous problem ($\xL^\infty$-estimate for the density, $\xL^\infty(\xL^2)$- and $\xL^2(\xH^1)$-estimates for the velocity), which yields, by a topological degree technique, the existence of a solution.
Then, invoking compactness arguments and passing to the limit in the scheme, we prove that any sequence of solutions (obtained with a sequence of discretizations the space and time step of which tend to zero) converges up to the extraction of a subsequence to a weak solution of the continuous problem.
\end{abstract}

\noindent \textbf{Key-words\,:} Variable density, Navier-Stokes equations, Staggered schemes, Convergence analysis

%
%
%
%
\section{Introduction}

Since seminal papers published from the middle of the sixties \cite{har-65-num,har-68-num,har-71-num}, low-order staggered schemes for fluid flow computations have received a considerable attention.
This interest is essentially motivated by the fact that they combine a low computational cost with the so-called {\em inf-sup} or LB stability condition (see \eg\ \cite{gir-86-fin}), which prevents from the odd-even decoupling of the pressure in the incompressible limit.
In addition, they may be combined, still keeping basically the same order of accuracy, with finite volume approximations for possible additional conservation equations, which allows, thanks to standard techniques, to obtain discrete convection operators satisfying maximum principles (\eg\ \cite{lar-91-how}).

\medskip
Two different types of space discretizations fall in the class of staggered approximations.
The first one, essentially able to cope with structured meshes (with cell faces normal to the coordinate axes), is the well-known MAC scheme \cite{har-65-num,har-68-num,har-71-num}; it is characterized by the fact that the unknowns for the $i^{th}$ component of the velocity are associated with the cell faces normal to the $i^{th}$ coordinate axis.
The second type of approximation has been developed in the finite-element framework; it is based on general simplices (for the so-called Crouzeix-Raviart element \cite{cro-73-con}) or on general quadrilaterals or hexahedra (for the so-called Rannacher-Turek element \cite{ran-92-sim}).
The velocity unknowns are the same for each component, and are associated with all the faces of the mesh (so, compared to the MAC scheme, the price to pay for the generality of the mesh is a multiplication by the space dimension $d$ of the number of unknowns).

\medskip
Recently, for MAC, Crouzeix-Raviart and Rannacher-Turek approximations, discretizations of the convection operator in the momentum balance equation have been developed with the aim to obtain a scheme preserving the kinetic energy balance \cite{her-10-kin, ans-11-anl, boy-14-sta}.
These techniques, implemented in the open-source software ISIS \cite{isis}, have brought many outcomes, both from the theoretical and the practical points of view.
First, the kinetic energy conservation property yields stability estimates (see \cite{ans-11-anl, boy-14-sta} for quasi-incompressible flow and \cite{gal-08-unc, gra-16-unc} for barotropic and non-barotropic compressible Navier-Stokes equations), and has been observed in numerical experiments to actually dramatically increase the reliability of the scheme.
Second, the non-dissipation of the kinetic energy is a prerequisite for numerical schemes for Large Eddy Simulation (\eg\ \cite{nic-00-con, des-08-hig, mor-10-ske, boy-14-sta}), and a theoretical proof of this feature thus strongly supports this kind of application.
Finally, for Euler's equations, having at hand a discrete kinetic energy balance has been a key point in \cite{her-14-imp, gra-16-unc} to derive a consistent staggered scheme preserving the convex set of admissible states.

\medskip
The discrete form of the convection operator, which is similar in all these applications and for all the considered space discretizations, may thus be seen as a decisive building block of a class of schemes able to cope with all regimes, \ie\ from incompressible to compressible high Mach number flows.
It is a finite volume form (see \cite{sch-89-non,sch-96-ano} for a similar development for the finite element context, restricted to constant density flows), written on {\em dual cells}, \ie\ cells centered at the location of the velocity unknowns, namely the faces.
The difficulty for its construction lies in the fact that, as in the continuous case, the derivation of the kinetic energy identity needs that a mass balance equation be satisfied on the same (dual) cells, while the mass balance in the scheme is naturally written on the primal cells.
We thus have developed a procedure to define the density on the dual meshes and the mass fluxes through the dual faces from the primal cell density and the primal faces mass fluxes, which ensures a discrete mass balance.
However, especially for the Rannacher-Turek approximation, the quantities associated with the dual mesh are defined only through necessary conditions to obtain the desired mass conservation, in a way which is somehow reminiscent of the techniques used for the derivation of the mimetic schemes.
As a consequence, we obtain a convection operator the definition of which is not in closed form, at least at first glance, and the consistency of which is far from obvious.

\medskip
The aim of this paper is to prove this consistency property.
More precisely speaking, on a model problem and with a given scheme, we prove that the limit of a converging sequence of solutions obtained with a sequence of discretizations with vanishing space and time steps is necessarily a weak solution to the problem at hand.
For this latter, we choose the time-dependent variable density incompressible Navier-Stokes equations, which (from the consistency point of view) retain the essential mathematical difficulties of compressible flows; indeed, the partial differential equations in which we have to pass to the limit, namely the mass and momentum balance equations, are the same as for compressible flows.
For the scheme, we focus here on the Rannacher-Turek discretization, and on an implicit time discretization.

\medskip
In addition, we also prove estimates on the solution and, by compactness arguments, the existence of converging sequences of solutions (which, of course, would be more difficult for compressible Navier-Stokes equations).
The result presented here is thus in fact a convergence result for the proposed scheme on time-dependent variable density incompressible Navier-Stokes equations, which seems to be rather new in the literature; indeed, only one similar result is known to us, for a different (Discontinuous Galerkin) space approximation \cite{liu-07-conv}. 
Note also that, as a by-product, we obtain the existence of weak solutions to the continuous problem, without invoking arguments of the  continuous theory \cite{lio-96-mat} itself, except a result issued from the analysis of renormalized solutions of the transport equation \cite{dip-89-ord}.
Various extensions of this work are ongoing: for instance, the same convergence result may be proven for the MAC scheme, with rather simpler arguments, and consistency may be extended for Euler's equations.

\medskip
This paper is organized as follows.
We state the continuous problem and recall its essential properties in Section \ref{sec:cont}, then the space discretization and the scheme are given in Sections \ref{sec:meshes} and \ref{sec:scheme} respectively.
The convergence theorem is stated in Section \ref{sec:conv}, and the three next sections are devoted to its proof: we first gather in Section \ref{sec:math} some useful mathematical tools, then establish the estimates satisfied by the discrete solution and its existence (Section \ref{sec:prop}), and, finally, prove the theorem (Section \ref{sec:proof}).
%
%
\section{The continuous problem}\label{sec:cont}

The continuous problem addressed in this paper reads, in its strong form:
\begin{subequations}\label{eq:pb}
\begin{align} \displaystyle \label{eq:mass} &
\partial_t\rho +\dive (\rho\bfu)=0,
\\[0ex]\label{eq:mom} &
\partial_t (\rho\bfu)+\divv (\rho\bfu \otimes \bfu)-\lapi \bfu +\gradi p=0,
\\[0ex] \label{eq:div} &
\dive \bfu = 0.
\end{align}
\end{subequations}
This problem is posed for $(\bfx,t)$ in $\Omega\times(0,T)$ where $T\in\xR_+^*$ and $\Omega$ is an open bounded connected subset of $\xR^d$, with $d\in\lbrace2,3\rbrace$, which is polygonal if $d=2$ and polyhedral if $d=3$.
The variables $\rho$, $\bfu=(u_1,\ldots, u_d)^T$ and $p$ are respectively the density, the velocity and the pressure of the flow.
The three above equations respectively express the mass conservation, the momentum balance and the incompressibility of the fluid.
This system is supplemented with initial and boundary conditions:
\[
\bfu|_{\dv\Omega}=0, \qquad \bfu|_{t=0}=\bfu_0, \qquad \rho|_{t=0}=\rho_0. 
\]
Let us suppose that the initial data satisfies the following properties:
\begin{subequations}\label{eq:H_ini}
\begin{align} \label{eq:H_rho} &
\begin{matrix}
\rho_0 \in \xL^\infty(\Omega),\ 0< \rhomin \leq \rho_0 \leq \rhomax,\mbox{ with } \hfill \\
\rhomin = {\rm ess~min}_{\bfx \in \Omega} \rho_0(\bfx), \ \rhomax = {\rm ess~sup}_{\bfx \in \Omega} \rho_0(\bfx),
\hfill \end{matrix}
\\ \label{eq:H_u} &
\bfu_0 \in \xL^2(\Omega)^d.
\end{align}
\end{subequations}
A well-known consequence of equations \eqref{eq:mass} and \eqref{eq:div} is the following maximum principle:
\[
\rhomin \leq \rho(\bfx,t) \leq \rhomax, \qquad  \text{for a.e.~} (\bfx,t) \in \Omega\times(0,T), 
\]
which shows that the natural regularity for $\rho$ is $\rho \in \xL^{\infty}(\Omega\times (0,T))$.
For the velocity $\bfu$, a classical formal calculation allows to derive natural estimates for smooth solutions.
Taking the scalar product of \eqref{eq:mom} by $\bfu$ and using twice the mass conservation equation \eqref{eq:mass} yields
\[
\dv_t(\frac 1 2 \rho |\bfu|^2)+\dive (\frac 1 2 \rho |\bfu|^2 \bfu)-\lapi \bfu \cdot \bfu  + \gradi p \cdot \bfu  =0. 
\]
Integrating over $\Omega$, one gets, since $\dive \bfu=0$ and $\bfu|_{\dv\Omega}=0$, that, for all $t \in (0,T)$,
\[
\frac d {dt} \int_\Omega \frac 1 2 \rho(\bfx,t)\, |\bfu(\bfx,t)|^2 \dx
+ \int_\Omega \gradi \bfu(\bfx,t) :\gradi \bfu(\bfx,t) \dx =0.
\]
Integrating over the time yields
\[
\int_\Omega \frac 1 2 \, \rho(\bfx,\tilde t)\, |\bfu(\bfx,\tilde t)|^2 \dx
+ \int_0^{\tilde t} \int_\Omega |\gradi \bfu(\bfx,t)|^2 \dx \dt =
\int_\Omega \frac 1 2 \, \rho_0(\bfx)\, |\bfu_0(\bfx)|^2 \dx, \quad \forall \tilde t\in (0,T). 
\]
Since the density is bounded from below by a positive constant, this shows that the natural regularity for $\bfu$ is $\bfu \in \xL^{\infty}((0,T);\xL^2(\Omega)^d) \cap \xL^2((0,T);\xH_0^1(\Omega)^d)$.
This leads to define the weak solutions to problem \eqref{eq:pb} as follows.

\begin{defi} \label{def:weaksol}
Let $\rho_0 \in \xL^{\infty}(\Omega)$ such that  $\rho_0 >0$ for \emph{a.e.} $\bfx \in \Omega$, and let $\bfu_0\in \xL^2(\Omega)^d$.
A pair $(\rho,\bfu)$ is a weak solution of problem \eqref{eq:pb} if it satisfies the following properties:
\begin{enumerate}
\item[(i)] $\rho \in \lbrace \rho \in \xL^{\infty}(\Omega\times(0,T)),~\rho >0 \text{ a.e. in } \Omega\times(0,T) \rbrace$.

\medskip
\item[(ii)] $\bfu \in \lbrace \bfu \in \xL^{\infty}((0,T);\xL^2(\Omega)^d) \cap \xL^2((0,T);\xH_0^1(\Omega)^d),~\dive\bfu =0 \rbrace$.

\medskip
\item[(iii)] For all $\phi$ in $ \xC^{\infty}_c(\Omega \times [0,T) )$,
\[
- \int_0^T   \int_\Omega \rho(\bfx,t) \Big (\dv_t \phi(\bfx,t) + \bfu(\bfx,t) \cdot  \gradi \phi(\bfx,t)\Big) \dx \dt 
 = \int_\Omega \rho_0(\bfx) \phi(\bfx,0) \dx. 
\]

\medskip
\item[(iv)]  For all $\bfv$ in $ \lbrace \bfv \in \xC^{\infty}_c(\Omega \times [0,T) )^d, \dive\bfv=0 \rbrace $, 
\begin{multline*}
\int_0^T  \int_\Omega \Bigl( - \rho(\bfx,t) \bfu(\bfx,t) \cdot \dv_t \bfv(\bfx,t)
- (\rho(\bfx,t) \bfu(\bfx,t) \otimes \bfu(\bfx,t)) : \gradi \bfv(\bfx,t)
\\
+ \gradi \bfu(\bfx,t) : \gradi \bfv(\bfx,t) \Bigr) \dx \dt
= \int_\Omega \rho_0(\bfx) \bfu_0(\bfx) \cdot \bfv(\bfx,0) \dx. 
\end{multline*}
\end{enumerate}
\end{defi}

\medskip

\begin{rmrk}
Thanks to a theorem due to de Rham, one can actually prove that problem \eqref{eq:pb} is satisfied in the distributional sense, with $p\in\xW^{-1,\infty}((0,T);\xL^2_0(\Omega))$ where $\xL^2_0(\Omega)=\xL^2(\Omega)/\xR$. See for instance \cite{boy-13-math}.
\end{rmrk}

%
%
\section{Definition of the meshes} \label{sec:meshes}

Let $\Omega$, the computational domain, be an open bounded subset of $\xR^d$, for $d\in\lbrace2,3\rbrace$, and let us suppose that $\Omega$ is polygonal, for $d=2$ and polyhedral, for $d=3$.
We denote by $\dv \Omega =  \overline{\Omega}\setminus\Omega$ its boundary.
In the following, the notation $|K|$ or $|\edge|$ stands indifferently for the $d$-dimensional or the $(d-1)$-dimensional measure of the subset $K$ of $\xR^d$ or $\edge$ of $\xR^{d-1}$ respectively.

\begin{defi}[Staggered discretization] \label{def:disc}
A staggered discretization of $\Omega$, denoted by $\disc$, is given by $\disc=(\mesh,\edges)$, where:
\begin{itemize}
\item[-] $\mesh$, the primal mesh, is a finite family of non empty convex quadrilaterals (d=2) or hexahedra (d=3) of $\Omega$ such that $\overline{\Omega}= \dsp{\cup_{K \in \mesh} \overline K}$.

\medskip
\item[-] For any $K\in\mesh$, let $\dv K  = \overline K\setminus K$ be the boundary of $K$.
The surface $\dv K$ is the union of bounded subsets of hyperplanes of $\xR^d$, which we call faces.
We denote by $\edges$ the set of faces of the mesh, and we suppose that two neighboring cells share a whole face: for all $\edge\in\edges$, either $\edge\subset \dv\Omega$ or there exists $(K,L)\in \mesh^2$ with $K \neq L$ such that $\overline K \cap \overline L  = \overline \edge$; we denote in the latter case $\edge = K|L$.
We denote by $\edgesext$ and $\edgesint$ the set of external and internal faces: $\edgesext=\lbrace \edge \in \edges, \edge \subset \dv \Omega \rbrace$ and $\edgesint=\edges \setminus \edgesext$.
For $K \in \mesh$, $\edges(K)$ stands for the set of faces of $K$.

\medskip
\item[-] We define a dual mesh associated with the faces $\edge\in\edges$ as follows.
When $K\in\mesh$ is a rectangle or a cuboid, for $\edge \in \edges(K)$, we define $D_{K,\edge}$ as the cone with basis $\edge$ and with vertex the mass center of $K$ (see Figure \ref{fig:mesh}).
We thus obtain a partition of $K$ in $2d$ sub-volumes, each sub-volume having the same measure $| D_{K,\edge}|= |K|/(2d)$.
We extend this definition to general quadrangles and hexahedra, by supposing that we have built a partition still of equal-volume sub-cells, and with the same connectivities.
For $\edge \in \edgesint$, $\edge=K|L$, we now define the dual (or diamond) cell $D_\edge$ associated with $\edge$ by $D_\edge=D_{K,\edge} \cup D_{L,\edge}$.
For $\edge\in\edges(K)\cap\edgesext$, we define $D_\edge=D_{K,\edge}$.
\end{itemize}
\end{defi}

\begin{rmrk}[Dual mesh and general cells]
Note that, for a general mesh, the shape of the dual cells does not need to be specified.
In addition, for a general quadrangle $K$, the definition of the volumes $\{D_{K,\edge},\ \edge \in \edges(K)\}$ is of course possible, but $D_{K,\edge}$ may be no longer a cone; indeed, if $K$ is far from a parallelogram, it may not be possible to build a cone having $\edge$ as basis, the opposite vertex lying in $K$ and a volume equal to $|K|/(2d)$.
\end{rmrk}

We denote by $\edgesd(D_\edge)$ the set of faces of $D_\edge$, and by $\edged=D_\edge|D_{\edge'}$ the face separating two diamond cells $D_\edge$ and $D_{\edge'}$.
As for the primal mesh, we denote by $\edgesdint$ the set of dual faces included in the domain and by $\edgesdext$ the set of dual faces lying on the boundary $\dv \Omega$.
In this latter case, there exists $\edge\in\edgesext$ such that $\edged=\edge$.
The unit vector normal to $\edge \in \edges(K)$ outward $K$ is denoted by $\bfn_{K,\edge}$.

\begin{figure}[ht]
\begin{center}
\newgray{grayml}{.8}
\newgray{grayc}{.97}
\psset{unit=1.0cm}
\begin{pspicture}(-2,0)(12,7)
\rput[bl](0,1){
   \pspolygon*[linecolor=grayml](3.244,3.04)(6,2)(7.137931034,4.275862069)(5.4,5)
   \rput[bl](4.6,3.5){{$D_\edge$}}
   \pspolygon*[linecolor=grayc](1,1)(3.244,3.04)(6,2)(3.7,0.5)
   \rput[bl](3.2,0.8){{$D_{\edge'}$}}
   \rput[bl]{10}(2.6,1.4){$\edge'=K|M$}
   \psline[linecolor=black, linewidth=2pt]{-}(1,1)(6,2)(5.4,5)(0.7,4)(1,1)
   \psline[linecolor=black, linewidth=0.5pt]{-}(1,1)(5.4,5)
   \psline[linecolor=black, linewidth=0.5pt]{-}(6,2)(0.7,4)
   \rput[bl](1.1,1.6){{$K$}}
   \psline[linecolor=black, linewidth=2pt]{-}(6,2)(9,3.5)(8,6)(5.4,5)
   \psline[linecolor=black, linewidth=0.5pt]{-}(6,2)(8,6)
   \psline[linecolor=black, linewidth=0.5pt]{-}(9,3.5)(5.4,5)
   \rput[bl](7.3,3){{$L$}}
   \psline[linecolor=black, linewidth=2pt]{-}(1,1)(6,2)(6.4,0)(1.4,-1)(1,1)
   \psline[linecolor=black, linewidth=0.5pt]{-}(1,1)(6.4,0)
   \psline[linecolor=black, linewidth=0.5pt]{-}(6,2)(1.4,-1)
   \rput[bl](1.5,-0.5){{$M$}}
   \psline[linecolor=black, linewidth=2pt]{-}(6,2)(9,3.5)(8,0)(6.4,0)
   \psline[linecolor=black, linewidth=0.5pt]{-}(6,2)(8,0)
   \psline[linecolor=black, linewidth=0.5pt]{-}(9,3.5)(6.4,0)
   \rput[bl](7.5,1){{$N$}}
   \rput[bl]{-79}(5.65,4.15){$\edge=K|L$}
   \rput[bl]{-19}(3.8,2.8){$\edged=D_\edge|D_{\edge'}$}
}
\end{pspicture}
\end{center}
\caption{Notations for a staggered discretization}\label{fig:mesh}
\end{figure}
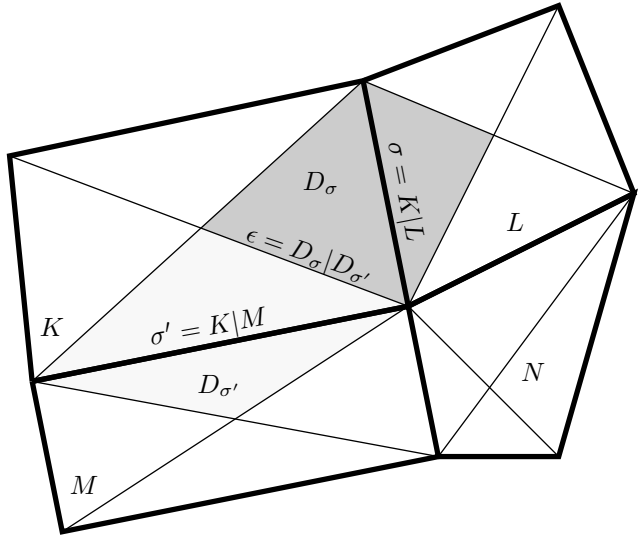

\medskip
For $K\in\mesh$, we denote by $h_K$ the diameter of $K$.
The size of the discretization is defined by
\begin{equation} \label{eq:def_h}
h_\disc= \sup\,\bigl\lbrace h_K,\ K\in \mesh \bigr\rbrace.
\end{equation}
In addition, for the consistency of the finite element approximation of the diffusion term, we need a measure of the difference between the cells of $\mesh$ and parallelograms ($d=2$) or parallelotopes ($d=3$), as defined in \cite{ran-92-sim}.
For $K \in \mesh$, we denote by $\bar \alpha_K$ the maximum of the angles between the normal vectors of opposite faces, choosing the orientation which maximize the angle, and set $\alpha_K = \pi - \bar \alpha_K$ (so $\alpha_K=0$ if $K$ is a parallelogram or a parallelotope, and $\alpha_K> 0 $ otherwise).
Then we define $\alpha_\disc$ as:
\begin{equation} \label{eq:def_alpha}
\alpha_\disc = \max\, \bigl\lbrace \alpha_K,\ K \in \mesh \bigr\rbrace.
\end{equation}
For $K\in\mesh$, we denote by $r_K$ the diameter of the largest ball included in $K$.
We define the real number $\theta_\mesh$ by:
\[
\theta_\mesh = \max\, \Bigl\lbrace \frac {h_K}{r_K},\ K\in \mesh \Bigr\rbrace.
\]
For $\edge \in \edges$, we denote by $h_\Ds$ the diameter of $D_\edge$ and by $r_\Ds$ the diameter of the largest ball included in $D_\edge$, and we define $\theta_{\edges,1}$ by:
\begin{equation} \label{eq:theta_edges1}
\theta_{\edges,1} = \max\, \Bigl\lbrace \frac {r_\Ds}{h_{\Dsp}},
\ \edge,\edge' \in \edges \mbox{ such that } \dv D_\edge \cap \dv D_{\edge'}\neq \emptyset  \Bigr\rbrace .
\end{equation}
The parameter $\theta_{\edges,2}$ is defined by:
\begin{equation} \label{eq:theta_edges2}
\theta_{\edges,2}=\max\, \Bigl\lbrace \frac{|\dv D_{K,\edge}|}{|\dv K|},\ K \in \mesh,\ \edge \in \edges(K) \Bigr\rbrace .
\end{equation}
Finally, we also need to introduce the following quantity:
\[
\theta_{\edges,3}=\max\, \Bigl\lbrace \frac{h_\disc}{r_\Ds},\ \edge \in \edges \Bigr\rbrace .
\]

The regularity of the discretization is measured through the following positive real number:
\begin{equation}\label{eq:reg}
 \theta_\disc = \max\, \bigl\lbrace \theta_\mesh,\ \theta_{\edges,1},\ \theta_{\edges,2}, \ \theta_{\edges,3}\bigr\rbrace.
\end{equation}
The real number $\theta_\mesh$ is a classical measure of the regularity of the primal mesh.
In 2D, an inequality of the form $\theta_\mesh\leq C$ is the classical uniform-shape condition for ${\mathcal Q}_1$-elements (see \cite{gir-86-fin}).
Observe that, by construction of the dual mesh which imposes to the half-diamond cells $D_{K,\edge}$ to be of equal volume in every $K$, if $\theta_\mesh$ is bounded, we may suppose that a similar measure of the uniform-shape regularity for the dual mesh is also bounded.
The real number $\theta_{\edges,1}$ is an additional measure of the regularity of the dual mesh, which characterizes the difference of size between two neighboring cells.
The parameter $\theta_{\edges,2}$  measures the regularity of the dual faces.
As shown in the following, the boundedness of $\theta_{\edges,1}$ and $\theta_{\edges,2}$ for a sequence of discretizations is used to obtain compactness results for the numerical scheme.
Finally, still for a sequence of discretizations, imposing to $\theta_{\edges,3}$ to be bounded is a quasi-uniformity assumption of the mesh, which is necessary to justify inverse inequalities used in the proof of Lemma \ref{lmm:convconv}.
%
%
\section{The scheme} \label{sec:scheme}

\subsection{General form of the scheme}

Let us consider a uniform partition $0=t_0 < t_1 <\ldots < t_N=T$ of the time interval $(0,T)$, and let $\delta t=t_n-t_{n-1}$ for $n=1,\ldots,N$ be a constant time step. The discretization of problem \eqref{eq:pb} is staggered in the following sense. The degrees of freedom for the density and the pressure are associated with the primal mesh $\mesh$ while the degrees of freedom of the velocity are associated with the dual mesh, or equivalently with the set of faces $\edges$. Correspondingly, the initial discrete density and velocity are defined by
\begin{equation}\label{eq:inicond}
\begin{array}{ll} 
\displaystyle \rho_K^0 = \frac 1 {|K|} \int_K \rho_0(\bfx) \dx, & \qquad K \in \mesh, \\[4ex]
\displaystyle \bfu_\edge^0 = \frac 1 {|D_\edge|} \int_{D_\edge} \bfu_0(\bfx) \dx, & \qquad \edge \in \edgesint,
\end{array}
\end{equation}
and the Dirichlet boundary condition is taken into account by setting $\bfu_\edge^n=0$ for all $\edge \in \edgesext$ and all $n$ in $\lbrace0,1,..,N \rbrace$.

\medskip
The time advancement is defined by induction as follows. 

\medskip
\noindent \emph{For $1 \leq n \leq N$, let us suppose that $(\rho_K^{n-1})_{K\in\mesh} \subset \xR$,  $(\bfu_\edge^{n-1})_{\edge\in\edgesint}\subset\xR^d$ and $(p_K^{n-1})_{K\in\mesh} \subset \xR$ are known families of real numbers, and find $(\rho_K^n)_{K\in\mesh} \subset \xR$,  $(\bfu_\edge^n)_{\edge\in\edgesint}\subset\xR^d$ and $(p_K^n)_{K\in\mesh} \subset \xR$ such that
\[
\sum_{K \in \mesh} |K|\, p_K^n=0,
\]
and}
\begin{subequations} \label{eq:impl_scheme} 
\begin{align} 
& \displaystyle \dfrac 1 {\delta t}(\rho^n_K-\rho_K^{n-1}) + \frac 1 {|K|}\sum_{\edge \in\edges(K)} \fluxK^n=0, 
& K \in \mesh, \label{eq:sch_mass} \\
&
 \displaystyle \dfrac 1 {\delta t}(\rho^n_\Ds \bfu^n_\edge-\rho^{n-1}_\Ds \bfu_\edge^{n-1})
+ \frac 1 {|D_\edge|}\sum_{\edged \in\edgesd(D_\edge)} \fluxd^n \bfu^n_\edged
- (\lapi \bfu)_\edge^n
+ (\gradi p)_\edge^n =0, 
&  \edge \in \edgesint, \label{eq:sch_mom} \\ 
& (\dive \bfu)_K^n  =0,
& K \in \mesh. \label{eq:sch_div}
\end{align}
\end{subequations}

\medskip
Equation \eqref{eq:sch_mass} is obtained by discretization of the mass balance over the primal mesh, and $\fluxK^n$ stands for the mass flux across $\edge$ outward $K$, which, because of the Dirichlet boundary condition on the velocity, vanishes on external faces and is given on the internal faces by:
\[
\fluxK^n= |\edge| \ \rho^n_\edge\ \bfu_\edge^n\cdot \bfn_{K,\edge}, \qquad \edge=K|L\in\edgesint.
\]
The density at the face $\edge=K|L$ is approximated by the upwind technique:
\begin{equation}\label{eq:rho_upwind}
\rho^n_\edge=\left|
\begin{aligned}
& \rho^n_K \qquad \mbox{if } \bfu_\edge^n\cdot \bfn_{K,\edge} \geq 0, \\
& \rho^n_L \qquad \mbox{otherwise}.
\end{aligned} \right.
\end{equation}
The discretization of the discrete velocity divergence is built in a similar way:
\begin{equation}\label{eq:defdive}
(\dive \bfu)_K^n= \frac 1 {|K|} \sum_{\edge \in\edges(K)} |\edge| \ \bfu_\edge^n\cdot \bfn_{K,\edge}, \qquad K \in \mesh.
\end{equation}

\medskip
Let us now turn to the discretization \eqref{eq:sch_mom} of the momentum balance equation \eqref{eq:mom}.
The first two terms correspond to a finite volume approximation of the convection operator, the description of which is given below (see Section \ref{subsec:vel_conv}).
The space discretization of the diffusion term in the momentum equation relies on the \emph{parametric} Rannacher-Turek 
(or rotated bilinear) element associated with the primal mesh $\mesh$ (see \cite{ran-92-sim}).
The reference element $\widehat K$ for the rotated bilinear element is the unit $d$-cube and the discrete functional space on $\widehat K$ is $\tilde{Q}_1(\widehat K)$:
\[
\tilde{Q}_1(\widehat K)= {\rm span}\, \bigl\lbrace 1,\,(\bfx_{i})_{i=1,\ldots,d},\,(\bfx_{i}^2-\bfx_{i+1}^2)_{i=1,\ldots,d-1}\bigr\rbrace.
\]
The mapping from the reference element to the actual discretization cell is the standard $Q_1$ mapping, and the space of discrete functions over a cell $K$, let us say $\tilde Q_1(K)$, is obtained from $\tilde{Q}_1(\widehat K)$ by composition.
The set of shape functions over $K$ is the set $\lbrace \zeta_\edge,~\edge \in \edges(K) \rbrace$, such that $\zeta_\edge|_K$ belongs to $\tilde Q_1(K)$ and
\begin{equation} \label{eq:def_zeta}
\int_\edge \zeta_{\edge'}|_K(\bfx) \dedge(\bfx) = \delta_\edge^{\edge'}\ |\edge|,
\qquad \edge, \edge' \in \edges(K),
\end{equation}
with $\delta_\edge^{\edge'}=1$ if $\edge=\edge'$ and $\delta_\edge^{\edge'}=0$ otherwise.
The continuity of the average value of a discrete function $\tilde v$ across each face of the mesh is required, which is consistent with a location of the degrees of freedom at the center of the faces:
\begin{equation} \label{eq:def_saut_EF}
\int_\edge [\tilde v]_\edge(\bfx) \dedge(\bfx) =0,
\qquad [\tilde v]_\edge(\bfx) = \lim_{\substack{\bfy \to \bfx \\ \bfy\in L}} \tilde v(\bfy)-
\lim_{\substack{\bfy \to \bfx \\ \bfy\in K}} \tilde v(\bfy),
\quad \bfx \in \edge,
\quad \edge=K|L.
\end{equation}
The discretization of the diffusion term reads
\[
-(\lapi \bfu)_\edge^n  = \frac 1 {|D_\edge|} \sum_{K \in \mesh}\ \sum_{\edge'\in \edges(K)}
\bfu_{\edge'}^n \, \int_K \gradi \zeta_{\edge'} \cdot \gradi \zeta_\edge \dx.
\]

\medskip
Finally, the discretization of the discrete pressure gradient term reads as follows:
\begin{equation}
\label{eq:defgradp} 
(\gradi p)^n_\edge= \frac{|\edge|}{|D_\edge|}\ (p^n_L-p^n_K)\ \bfn_{K,\edge}, \qquad \edge=K|L \in \edgesint.
\end{equation}

%
%
\subsection{The velocity convection operator}\label{subsec:vel_conv}

In this section, we describe the approximation of the convection operator $\partial_t(\rho \bfu) + \dive(\rho \bfu \otimes \bfu)$ which appears in the momentum balance equation.
As mentioned in the introduction, this discrete operator has already been used as a building brick for various schemes: variable density low Mach number flows \cite{ans-11-anl} (as here), barotropic and non-barotropic \cite{gal-08-unc, her-14-imp} compressible flows, drift-flux two-phase flow model \cite{gas-10-unc, gas-11-dis, her-13-pre}.
It is of finite volume type, and takes the general form given by the first two terms of \eqref{eq:sch_mom}.
The quantity $\rho_\Ds$ is an approximation of the density on the dual cell $D_\edge$, while $\fluxd$ is the mass flux across the edge $\edged$ of the dual cell $D_\edge$. These quantities are built so that a finite volume discretization of the mass balance \eqref{eq:sch_mass} holds over the internal dual cells:
\begin{equation}\label{eq:mass_D}
\frac{|D_\edge|}{\delta t} \ (\rho^n_\Ds-\rho^{n-1}_\Ds)
+ \sum_{\edged\in\edgesd(D_\edge)} \fluxd^n=0, \qquad  \edge\in\edgesint.
\end{equation}
This is crucial in order to reproduce, at the discrete level, the derivation of a kinetic energy balance equation (see Section \ref{sec:prop} below), a consequence of which are discrete analogues of the usual $\xL^\infty(\xL^2)$- and $\xL^2(\xH^1)$- stability estimates for the velocity.

\medskip
Let us first begin with the time derivative term.
The values $\rho^n_\Ds$ and $ \rho^{n-1}_\Ds$ are approximations of the density on the dual cell $D_\edge$ at time $t_n$ and $t_{n-1}$ respectively.  For $\edge$ in $\edgesint$ such that $\edge=K|L$, the approximate densities on the dual cell $D_\edge$ are given by the following weighted average:
\begin{equation} \label{eq:pd2}
|D_\edge|\ \rho^k_\Ds= \xi_K^\edge|K|\ \rho^k_K + \xi_L^\edge|L|\ \rho^k_L, \qquad \mbox{ for } k=n-1 \mbox{ and } k=n,
\end{equation}
where
\begin{equation} \label{eq:xiksigma}
\xi_K^\edge = \frac{|D_{K,\edge}|}{|K|}, \qquad K\in\mesh, ~\edge\in\edges(K).
\end{equation}

\medskip
The set of dual fluxes $\fluxd^n$ with $\edged$ included in the primal cell $K$, is computed by solving a linear system depending on the primal fluxes $(\fluxK^n)_{\edge\in\edges(K)}$, appearing in the discrete mass balance \eqref{eq:sch_mass}.
More precisely, we have the following definition for the dual fluxes, in which we omit for short the time dependence on $n$.

\begin{defi}[Definition of the dual fluxes from the primal ones]\label{def:dualfluxes}
The fluxes through the faces of the dual mesh are defined so as to satisfy the following three constraints:
\begin{itemize}
\item[(H1)] The discrete mass balance over the half-diamond cells is satisfied, in the following sense. For all primal cell $K$ in $\mesh$, the set $(\fluxd)_{\edged\subset K}$ of dual fluxes included in $K$ solves the following linear system
\begin{equation}\label{eq:F_syst}
F_{K,\edge} + \sum_{\edged \in \edgesd(D_\sigma),\ \edged \subset K} F_{\edge,\edged}= \xi_K^\edge \ \sum_{\edge' \in \edges(K)} F_{K,\edge'}, \qquad \edge \in \edges(K).
\end{equation}
\item[(H2)] The dual fluxes are conservative, \ie\ for any dual face $\edged=D_\edge|D_\edge'$, we have $F_{\edge,\edged}=-F_{\edge',\edged}$.
\item[(H3)] The dual fluxes are bounded with respect to the primal fluxes $(F_{K,\edge})_{\edge \in \edges(K)}$, in the sense that there exists a universal constant real number $C$ such that:
\begin{equation}\label{eq:F_bounded}
|F_{\edge,\epsilon}| \leq C \ \max \,\left \lbrace |F_{K,\edge}|,\ \edge \in \edges(K) \right \rbrace, \quad K \in \mesh,\  \edge \in \edges(K),\  \epsilon \in \edgesd(D_\sigma),\ \edged\subset K .
\end{equation}
\end{itemize}
\end{defi}

In fact, the definition \ref{def:dualfluxes} is not complete, since the system of equations \eqref{eq:F_syst} has an infinite number of solutions, which makes necessary to impose in addition the constraint \eqref{eq:F_bounded}; however, assumptions (H1)-(H3) are sufficient for the subsequent developments of this paper (and thus, in particular, imply the consistency of the discrete convection operator).
Note that, since \eqref{eq:F_syst} is linear with respect to the $F_{\edge,\edged},\ \edge \in \edges(K), \edged \in \edgesd(D_{\edge}),\ \edged \subset K$, a solution of \eqref{eq:F_syst} may be expressed as:
\[
F_{\edge,\edged} = \sum_{\edge' \in \edges(K)} (\alpha_K)_\edge^{\edge'} F_{K,\edge'},\qquad \edge \in \edges(K),\
\edged \in \edgesd(D_{\edge}) \mbox{ and } \edged \subset K,
\]
and the constraint \eqref{eq:F_bounded} amounts to requiring to the coefficients $((\alpha_K)_\edge^{\edge'})_{\edge,\edge'\in \edges(K)}$ to be bounded by a universal constant. In practice, one has $|(\alpha_K)_\edge^{\edge'}|\leq 1$ for all $\edge,\edge'\in \edges(K)$ and all $K\in\mesh$ (see \cite{BLPS-fvca7}).

\medskip
We thus would be able to cope with a quite general definition of the diamond cells, since, up to now, even their volume is not fixed.
In practice, as said in Definition \ref{def:disc}, we however choose to impose that $|D_{K,\edge}|=|K|/(2d)$; in other words, the real number $\xi_K^\edge$ in \eqref{eq:xiksigma} is given by $\xi_K^\edge = 1/(2d)$ for all $K \in \mesh$ and $\edge \in \edges(K)$.
In these conditions, System \eqref{eq:F_syst} is now completely independent from the cell $K$ under consideration.
We may thus consider a particular geometry for $K$, let us say $K=(0,1)^d$, and find an expression for the coefficients $((\alpha_K)_\edge^{\edge'})_{\edge,\edge'\in \edges(K)}$ which we will apply to all the cells, thus automatically satisfying the constraint \eqref{eq:F_bounded}.
A technique for this computation is described in \cite[Section 3.2]{ans-11-anl}.
The idea is to build a momentum field $\bfw$ with constant divergence, such that:
\[
\int_\edge \bfw \cdot \bfn_{K,\edge} \dedge(\bfx) = F_{K,\edge}, \qquad \forall \edge \in \edges(K).
\]
Then an easy computation shows that the definition
\[
F_{\edge,\edged} = \int_\edged \bfw \cdot \bfn_{\edge,\edged} \dedge(\bfx),
\]
where the unit vector normal to $\edged $ outward $D_\edge$ is denoted by $\bfn_{\edge,\edged}$, satisfies \eqref{eq:F_syst} (see \cite[Lemma 3.2]{ans-11-anl}).
The set of coefficients $((\alpha_K)_\edge^{\edge'})_{\edge,\edge'\in \edges(K)}$ obtained for a quadrangle is given in \cite[Section 3.2]{ans-11-anl}; extension to the three-dimensional case is straightforward.

\bigskip
To complete the definition of the convective flux, we just have now to give the expression of the velocity at the dual face, \ie\ of the quantity $\bfu_\edged^n$ in \eqref{eq:sch_mom}.
As already said, a dual face lying on the boundary is also a primal face, and the flux across that face is zero.
Therefore, the values $\bfu_\edged^n$ are only needed at the internal dual faces; we choose them to be centered:
\[
\bfu_\edged^n = \frac 1 2 ( \bfu_\edge^n + \bfu_{\edge'}^n), \qquad \mbox{for } \edged=D_\edge|D_\edge'.
\]
%
%
\section{The convergence theorem} \label{sec:conv}

We begin by associating functions with the discrete unknowns of the scheme.
To this purpose, we first define the following sets of discrete functions of the space variable.

\begin{defi}[Discrete spaces]\label{def:disc_space}
Let $\disc=(\mesh,\edges)$ be a staggered discretization of $\Omega$ in the sense of Definition \ref{def:disc}.
We denote by $\xH_\mesh(\Omega)\subset \xL^\infty(\Omega)$ the space of functions which are piecewise constant on each primal mesh cell $K\in\mesh$.
For all $w\in \xH_\mesh(\Omega)$ and for all $K\in\mesh$, we denote by $w_K$ the constant value of $w$ in $K$, so the function $w$ reads:
\[
w(\bfx)= \sum_{K \in \mesh} w_K\, \mathcal{X}_K(\bfx) \qquad \mbox{for a.e. } \bfx \in \Omega,
\]
where $\mathcal{X}_K$ stands for the characteristic function of $K$.\\[0.5ex]
Similarly, we denote by $\xH_\edges(\Omega)\subset \xL^\infty(\Omega)$ the space of functions which are piecewise constant on each diamond cell of the dual mesh $D_\edge,~\edge\in\edges$.
For all $u \in \xH_\edges(\Omega)$ and for all $\edge\in\edges$, we denote by $u_\edge$ the constant value of $u$ in $D_\edge$, so the function $u$ reads:
\[
u(\bfx)= \sum_{\edge\in\edges} u_\edge\, \mathcal{X}_{D_\edge}(\bfx) \qquad \mbox{for a.e. } \bfx \in \Omega,
\]
where $\mathcal{X}_{D_\edge}(\bfx)$ stands for the characteristic function of $D_\edge$.
Finally we denote $\xH_{\edges,0}(\Omega)=\bigl\lbrace u\in \xH_\edges(\Omega),\ u_\edge=0 \text{ for all } \edge \in \edgesext \bigr\rbrace$. 
\end{defi}

Then, with the discrete unknowns computed by induction through the scheme, we associate piecewise constant functions on each time interval $(t_{n-1},t_n]$ as follows:
\[
\rho(\bfx,t) = \rho^n(\bfx),\quad p(\bfx,t)=p^n(\bfx),\quad \bfu(\bfx,t)=\bfu^n(\bfx),\quad \mbox{for a.e. } t \in (t_{n-1},t_n],
\]
where $\rho^n \in\xH_\mesh(\Omega)$, $p^n \in\xH_\mesh(\Omega)$ and $\bfu \in \xH_{\edges,0}(\Omega)^d$ are the discrete functions defined by $(\rho_K^n)_{K \in \mesh}$, $(p_K^n)_{K \in \mesh}$ and $(\bfu_\edge^n)_{\edge \in \edges}$ respectively.
Definition \ref{def:disc_space} thus yields:
\begin{equation}\label{eq:def_rho_u_e_p}
\begin{array}{ll} \displaystyle
\rho(\bfx,t) = \sum_{n=1}^N \sum_{K \in \mesh} \rho_K^n\, \mathcal{X}_K(\bfx)\, \mathcal{X}_{(n-1,n]}(t),
\\ \displaystyle
p(\bfx,t) =\sum_{n=1}^N \sum_{K\in\mesh} p_K^n\,\mathcal{X}_K(\bfx)\, \mathcal{X}_{(n-1,n]}(t),
\\[3ex] \displaystyle
\bfu(\bfx,t)=\sum_{n=1}^N \sum_{\edge\in\edges} \bfu_\edge^n\, \mathcal{X}_{D_\edge}(\bfx)\, \mathcal{X}_{(n-1,n]}(t),
\end{array}
\end{equation}
where $\mathcal{X}_{(n-1,n]}$ stands for the characteristic function of the time interval $(t_{n-1},t_n]$.

\medskip
We then define a regular sequence of discretizations as follows.
\begin{defi}[Regular sequence of discretizations]\label{def:reg_disc}
Let $(\disc\m,\delta t\m)_{m\in\xN}$ be a sequence of staggered discretizations (in the sense of Definition \ref{def:disc}) and time steps.
For $m \in \xN$, let $h\m$ and $\theta\m$ be the space step and the regularity parameter associated with $\disc\m$  by equations \eqref{eq:def_h} and \eqref{eq:reg} respectively, and let $\alpha\m$ be the measure of the deviation of the geometry of the cells from parallelograms, as defined by \eqref{eq:def_alpha}.
Then this sequence $(\disc\m,\delta t\m)_{m\in\xN}$ is said regular if:
\begin{itemize}
\item[$(i)$] for all $m \in \xN,\ \theta\m \leq \theta_0$ for some positive real number $\theta_0$,
\item[$(ii)$] the sequences of space steps $(h\m)_{m \in \xN}$ and time steps $(\delta t\m)_{m \in \xN}$ tend to zero when $m$ tends to $+ \infty$,
\item[$(iii)$] the sequence of parameters $(\alpha\m)_{m \in \xN}$ tends to zero when $m$ tends to $+ \infty$.
\end{itemize}
\end{defi}
\begin{rmrk}[A particular construction of regular sequence of discretizations]
For $d=2$, a sequence of discretizations satisfying the assumption $(iii)$ above is obtained by successively dividing each quadrangle in four sub-quadrangles, splitting it along the lines joining the mid-points of opposite faces.
The extension of this construction to the three-dimensional case (splitting now each hexahedron in 8 subvolumes) is not as easy as it seems, the difficulty being to keep the faces plane.
\end{rmrk}

The following theorem is the main result of this paper; its proof is the purpose of the rest of the paper.
\begin{thrm}
\label{thrm:convergence}
Let $(\disc\m,\delta t\m)_{m\in\xN}$ be a regular sequence of staggered discretizations and time steps.
Then, under assumptions \eqref{eq:H_rho} and \eqref{eq:H_u} for the initial data, for $m \in \xN$, there exists a discrete solution to the scheme \eqref{eq:impl_scheme}.
Let us denote by $(\rho\m,\bfu\m,p\m)_{m\in \xN}$ the corresponding discrete functions as defined in \eqref{eq:def_rho_u_e_p}.
Then, there exists $(\bar \rho, \bar \bfu)$ weak solution of problem \eqref{eq:pb} in the sense of Definition \ref{def:weaksol}, such that, up to a subsequence, $\rho\m$ strongly  converges to $\bar \rho$ in $\xL^q((0,T)\times \Omega)$ for all $q$ in $[1,\infty)$ and $\bfu \m$ strongly converges to $\bar \bfu$ in $\xL^q((0,T);\xL^2(\Omega)^d)$ for all $q$ in $[1,\infty)$.
\end{thrm}
%
%
\section{Preliminary lemmata} \label{sec:math}

We gather in this section some preliminary mathematical results which are useful for the analysis of the scheme.

%
%
\subsection{Properties of the discrete divergence and gradient operators}

We define the discrete divergence and gradient operators in the following way:
\begin{eqnarray} \label{eq:divdisc}
& \dive_\mesh: 
&\left \lbrace 
\begin{array}{ccll}
\xH_{\edges,0}(\Omega)^d & \longrightarrow & \xH_\mesh(\Omega) & \\
\bfu & \longmapsto & \dive_{\mesh} \bfu(\bfx) =  (\dive \bfu)_K, &\forall \bfx \in K,~K\in\mesh, 
\end{array}
\right.
\\[2ex] \label{eq:graddisc}
& \gradi_\edges : 
& \left \lbrace 
\begin{array}{ccll}
\xH_\mesh(\Omega)  & \longrightarrow & \xH_{\edges,0}(\Omega)^d \\
p & \longmapsto  & \gradi_\edges p(\bfx) = (\gradi p)_\edge, & \forall \bfx \in D_\edge, ~ \edge \in\edgesint,
\end{array} 
\right.
\end{eqnarray}
where $(\dive \bfu)_K$ and $(\gradi p)_\edge$ are defined in \eqref{eq:defdive} and \eqref{eq:defgradp} respectively.
The following lemma gives two first important properties of these operators.

\begin{lem} \label{lmm:dual_divgrad} 
Let $\bfv\in\xH_0^1(\Omega)^d$.
For a given discretization $\disc=(\mesh,\edges)$, for $\edge\in\edges$, let 
\[
\bfv_\edge= \frac 1 {|\edge|}\ \int_\edge \bfv(\bfx)\dedge(\bfx),
\]and let $P_{\edges}\bfv$ be the function of $\xH_{\edges,0}(\Omega)^d$ defined by $P_{\edges}\bfv(\bfx)=\bfv_\edge$ for every $\bfx$ in $D_\edge$ and every $\edge \in \edges$.
Then for all $p$ in $\xH_{\mesh}(\Omega)$,
\begin{equation} \label{eq:div_stab}
\int_\Omega p(\bfx)\, \dive_\mesh (P_\edges\bfv)(\bfx)  \dx = \int_\Omega p(\bfx)\, \dive \bfv(\bfx) \dx.
\end{equation}
In addition, the discrete divergence and the discrete gradient are dual in the following sense.
For any $\bfv$ in $\xH_{\edges,0}(\Omega)^d$ and any $p$ in $\xH_\mesh(\Omega)$, one has
\begin{equation} \label{eq:dual_divgrad1} 
\int_\Omega p(\bfx)\, \dive_\mesh\bfv(\bfx)  \dx + \int_\Omega \bfv(\bfx) \cdot \gradi_\edges p(\bfx) \dx =0.
\end{equation}
\end{lem}

\begin{proof}
The first relation is an obvious consequence of relation \eqref{eq:defdive} defining the discrete divergence operator and of the fact that $p$ is piecewise constant on the cells $K \in \mesh$. For the second relation, the same observations yield that for any pair $(\bfv,p)$ in $\xH_{\edges,0}(\Omega)^d\times \xH_\mesh(\Omega)$, one has
\begin{equation} \label{eq:dual_divgrad2}
\sum_{K \in \mesh} |K| \ p_K \ ( \dive \bfv)_K + \sum_{\edge \in \edges} |D_\edge|\ \bfv_\edge \cdot (\gradi p)_\edge=0. 
\end{equation}
Note that, because of the homogeneous Dirichlet boundary condition on $\bfv$, the discrete gradient does not need to be defined at the external faces and the second sum in \eqref{eq:dual_divgrad2} is actually a sum over $\edgesint$.
\end{proof}

We also have the following convergence property for the discrete gradient defined in \eqref{eq:graddisc}.

\begin{lem}[Weak convergence of the discrete gradient] \label{lmm:compact_fv}
Let $(\disc\m)_{m\in\xN}$ be a sequence of regular discretizations of $\Omega$ in the sense of Definition \ref{def:reg_disc}.
For $m\in\xN$, let $p\m \in \xH_{\mesh\m}(\Omega)$ and assume that there exists $C$ in $\xR_+$ such that, for all $m \in \xN$, $\norm{\gradi_{\edges \m}p\m}_{\xL^q(\Omega)^d} \leq C$ for some $q$ in $[1, \infty]$.
Assume also that there exists $\bar p$ in $\xW^{1,q}(\Omega)$ such that $p\m$ converges to $\bar p$ in the distribution sense as $m$ tends to $+\infty$, \ie:
\[
\forall \bfphi\in \xC_c^\infty(\Omega), \quad
\lim_{m \to +\infty} \int_\Omega p\m(\bfx) \cdot \bfphi(\bfx) \dx = \int_\Omega \bar p(\bfx) \cdot \bfphi(\bfx) \dx.
\]
Then $\gradi_{\edges\m}p\m$ converges to $\gradi \bar p$ in the distribution sense:
\[
\forall \bfphi\in \xC_c^\infty(\Omega)^d, \quad
\lim_{m \to +\infty} \int_\Omega \gradi_{\edges\m}p\m(\bfx) \cdot \bfphi(\bfx) \dx = \int_\Omega \gradi \bar p(\bfx) \cdot \bfphi(\bfx) \dx.
\]
In addition, for $q\in (1,\infty)$ (resp. $q=+\infty$), if $p\m$ weakly (resp. weakly-$\star$) converges to $\bar p$ in $\xL^q(\Omega)$ (resp. in $\xL^\infty(\Omega)$), then $\gradi_{\edges\m}p\m$ also converges to $\gradi \bar p$ weakly (resp. weakly-$\star$) in $\xL^q(\Omega)^d$ (resp. in $\xL^\infty(\Omega)^d$).
\end{lem}

\begin{proof}
Let $\bfphi\in \xC_c^\infty(\Omega)^d$.
For a given discretization $\disc=(\mesh,\edges)$, for $\edge\in\edges$, let
\[
\bfphi_\edge=\frac 1 {|\edge|}\ \int_\edge \bfphi(\bfx) \dedge(\bfx),
\]
and let $P_{\edges}\bfphi$ be the function of $\xH_{\edges,0}(\Omega)^d$ defined by $P_{\edges}\bfphi(\bfx)=\bfphi_\edge$ if $\bfx\in D_\edge$, for every $\edge \in \edges$.
Let $q'=+\infty$ if $q=1$, $q'$ be given by $1/q+1/q'=1$ if $q \in (1,\infty)$, and $q'=1$ if $q=+\infty$.
With the assumptions of the lemma, an easy calculation shows that $\norm{P_{\edges\m}\bfphi-\bfphi}_{\xL^{q'}(\Omega)^d}\leq 2\, |\Omega|^{1/q'}\,\norm{\gradi \bfphi}_{\xL^\infty(\Omega)^{d\times d}}\ h\m$ (with $|\Omega|^{1/q'}=1$ for $q'=+\infty$).
We may write
\[
\int_\Omega \gradi_{\edges\m}p\m(\bfx) \cdot \bfphi(\bfx) \dx =
\int_\Omega \gradi_{\edges\m}p\m(\bfx) \cdot (P_{\edges\m}\bfphi)(\bfx) \dx + R,
\]
with $|R|\leq \norm{\gradi_{\edges \m}p\m}_{\xL^q(\Omega)^d} \norm{P_{\edges\m}\bfphi-\bfphi}_{\xL^{q'}(\Omega)^d} \leq 2 C\ |\Omega|^{1/q'}\norm{\gradi \bfphi}_{\xL^\infty(\Omega)^{d\times d}}\ h\m$ which tends to zero as $m\to+\infty$.
Now invoking successively  \eqref{eq:dual_divgrad1}, \eqref{eq:div_stab} and the convergence of $p\m$ to $\bar p$ in the distribution sense, we get
\begin{multline*}
\int_\Omega \gradi_{\edges\m}p\m(\bfx) \cdot (P_{\edges\m}\bfphi)(\bfx) \dx =
-\int_\Omega p\m(\bfx)\, \dive \bfphi(\bfx) \dx
\\ 
\to -\int_\Omega \bar p(\bfx)\, \dive \bfphi(\bfx) \dx
= \int_\Omega \gradi \bar p(\bfx) \cdot \bfphi(\bfx) \dx, \qquad \text{as $m \to +\infty$,}
\end{multline*}
which shows that $\gradi_{\edges\m}p\m$ tends to $\gradi \bar p$ in the distributional sense.
The weak or weak-$\star$ convergence of $\gradi_{\edges\m}p\m$, for $q\in(1,\infty)$ or $q=+\infty$ respectively, follows by density.
\end{proof}
%
%
\subsection{Properties of the Rannacher-Turek element}

We gather in this section consistency and stability results for the Rannacher-Turek element, most of them given in \cite{ran-92-sim}, which are used in the analysis of the scheme.

\medskip
With every function $\bfu$ in $\xH_\edges(\Omega)^d$ (or, equivalently, with every set of degrees of freedom $(\bfu_\edge)_{\edge \in \edges}$), one classically associates, in the finite element context, the function $\tilde \bfu$ from $\Omega$ to $\xR^d$:
\[
\tilde \bfu(\bfx) = \sum_{\edge\in\edges} \bfu_\edge \zeta_\edge (\bfx), \qquad \mbox{for a.e. } \bfx \in \Omega,
\]
with $\zeta_\edge$ the shape function defined by \eqref{eq:def_zeta}.
This identification allows to introduce the broken Sobolev $\xH^1$ semi-norm $\norm{.}_\brok$, given for any $\bfu \in \xH_{\edges}(\Omega)^d$ by:
\[
\norm{\bfu}_\brok^2=\sum_{K\in \mesh} \int_K \gradi \tilde{\bfu}(\bfx):\gradi \tilde{\bfu}(\bfx)\dx.
\]
The semi-norm $\norm{\bfu}_\brok$ is in fact a norm on the space $\xH_{\edges,0}(\Omega)^d$, thanks to the discrete Poincar\'e inequality (see \cite{ran-92-sim}) stated in the following lemma.

\begin{lem}[Discrete Poincar\'e inequality] \label{lmm:poincare_brok}
Let $\disc=(\mesh,\edges)$ be a staggered discretization of $\Omega$ in the sense of Definition \ref{def:disc}, such that $\theta_\disc \leq \theta_0$, with $\theta_\disc$ defined by \eqref{eq:reg}.
Then there exists $C$, only depending on $d$, $\Omega$ and $\theta_0$ such that
\[
\norm{\bfu}_{\xL^2(\Omega)^d} \leq C\, \norm{\bfu}_\brok, \quad \forall \bfu \in \xH_{\edges,0}(\Omega)^d.
\]
\end{lem}

Let us now denote by $r_\edges$ the following natural interpolation operator from $\xH^1_0(\Omega)^d$ to $\xH_{\edges,0}(\Omega)^d$:
\begin{equation}\label{eq:rh}
\begin{array}{l|lcl}
r_\edges : \qquad
&
\xH^1_0(\Omega)^d  & \longrightarrow & \xH_{\edges,0}(\Omega)^d
\\ &
\bar \bfu & \mapsto & \displaystyle 
r_\edges \bar \bfu(\bfx) = \sum_{\edge \in \edges} |\edge|^{-1} \left(\int_\edge \bar \bfu(\bfx) \dedge(\bfx) \right)
\,\mathcal{X}_{D_\edge}(\bfx),
\end{array}
\end{equation}
where $\mathcal{X}_{D_\edge}(\bfx)$ is the characteristic function of the dual cell $D_\edge$.
We have the following stability and approximation properties of $r_\edges$.

\begin{lem} \label{lmm:RT}
Let $\disc=(\mesh,\edges)$ be a staggered discretization of $\Omega$ in the sense of Definition \ref{def:disc}, such that $\theta_\disc \leq \theta_0$, with $\theta_\disc$ defined by \eqref{eq:reg}.
The interpolation operator $r_\edges$ enjoys the following properties:
\begin{itemize}
\item[--] Stability:
\[
\forall \bar \bfu \in \xH^1_0(\Omega)^d,\quad
\norm{r_\edges \bar \bfu}_\brok \leq C\ \norm{\bar \bfu}_{\xH^1(\Omega)^d},
\]
with $C$ only depending on $\Omega$ and $\theta_0$.
\medskip
\item[--] Approximation properties:
\[
 \begin{aligned} 
&\forall \bar \bfu \in \xH^2(\Omega)^d \cap \xH^1_0(\Omega)^d,\ \forall K \in \mesh, 
\\[1ex]
&\norm{\bar \bfu-\widetilde{r_\edges \bar \bfu}}_{\xL^2(K)^d} + h_K\, \norm{\gradi (\bar \bfu-\widetilde{r_\edges \bar \bfu})}_{\xL^2(K)^{d\times d}} \leq
C\, h_K\,(h_K + \alpha_K)\, |\bar \bfu|_{\xH^2(K)^d},
\end{aligned}
\]
with $C$ only depending on $\Omega$ and $\theta_0$.
\end{itemize}
\end{lem}

We also have the following classical stability property, which is used when passing to the limit in the momentum equation (see Section \ref{subsec:convergence_u}).

\begin{lem} \label{lmm:controle_sauts_EF}
Let $\disc=(\mesh,\edges)$ be a staggered discretization of $\Omega$ in the sense of Definition \ref{def:disc}, such that $\theta_\disc \leq \theta_0$, with $\theta_\disc$ defined by \eqref{eq:reg}.
For $\edge$ in $\edgesint$, let $[\tilde \bfu]_\edge$ be the jump of $\tilde \bfu$ across $\edge$ as defined in \eqref{eq:def_saut_EF}, and for $\edge$ in $\edgesext\cap\edges(K)$, let $\displaystyle [\tilde \bfu]_\edge(\bfx)=\lim_{\substack{\bfy\to\bfx \\ \bfy \in K}}\tilde \bfu(\bfx)$.\\
Then there exists $C$, only depending on $d$, $\Omega$ and $\theta_0$ such that, for all $\bfu$ in $\xH_{\edges}(\Omega)^d$,
\begin{equation} \label{eq:controle_sauts_EF}
\Big ( \sum_{\edge \in \edges}\frac 1 {h_\edge} \int_\edge [\tilde \bfu]_\edge^2(\bfx) \dedge(\bfx) \Big )^{\frac 12}
\leq C\, \norm{\bfu}_\brok,
\end{equation}
where, for all $\edge$ in $\edges$, $h_\edge={\rm diam}(\edge)$.
\end{lem}

Finally, the following lemma states that the pair of approximation spaces $\xH_{\edges,0}(\Omega)^d$ for the velocity (endowed with the finite element broken norm) and $\xH_\mesh(\Omega)$ for the pressure is \textit{inf-sup} stable.

\begin{lem} \label{lmm:inf-sup}
Let $\disc=(\mesh,\edges)$ be a staggered discretization of $\Omega$ in the sense of Definition \ref{def:disc}, such that $\theta_\disc \leq \theta_0$, with $\theta_\disc$ defined by \eqref{eq:reg}. 
Then there exists $\beta>0$, depending only on $\Omega$ and $\theta_0$, such that for all $p$ in $ \xH_{\mesh}(\Omega)$, there exists $\bfu$ in $\xH_{\edges,0}(\Omega)^d$ satisfying
\[
\norm{\bfu}_\brok=1 \text{ and } \int_\Omega p(\bfx)\,\dive_{\mesh} \bfu (\bfx) \dx \geq \beta\, \norm{p-m(p)}_{L^2(\Omega)},
\]
where $m(p)$ stands for the mean value of $p$ over $\Omega$.
\end{lem}
%
%
\subsection{Discrete functional analysis}

We introduce the following finite volume discrete $\xH^1$-norm on the space $\xH_{\edges,0}(\Omega)^d$:
\[
\norm{\bfu}_\fv^2 = \sum_{K \in \mesh} h_K ^{d-2} \sum_{\edge,\edge' \in \edges(K)}\ |\bfu_\edge-\bfu_{\edge'}|^2,
\]
which, by an easy computation, may be shown to be equivalent, over a regular sequence of discretizations, to the usual finite volume $\xH^1$-norm defined in \cite{eym-00-book}.
The following lemma is obtained by using standard properties of the $Q_1$ mapping, and then invoking a norm equivalence argument for the finite dimensional space of discrete functions on the reference element.
It allows to inherit all the discrete functional analysis results associated with the finite volume $\xH^1$-norm.

\begin{lem} \label{lmm:H1ns}
Let $\disc=(\mesh,\edges)$ be a staggered discretization of $\Omega$ in the sense of Definition \ref{def:disc}, such that $\theta_\disc \leq \theta_0$, with $\theta_\disc$ defined by \eqref{eq:reg}.
Then :
\[
\norm{\bfu}_\fv \leq C\ \norm{\bfu}_\brok, \quad \forall \bfu \in \xH_{\edges,0}(\Omega)^d,
\]
where $C$ only depends on $\Omega$, $d$ and $\theta_0$.
\end{lem}

We begin by giving a crucial discrete Sobolev embedding property, which is a consequence of Lemma \ref{lmm:H1ns} and of the results stated in \cite["{\em Discrete functional analysis}" appendix]{eym-10-sushi}.

\begin{lem}[Discrete Sobolev embedding] \label{lmm:injsobolev_br}
Let $\disc=(\mesh,\edges)$ be a staggered discretization of $\Omega$ in the sense of Definition \ref{def:disc}, such that $\theta_\disc \leq \theta_0$, with $\theta_\disc$ defined by \eqref{eq:reg}.
Then there exists $C(q,d,\theta_0)>0$ such that
\[
\norm{\bfu}_{\xL^q(\Omega)^d} \leq C(q,d,\theta_0)\, \norm{\bfu}_\fv, \quad \forall \bfu \in \xH_{\edges,0}(\Omega)^d,
\]
for all $q\in [1,2^*]$ (with $2^*=6$) if $d=3$ and for all $q\in[1,\infty)$ if $d=2$.
\end{lem}

\begin{rmrk}
Actually, the important regularity assumption for proving Lemma \ref{lmm:injsobolev_br} is the boundedness of the parameters $\theta_{\edges,1}$ and $\theta_{\edges,2}$ defined in \eqref{eq:theta_edges1}-\eqref{eq:theta_edges2}. Note that for d=2, one has $C(q,d,\theta_0)\to\infty$ as $q\to \infty$.
\end{rmrk}

The following Lemma is a consequence of Lemma \ref{lmm:H1ns} and an adaptaion of the proof of \cite[Lemma 3.3]{eym-00-book}.
\begin{lem} \label{lmm:controle_sauts_br}
Let $\disc=(\mesh,\edges)$ be a staggered discretization of $\Omega$ such that $\theta_\disc \leq \theta_0$, with $\theta_\disc$ defined by \eqref{eq:reg}.
For any $\bfu\in \xH_{\edges,0}(\Omega)^d$, we define its extension $\bfu^\sharp$ to the whole space $\xR^d$ by setting $\bfu^\sharp=0$ on $\xR^d \setminus \Omega$. 
Then there exists $C$, only depending on $\Omega$, $d$ and $\theta_0$ such that
\begin{equation} \label{eq:controle_sauts_br}
\norm{\bfu^\sharp(.+\eta)-\bfu^\sharp}_{\xL^2(\xR^d)^d}^2 
\leq C \norm{\bfu}_\brok^2\, |\eta|\, (|\eta|+h_\disc), \quad \forall \eta \in \xR^d,~ \forall \bfu\in \xH_{\edges,0}(\Omega)^d.
\end{equation}
\end{lem}

Finally, an important consequence of Lemma \ref{lmm:controle_sauts_br} is the following compactness result, whose proof is similar to that of \cite[Theorem 3.10]{eym-00-book}.

\begin{lem}[Discrete Rellich theorem] \label{lmm:compact_br}
Let $(\disc\m)_{m\in\xN}$ be a regular sequence of discretizations in the sense of Definition \ref{def:reg_disc}.
For $m\in\xN$, let $\bfu\m \in \xH_{\edges\m,0}(\Omega)^d$ and assume that there exists $C \in \xR$ such that, for all $m\in\xN$, $\norm{\bfu\m}_{\edges\m,\rm b} \leq C$.
Then, there exists $\bar \bfu$ in $\xH_0^1(\Omega)^d$ and a subsequence of $(\bfu\m)_{m\in\xN}$ (not relabeled) such that $\bfu\m$ converges strongly in $\xL^2(\Omega)^d$ towards $\bar \bfu$ as $m\to\infty$.
\end{lem}
%
%
\subsection{Estimates on the dual mass convection term}
In the following, when confusion is possible, if the dual fluxes $\fluxd$ are computed from the fields $\rho\in \xH_\mesh(\Omega)$ and $\bfu\in \xH_{\edges,0}(\Omega)^d$ through $\fluxK$, we denote $\fluxd(\rho,\bfu)$ and $\fluxK(\rho,\bfu)$ for the sake of clarity. 
Let us define the mapping:
\begin{equation} \label{eq:mass_form_edges}
\mcal{Q}^{\rm mass}_\edges : \left \lbrace \begin{array}{cccl}
		  &  \hspace{-0.5cm} \xH_\mesh(\Omega) \times (\xH_{\edges,0}(\Omega)^d)^3 &\hspace{-0.3cm} \longrightarrow \hspace{-0.3cm} & \xR \\
		  &  (\rho,\bfu,\bfv,\bfw) & \hspace{-0.3cm} \longmapsto \hspace{-0.3cm} & \ds \mcal{Q}^{\rm mass}_\edges(\rho,\bfu,\bfv,\bfw) =
		     \sum_{\edge\in\edgesint}  (\bfv_\edge\cdot\bfw_\edge) \sum_{\edged \in\edgesd(D_\edge)} \fluxd(\rho,\bfu).
\end{array} \right.
\end{equation}
The mapping $\mcal{Q}^{\rm mass}_\edges$ is a discrete counterpart of the variational formulation of the convection term in the mass conservation equation  $\int_\Omega \dive(\rho \bfu) \phi$, for test functions $\phi$ of the form $\phi=\bfv\cdot\bfw$ where $\bfv,\, \bfw\in \xH^1_0(\Omega)^d$. This discrete variational formulation is obtained from the dual mass conservation equation \eqref{eq:mass_D} by multiplying  the local mass flux by the test function $\phi_\edge=(\bfv_\edge\cdot\bfw_\edge)$ and summing over $\edge\in\edges$ (we recall that $\fluxd(\rho,\bfu)$ vanishes at external faces).

The next lemma provides an estimate on the function $\mcal{Q}^{\rm mass}_\edges$, which is a discrete counterpart of a similar estimate satisfied in the continuous setting.
\begin{lem}[Estimate on $\mcal{Q}^{\rm mass}_\edges$] \label{lmm:estimates_Qmass}
Let $\disc=(\mesh,\edges)$ be a staggered discretization of $\Omega$ in the sense of Definition \ref{def:disc}, such that $\theta_\disc \leq \theta_0$.
Then, there exists a constant $C$, only depending on $\Omega$, $d$ and $\theta_0$, such that
\begin{equation} \label{eq:estimates_Qmass1}
|\mcal{Q}^{\rm mass}_\edges(\rho,\bfu,\bfv,\bfw)| \leq C \norm{\rho}_{\xL^{\infty}(\Omega)} \norm{\bfu}_{\xL^6(\Omega)^d} \left ( \norm{\bfv}_{\xL^3(\Omega)^d} \norm{\bfw}_{\brok} + \norm{\bfw}_{\xL^3(\Omega)^d} \norm{\bfv}_{\brok}\right ).
\end{equation}
for all $\rho$ in $\xH_\mesh(\Omega)$ and $\bfu$, $\bfv$, $\bfw$ in $\xH_{\edges,0}(\Omega)^d$.
\end{lem}

\begin{proof}
By definition of $\mcal{Q}^{\rm mass}_\edges(\rho,\bfu,\bfv,\bfw)$, we have, recalling that $ (\fluxd)_{\edged\subset K}$ solve \eqref{eq:F_syst} with $\xi_K^\edge=1/(2d)$:
\[
\begin{aligned}
 \mcal{Q}^{\rm mass}_\edges(\rho,\bfu,\bfv,\bfw) 
 & = \frac{1}{2d}\sum_{\substack{\edge\in\edgesint \\ \edge= K|L}}  (\bfv_\edge\cdot\bfw_\edge) \Big( \sum_{\edge' \in\edges(L)} F_{L,\edge'}(\rho,\bfu) +   \sum_{\edge' \in\edges(K)} F_{K,\edge'}(\rho,\bfu)  \Big ) \\
 & = \frac{1}{2d}\sum_{\substack{\edge\in\edgesint \\ \edge= K|L}}  (\bfv_\edge\cdot\bfw_\edge) \Big( \sum_{\edge' \in\edges(L)} |\edge'|\rho_{\edge'} \bfu_{\edge'} \cdot \bfn_{L,\edge'} +   \sum_{\edge' \in\edges(K)} |\edge'|\rho_{\edge'} \bfu_{\edge'} \cdot \bfn_{K,\edge'}  \Big ). 
\end{aligned}
\]
Reordering the sum, one gets:
\[
\begin{aligned}
 \mcal{Q}^{\rm mass}_\edges(\rho,\bfu,\bfv,\bfw) 
  & = -\frac{1}{2d}\sum_{\substack{\edge\in\edgesint \\ \edge= K|L}}  |\edge|\rho_{\edge} \bfu_{\edge} \cdot \bfn_{K,\edge} \Big( \sum_{\edge' \in\edges(L)} (\bfv_{\edge'}\cdot\bfw_{\edge'}) -   \sum_{\edge' \in\edges(K)} (\bfv_{\edge'}\cdot\bfw_{\edge'})  \Big ) \\
  & = -\frac{1}{2d}\sum_{\substack{\edge\in\edgesint \\ \edge= K|L}}  |\edge|\rho_{\edge} \bfu_{\edge} \cdot \bfn_{K,\edge} \Big( \sum_{\edge' \in\edges(L)} (\bfv_{\edge'}\cdot\bfw_{\edge'}-\bfv_{\edge}\cdot\bfw_{\edge}) +  \sum_{\edge' \in\edges(K)} (\bfv_{\edge}\cdot\bfw_{\edge}-\bfv_{\edge'}\cdot\bfw_{\edge'})  \Big ).
\end{aligned}
\]
Hence,
\[
\begin{aligned}
 |\mcal{Q}^{\rm mass}_\edges(\rho,\bfu,\bfv,\bfw)| 
  & \leq \frac{1}{2d}\sum_{\substack{\edge\in\edgesint \\ \edge= K|L}}  |\edge|\, |\rho_{\edge}|\, |\bfu_{\edge}|\,\sum_{\edge' \in\edges(K)\cup\edges(L)} |\bfv_{\edge'}\cdot\bfw_{\edge'}-\bfv_{\edge}\cdot\bfw_{\edge}|\\
  & \leq \frac{1}{2d} \norm{\rho}_{\xL^\infty(\Omega)}\,\sum_{\substack{\edge\in\edgesint \\ \edge= K|L}}  |\edge|\, |\bfu_{\edge}|\,\sum_{\edge' \in\edges(K)\cup\edges(L)} |\bfv_{\edge'}\cdot\bfw_{\edge'}-\bfv_{\edge}\cdot\bfw_{\edge}|.
\end{aligned}
\]
Now, observing that $\bfv_{\edge'}\cdot\bfw_{\edge'}-\bfv_{\edge}\cdot\bfw_{\edge} = \bfv_{\edge'}\cdot (\bfw_{\edge'}-\bfw_{\edge})+\bfw_{\edge} \cdot(\bfv_{\edge'}-\bfv_{\edge})$, one obtains:
$
 |\mcal{Q}^{\rm mass}_\edges(\rho,\bfu,\bfv,\bfw)| 
  \leq \frac{1}{2d} \norm{\rho}_{\xL^\infty(\Omega)} (T_1+T_2)
$
with :
\[
 T_1 = \sum_{\substack{\edge\in\edgesint \\ \edge= K|L}}  |\edge|\, |\bfu_{\edge}|\sum_{\edge' \in\edges(K)\cup\edges(L)} |\bfv_{\edge'}| \,|\bfw_{\edge'}-\bfw_{\edge}|, \qquad  T_2 = \sum_{\substack{\edge\in\edgesint \\ \edge= K|L}}  |\edge|\, |\bfu_{\edge}|\sum_{\edge' \in\edges(K)\cup\edges(L)} |\bfw_{\edge}| \,|\bfv_{\edge'}-\bfv_{\edge}|.
\]
Let us estimate $T_1$, a similar calculation yields the estimate on $T_2$. Using twice the Cauchy-Schwarz inequality yields:
\[
 \begin{aligned}
 T_1 
 &\leq \sum_{\substack{\edge\in\edgesint \\ \edge= K|L}}  |\edge|\, |\bfu_{\edge}| \Big ( \sum_{\edge' \in\edges(K)\cup\edges(L)} |\bfv_{\edge'}|^{2} \Big )^\frac 12  \Big ( \sum_{\edge' \in\edges(K)\cup\edges(L)} |\bfw_{\edge'}-\bfw_{\edge}|^2 \Big )^\frac 12 \\
  &=  \sum_{\substack{\edge\in\edgesint \\ \edge= K|L}}  |D_\edge|^{\frac 12}\, |\bfu_{\edge}| \Big ( \sum_{\edge' \in\edges(K)\cup\edges(L)} |\bfv_{\edge'}|^{2} \Big )^\frac 12  \Big ( \frac{|\edge|^2}{|D_\edge|} \sum_{\edge' \in\edges(K)\cup\edges(L)} |\bfw_{\edge'}-\bfw_{\edge}|^2 \Big )^\frac 12 \\
  &\leq \underbrace{\bigg ( \sum_{\substack{\edge\in\edgesint \\ \edge= K|L}}  |D_\edge|\, |\bfu_{\edge}|^2  \sum_{\edge' \in\edges(K)\cup\edges(L)} |\bfv_{\edge'}|^{2} \bigg )^{\frac 12} }_{T_{1,1}} \, \underbrace{\bigg ( \sum_{\substack{\edge\in\edgesint \\ \edge= K|L}}  \frac{|\edge|^2}{|D_\edge|} \sum_{\edge' \in\edges(K)\cup\edges(L)} |\bfw_{\edge'}-\bfw_{\edge}|^2 \bigg )^{\frac 12}}_{T_{1,2}}.
 \end{aligned}
\]
In $T_{1,1}$, H\"older's inequality with powers $3$ and $3/2$ gives:
\[
 \begin{aligned}
  T_{1,1}
  & = \bigg ( \sum_{\substack{\edge\in\edgesint \\ \edge= K|L}}  |D_\edge|^{\frac 13}\, |\bfu_{\edge}|^2 \, |D_\edge|^{\frac 23}\sum_{\edge' \in\edges(K)\cup\edges(L)} |\bfv_{\edge'}|^{2} \bigg )^{\frac 12}\\
  & \leq \bigg ( \sum_{\substack{\edge\in\edgesint \\ \edge= K|L}}  |D_\edge|\, |\bfu_{\edge}|^6  \bigg )^{\frac 16} 
  \bigg ( \sum_{\substack{\edge\in\edgesint \\ \edge= K|L}} |D_\edge| \Big (\sum_{\edge' \in\edges(K)\cup\edges(L)} |\bfv_{\edge'}|^{2} \Big )^{\frac 32} \bigg )^{\frac 13}\\
  &\leq C(\theta_0)\,\norm{\bfu}_{\xL^6(\Omega)^d} \, \bigg ( \sum_{\substack{\edge\in\edgesint \\ \edge= K|L}}  \Big (\sum_{\edge' \in\edges(K)\cup\edges(L)} |D_{\edge'}|^{\frac 23} \, |\bfv_{\edge'}|^{2} \Big )^{\frac 32} \bigg )^{\frac 13}
 \end{aligned}
\]
Using again H\"older's inequality with powers $3$ and $3/2$  in the sum over $\edge'$ and recalling that ${\rm card} \lbrace \edges(K)\cup\edges(L)\rbrace=4d-1$
yields $T_{1,1}\leq  C(d,\theta_0)\norm{\bfu}_{\xL^6(\Omega)^d} \norm{\bfv}_{\xL^3(\Omega)^d}$. It remains to estimate $T_{1,2}$, which is done as follows:
\[
\begin{aligned}
T_{1,2}^2= \sum_{\substack{\edge\in\edgesint \\ \edge= K|L}}  \frac{|\edge|^2}{|D_\edge|} & \sum_{\edge' \in\edges(K)\cup\edges(L)} |\bfw_{\edge'}-\bfw_{\edge}|^2 \\
& =  \sum_{\substack{\edge\in\edgesint \\ \edge= K|L}}  \frac{|\edge|^2}{|D_\edge|} \Big ( \sum_{\edge' \in\edges(K)} |\bfw_{\edge'}-\bfw_{\edge}|^2 + \sum_{\edge' \in\edges(L)} |\bfw_{\edge'}-\bfw_{\edge}|^2 \Big ) \\
& \leq \sum_{\substack{\edge\in\edgesint \\ \edge= K|L}}  \frac{|\edge|^2}{|D_\edge|} \Big ( \sum_{\edge',\edge'' \in\edges(K)} |\bfw_{\edge'}-\bfw_{\edge''}|^2 + \sum_{\edge',\edge'' \in\edges(L)} |\bfw_{\edge'}-\bfw_{\edge''}|^2 \Big ).
\end{aligned}
\]
Reordering the sum and using the regularity of the discretization yields:
\[
 T_{1,2}^2\leq C(\theta_0) \sum_{K\in\mesh} \sum_{\edge \in\edges(K)} h_K^{d-2}\sum_{\edge',\edge'' \in\edges(K)} |\bfw_{\edge'}-\bfw_{\edge''}|^2 =C(\theta_0) \norm{\bfw}_\fv^2 \leq C(d,\Omega,\theta_0) \norm{\bfw}_\brok^2,
\]
by Lemma \ref{lmm:H1ns}. Hence $T_1 \leq C(d,\Omega,\theta_0)\norm{\bfu}_{\xL^6(\Omega)^d} \norm{\bfv}_{\xL^3(\Omega)^d} \norm{\bfw}_\brok$, and a similar calculation gives $T_2 \leq C(d,\Omega,\theta_0)\norm{\bfu}_{\xL^6(\Omega)^d} \norm{\bfw}_{\xL^3(\Omega)^d} \norm{\bfv}_\brok$, which concludes the proof.
\end{proof}

%
%
\subsection{Estimates on the momentum convection term}

We define a discrete counterpart of $\int_\Omega \divv(\rho\bfu\otimes\bfv)\cdot\bfw$, the variational formulation of the convection term in the momentum balance equation as follows:
\begin{equation} \label{eq:quadrilin_form_edges}
\mcal{Q}^{\rm mom}_\edges : \left \lbrace \begin{array}{cccl}
		  &  \hspace{-0.5cm} \xH_\mesh(\Omega) \times (\xH_{\edges,0}(\Omega)^d)^3 &\hspace{-0.3cm} \longrightarrow \hspace{-0.3cm} & \xR \\
		  &  (\rho,\bfu,\bfv,\bfw) & \hspace{-0.3cm} \longmapsto \hspace{-0.3cm} & \ds \mcal{Q}^{\rm mom}_\edges(\rho,\bfu,\bfv,\bfw) =
		     \sum_{\edge\in\edgesint} \bfw_\edge \cdot \sum_{\edged \in\edgesd(D_\edge)} \fluxd(\rho,\bfu)\, \bfv_\edged,
\end{array} \right.
\end{equation}
where $\bfv_\edged=\frac 12 (\bfv_\edge+\bfv_{\edge'})$ for $\edged=D_\edge|D_{\edge'}$ (and $\fluxd(\rho,\bfu)$ vanishes at external faces).

\medskip
This section is devoted to derive some estimates on $\mcal{Q}^{\rm mom}_\edges(\rho,\bfu,\bfv,\bfw)$.
First, the analysis of the scheme actually requires an equivalent (or nearly equivalent) re-formulation of this form on the primal mesh $\mesh$, that makes use of the primal fluxes $\fluxK(\rho,\bfu)$.
Indeed, contrary to the dual fluxes $\fluxd(\rho,\bfu)$, the expression of  $\fluxK(\rho,\bfu)$ with respect to the unknowns $\rho\in \xH_\mesh(\Omega)$ and $\bfu\in \xH_{\edges,0}(\Omega)^d$ is simple.
This motivates the introduction of the following auxiliary mapping:
\[
\mcal{Q}^{\rm mom}_\mesh : \left \lbrace \begin{array}{cccl}
		  &  \hspace{-0.5cm}\xH_\mesh(\Omega) \times (\xH_{\edges,0}(\Omega)^d)^3 &\hspace{-0.3cm} \longrightarrow \hspace{-0.3cm}& \xR \\
		  &  (\rho,\bfu,\bfv,\bfw) & \hspace{-0.3cm}\longmapsto \hspace{-0.3cm} & \ds \mcal{Q}^{\rm mom}_\mesh(\rho,\bfu,\bfv,\bfw)
		     =\hspace{-0.1cm}\sum_{K\in \mesh} \bfw_K \cdot \hspace{-0.2cm}\sum_{\edge \in \edges(K)} \fluxK(\rho,\bfu)\ \bfv_\edge,
\end{array} \right.
\]
where $\bfw_K =\sum_{\edge \in \edges(K)} \xi_K^\edge\, \bfw_\edge$ is a convex combination of $(\bfw_\edge)_{\edge \in \edges(K)}$, whose coefficients are defined in \eqref{eq:xiksigma} (thus, as said before, we have in practice $\xi_K^\edge = 1/{\rm card}\, \edges(K)=1/(2d)$).
The following lemma provides a bound of the error made when replacing $\mcal{Q}^{\rm mom}_\edges$ by $\mcal{Q}^{\rm mom}_\mesh$ in the weak formulation of the scheme.

\begin{lem} \label{lmm:convconv}
Let $\disc=(\mesh,\edges)$ be a staggered discretization of $\Omega$ in the sense of Definition \ref{def:disc}, such that $\theta_\disc \leq \theta_0$.
Let $\rho \in \xH_\mesh(\Omega)$, and $\bfu$ and $\bfv$ be two elements of $\xH_{\edges,0}(\Omega)^d$.
Then, for any $\eps\in(0,1]$ if $d=2$ and $\eps \in [1/2,1]$ if $d=3$, there exists $C$ depending only on $\Omega$, $d$, $\theta_0$ and $\eps$ such that:
\begin{equation} \label{eq:convconv}
|\mcal{Q}^{\rm mom}_\edges(\rho,\bfu,\bfu,\bfv)-\mcal{Q}^{\rm mom}_\mesh(\rho,\bfu,\bfu,\bfv)|
\leq C\, \norm{\rho}_{\xL^{\infty}(\Omega)}  \, \norm{\bfu}_\brok^2\, \norm{\bfv}_\brok\, h_\disc^{1-\eps}.
\end{equation}
\end{lem}

\begin{proof}
Let us denote $R(\rho,\bfu,\bfu,\bfv) =  \mcal{Q}^{\rm mom}_\edges(\rho,\bfu,\bfu,\bfv)-\mcal{Q}^{\rm mom}_\mesh(\rho,\bfu,\bfu,\bfv)$.
In the expression \eqref{eq:quadrilin_form_edges},
for $\edge=K|L$, let us split the sum over the fluxes through the faces of $D_\edge$ in the sum over the dual faces, on one side, included in $K$ and, on the other side, included in $L$.
We get by conservativity (\ie\ using $\fluxK=-F_{L,\edge}$):
\[
\mcal{Q}^{\rm mom}_\edges(\rho,\bfu,\bfu,\bfv)=
\sum_{K \in \mesh}\ \sum_{\edge \in \edges(K)} \bfv_\edge \cdot \Bigl(
\fluxK(\rho,\bfu)\ \bfu_\edge + \hspace{-3ex}
\sum_{\begin{array}{c} \scriptstyle \edged\in\edgesd(D_\edge),\\[-0.5ex] \scriptstyle \edged\subset K,\ \edged=D_\edge|D_\edge' \end{array}}
\fluxd(\rho,\bfu)\ \frac{\bfu_\edge+ \bfu_{\edge'}} 2 \Bigr).
\]
Let us write $\mcal{Q}^{\rm mom}_\edges(\rho,\bfu,\bfu,\bfv)=T_1(\rho,\bfu,\bfu,\bfv)+T_2(\rho,\bfu,\bfu,\bfv)$ with:
\[
\begin{aligned}
& \begin{array}{ll}
\displaystyle
   T_1(\rho,\bfu,\bfu,\bfv) = 
\sum_{K \in \mesh} \bfv_K \cdot   \sum_{\edge \in \edges(K)} \Bigl(
\fluxK(\rho,\bfu)\ \bfu_\edge + \hspace{-4ex}
\sum_{\begin{array}{c} \scriptstyle \edged\in\edgesd(D_\edge),\\[-0.5ex] \scriptstyle \edged\subset K,\ \edged=D_\edge|D_\edge' \end{array}}
\hspace{-3ex} \fluxd(\rho,\bfu)\ \frac{\bfu_\edge+ \bfu_{\edge'}} 2 \Bigr),
  \end{array}
\\
&\begin{array}{ll}
\displaystyle
 T_2(\rho,\bfu,\bfu,\bfv)= 
\sum_{K \in \mesh}\ \sum_{\edge \in \edges(K)} (\bfv_\edge-\bfv_K)\cdot &\hspace{-1ex} \Bigl(
\fluxK(\rho,\bfu)\ \bfu_\edge
\\ &  \displaystyle+
\sum_{\begin{array}{c} \scriptstyle \edged\in\edgesd(D_\edge),\\[-0.5ex] \scriptstyle \edged\subset K,\ \edged=D_\edge|D_\edge' \end{array}}
\hspace{-3ex} \fluxd(\rho,\bfu)\ \frac{\bfu_\edge+ \bfu_{\edge'}} 2 \Bigr).
\end{array}
\end{aligned}
\]
By assumption $(H2)$ in Def. \ref{def:dualfluxes}, we remark that  $T_1(\rho,\bfu,\bfu,\bfv)=\mcal{Q}^{\rm mom}_\mesh(\rho,\bfu,\bfu,\bfv)$ so that $R(\rho,\bfu,\bfu,\bfv)=T_2(\rho,\bfu,\bfu,\bfv)$. Using now $(H1)$, we write $R(\rho,\bfu,\bfu,\bfv)=R_1(\rho,\bfu,\bfu,\bfv) + R_2(\rho,\bfu,\bfu,\bfv)$ with:
\[ \begin{aligned}
R_1(\rho,\bfu,\bfu,\bfv)= &
\sum_{K \in \mesh}\ \sum_{\edge \in \edges(K)} (\bfv_\edge-\bfv_K) \cdot \Big (
\sum_{\begin{array}{c} \scriptstyle \edged\in\edgesd(D_\edge) ,\\[-0.5ex] \scriptstyle \edged\subset K,\ \edged=D_\edge|D_\edge' \end{array}}
\fluxd(\rho,\bfu)\ \frac{\bfu_{\edge'}-\bfu_\edge} 2\Big ),
\\
R_2(\rho,\bfu,\bfu,\bfv)= &
\sum_{K \in \mesh}\ \sum_{\edge \in \edges(K)} (\bfv_\edge-\bfv_K)\cdot \ \bfu_\edge\ \xi_K^\edge \Bigl(\sum_{\edge'\in\edges(K)} F_{K,\edge'}(\rho,\bfu) \Bigr).
\end{aligned}\]
The assumption $(H3)$ yields $|\fluxd(\rho,\bfu)| \leq C \norm{\rho}_{\xL^{\infty}(\Omega)} \norm{\bfu}_{\xL^\infty(\Omega)^d}\ h_K^{d-1}$.
As a consequence, since $\bfv_K$ is a convex combination of the $(\bfv_\edge)_{\edge\in\edges(K)}$, we have for any $K \in \mesh$:
\begin{multline*} 
\Bigl|\sum_{\edge \in \edges(K)} (\bfv_\edge-\bfv_K) \cdot \Big (
\sum_{\begin{array}{c} \scriptstyle \edged\in\edgesd(D_\edge),\\[-0.5ex] \scriptstyle \edged\subset K,\ \edged=D_\edge|D_\edge' \end{array}}
\fluxd(\rho,\bfu)\ \frac{\bfu_{\edge'}-\bfu_\edge} 2  \Big ) \Bigr| 
\\ 
\leq C\, \norm{\rho}_{\xL^{\infty}(\Omega)}\, \norm{\bfu}_{\xL^\infty(\Omega)^d}\, h_\disc \sum_{\edge,\,\edge',\,\edge'',\,\edge''' \in \edges(K)}
h_K^{d-2} \ |\bfv_\edge-\bfv_{\edge'}|\ |\bfu_{\edge''}-\bfu_{\edge'''}|,
\end{multline*}
and, for $\edge,\ \edge' \in \edges(K)$, the quantity $|\bfu_\edge-\bfu_{\edge'}|$ (or $|\bfv_\edge-\bfv_{\edge'}|$) appears in the sum a finite number of times which depends of the dimension $d$. Hence, by the Cauchy-Schwarz inequality:
\begin{equation}\label{eq:R1}
\begin{aligned}
|R_1(\rho,\bfu,\bfu,\bfv)| & 
\leq C\,  \norm{\rho}_{\xL^{\infty}(\Omega)}\, \norm{\bfu}_{\xL^{\infty}(\Omega)^d}\, \norm{\bfu}_\fv\,   \norm{\bfv}_\fv\,   h_\disc
\\ & 
\leq C'\, \norm{\rho}_{\xL^{\infty}(\Omega)}\, \norm{\bfu}_{\xL^{\infty}(\Omega)^d}\, \norm{\bfu}_\brok\, \norm{\bfv}_\brok\, h_\disc,
\end{aligned}
\end{equation}
by Lemma \ref{lmm:H1ns}. Let us now turn to $R_2(\rho,\bfu,\bfu,\bfv)$. By definition of $\bfv_K$, we have  $\sum_{\edge \in \edges(K)} \xi_K^\edge\ (\bfv_\edge-\bfv_K)=0$, and we obtain that:
\[
R_2(\rho,\bfu,\bfu,\bfv)=
\sum_{K \in \mesh}\ \sum_{\edge \in \edges(K)} (\bfv_\edge-\bfv_K) \cdot \xi_K^\edge\ (\bfu_\edge-\bfu_K)\ \Bigl(\sum_{\edge'\in\edges(K)} F_{K,\edge'}(\rho,\bfu) \Bigr),
\]
so, once again:
\begin{equation}\label{eq:R2}
\begin{aligned}
|R_2(\rho,\bfu,\bfu,\bfv)| & \leq C\norm{\rho}_{\xL^{\infty}(\Omega)}  \norm{\bfu}_{\xL^{\infty}(\Omega)^d}\ h_\disc
\sum_{K \in \mesh}\ h_K^{d-2} \sum_{\edge \in \edges(K)}  |\bfv_\edge-\bfv_K|\ |\bfu_\edge-\bfu_K| \\ 
 & \leq C\,  \norm{\rho}_{\xL^{\infty}(\Omega)}\, \norm{\bfu}_{\xL^{\infty}(\Omega)^d}\, \norm{\bfu}_\fv\,  \norm{\bfv}_\fv\,  h_\disc \\
 & \leq C'\, \norm{\rho}_{\xL^{\infty}(\Omega)}\, \norm{\bfu}_{\xL^{\infty}(\Omega)^d}\, \norm{\bfu}_\brok\, \norm{\bfv}_\brok\, h_\disc.
\end{aligned}
\end{equation}
We conclude thanks to an inverse inequality. Let $q$ and $q'$ in $[1,+\infty]$ such that $1/q+1/q'=1$.
We may write $\norm{\bfu}_{\xL^{\infty}(\Omega)^d} = |\bfu_\Sigma|$ for some $\Sigma\in\edgesint$.
Hence,
\[
\norm{\bfu}_{\xL^{\infty}(\Omega)^d}= |D_\Sigma|^{-1} \int_{D_\Sigma}|\bfu(\bfx)| \dx 
\leq C(d)\, |D_\Sigma|^{1/q'-1}\, \norm{\bfu}_{\xL^q(D_\Sigma)^d}
\]
by H\"older's inequality, and therefore $\norm{\bfu}_{\xL^{\infty}(\Omega)^d} \leq C(d,\theta_0)\, h_\disc^{d(1/q'-1)}\, \norm{\bfu}_{\xL^q(\Omega)^d}$. 
Now by Lemma \ref{lmm:injsobolev_br}, we obtain (in particular) that $\norm{\bfu}_{\xL^q(\Omega)^d}\leq C(q)\, \norm{\bfu}_{\brok}$ for $q \in [2,+\infty)$ in 2D and for $q \in [3,6]$ in 3D, thus for $d(1/q'-1) \in [-1,0)$ in 2D and $d(1/q'-1) \in [-1,-1/2]$ in 3D.
Combining these bounds with \eqref{eq:R1} and \eqref{eq:R2} yields the inequality that we are seeking.
\end{proof}

Let us now give some estimates on the auxiliary form $\mcal{Q}^{\rm mom}_\mesh(\rho,\bfu,\bfu,\bfv)$, which are discrete counterparts to classical continuous estimates.

\begin{lem}[Estimates on $\mcal{Q}^{\rm mom}_\mesh$] \label{lmm:estimates_Bbar}
Let $\disc=(\mesh,\edges)$ be a staggered discretization of $\Omega$ in the sense of Definition \ref{def:disc}, such that $\theta_\disc \leq \theta_0$.
Then, there exists two constants $C_1$ and $C_2$, only depending on $\Omega$, $d$ and $\theta_0$, such that
\begin{equation} \label{eq:estimates_Bbar1}
|\mcal{Q}^{\rm mom}_\mesh(\rho,\bfu,\bfv,\bfw)| \leq C_1 \norm{\rho}_{\xL^{\infty}(\Omega)} \norm{\bfu}_{\xL^4(\Omega)^d} \norm{\bfv}_{\xL^4(\Omega)^d} \norm{\bfw}_\brok
\leq C_2 \norm{\rho}_{\xL^{\infty}(\Omega)} \norm{\bfu}_\brok \norm{\bfv}_\brok\ \norm{\bfw}_\brok,
\end{equation}
for all $\rho$ in $\xH_\mesh(\Omega)$ and $\bfu$, $\bfv$, $\bfw$ in $\xH_{\edges,0}(\Omega)^d$.
\end{lem}

\begin{proof}
By definition of $\mcal{Q}^{\rm mom}_\mesh(\rho,\bfu,\bfv,\bfw)$, we have
\[
\begin{aligned}
 \mcal{Q}^{\rm mom}_\mesh(\rho,\bfu,\bfv,\bfw) 
& =\sum_{K\in \mesh} \bfw_K \cdot \sum_{\edge \in \edges(K)} \fluxK(\rho,\bfu)\ \bfv_\edge
& =\sum_{K\in \mesh} \bfw_K \cdot \sum_{\edge \in \edges(K)}|\edge| (\rho_\edge \bfu_\edge \cdot \bfn_{K,\edge}) \ \bfv_\edge.
\end{aligned}
\]
Reordering the sum and applying the Cauchy-Schwarz inequality twice, we get 
\[ \begin{aligned}
|\mcal{Q}^{\rm mom}_\mesh(\rho,\bfu,\bfu,\bfv)|
&
= \Bigl| \sum_{\substack{\edge \in \edgesint \\ \edge=K|L}} \sqrt{|D_\edge|} (\rho_\edge \bfu_\edge
\cdot \bfn_{K,\edge})\ \bfv_\edge \cdot \frac{|\edge|}{\sqrt{|D_\edge|}}(\bfw_L-\bfw_K) \Bigr|
\\ &
\leq \norm{\rho}_{\xL^\infty(\Omega)}
\Bigl( \sum_{\substack{\edge \in \edgesint \\ \edge=K|L}} |D_\edge| |\bfu_\edge |^2|\bfv_\edge |^2 \Bigr)^{\frac 12}
\Bigl( \sum_{\substack{\edge \in \edgesint \\ \edge=K|L}}   \frac{|\edge|^2}{|D_\edge|}|\bfw_L-\bfw_K|^2  \Bigr)^{\frac 12}
\\ &
\leq \norm{\rho}_{\xL^\infty(\Omega)}\, \norm{\bfu}_{\xL^4(\Omega)^d}\, \norm{\bfv}_{\xL^4(\Omega)^d}
\Bigl( \sum_{\substack{\edge \in \edgesint \\ \edge=K|L}}   \frac{|\edge|^2}{|D_\edge|}|\bfw_L-\bfw_K|^2  \Bigr)^{\frac 12}.
\end{aligned}\]
Now, as $\bfw_L$ and $\bfw_K$ are convex combinations of $(\bfw_\edge)_{\edge \in \edges(L)}$ and $(\bfw_\edge)_{\edge \in \edges(K)}$ respectively, we get, using the fact that the number of faces of an element is equal to $2d$:
\[
 \begin{aligned}
|\bfw_L-\bfw_K|^2 & \leq (2d)^2 \sum_{\substack{\edge' \in \edges(L) \\ \edge'' \in \edges(K)}} |\bfw_{\edge'}-\bfw_{\edge''}|^2
& \leq (2d)^3 \Bigl( \sum_{\edge',\edge'' \in \edges(L) } |\bfw_{\edge'}-\bfw_{\edge''}|^2 
+ \sum_{\edge',\edge'' \in \edges(K) } |\bfw_{\edge'}-\bfw_{\edge''}|^2  \Bigr).
\end{aligned}
\]
Hence, 
\[\begin{aligned}
\sum_{\substack{\edge \in \edgesint \\ \edge=K|L}}  \frac{|\edge|^2}{|D_\edge|}\, |\bfw_L-\bfw_K|^2
&
\leq  C(d) \sum_{\substack{\edge \in \edgesint \\ \edge=K|L}}   \frac{|\edge|^2}{|D_\edge|}
\Bigl( \sum_{\edge',\edge'' \in \edges(L) } |\bfw_{\edge'}-\bfw_{\edge''}|^2 
+ \sum_{\edge',\edge'' \in \edges(K) } |\bfw_{\edge'}-\bfw_{\edge''}|^2  \Bigr)
\\ &
\leq 2\, C(d) \sum_{K \in \mesh} \sum_{\edge \in\edges(K)} \frac{|\edge|^2}{|D_\edge|}
\sum_{\edge',\edge'' \in \edges(K) } |\bfw_{\edge'}-\bfw_{\edge''}|^2
\\[0.5ex] &
\leq  C(d,\theta_0) \sum_{K \in \mesh} \sum_{\edge \in \edges(K)} h_K^{d-2} \sum_{\edge',\edge'' \in \edges(K) } |\bfw_{\edge'}-\bfw_{\edge''}|^2
\\[0.5ex] &
\leq  C(d,\Omega,\theta_0)\ \norm{\bfw}_\brok^2,
\end{aligned} \]
by Lemma \ref{lmm:H1ns}.
This proves the first inequality in \eqref{eq:estimates_Bbar1}.
The second inequality follows from the fact that, by a discrete H\"older inequality, $\norm{\bfu}_{\xL^4(\Omega)^d} \leq |\Omega|^{1/12} \norm{\bfu}_{\xL^6(\Omega)^d}$ and from the discrete Sobolev inequality $\norm{\bfu}_{\xL^6(\Omega)^d} \leq C(d,\theta_0)\, \norm{\bfu}_\brok$ stated in Lemma \ref{lmm:injsobolev_br}.
\end{proof}
%
%
\section{Main properties of the scheme} \label{sec:prop}

We first establish stability properties enjoyed by the scheme (Section \ref{subsec:estimates}), which are the discrete analogues of estimates satisfied by the exact solutions of problem \eqref{eq:pb}: maximum principle for the density, and $\xL^\infty(\xL^2)$- and $\xL^2(\xH^1)$-estimates for the velocity.
This latter estimate is an easy consequence of a discrete kinetic energy balance, which is stated in Lemma \ref{lmm:kinetic_energy}.
In a second step (Section \ref{subsec:existence}), we show that these estimates yield the existence of a solution to the scheme, by an argument issued from the topological degree theory.
%
%
\subsection{Estimates on the discrete solution} \label{subsec:estimates}

Let us start by stating a discrete equivalent of the following transport equation satisfied by $\rho^2/2$, if $(\rho,\bfu)$ is solution to problem \eqref{eq:pb}:
\[
\dv_t ( \frac{\rho^2} 2  )  + \dive (\frac{\rho^2} 2 \bfu) = 0.
\]
This discrete identity is rather classical (see \eg\ \cite{eym-00-book}) and we only sketch its proof.

\begin{lem} \label{lmm:rho2}
Any solution to the scheme \eqref{eq:impl_scheme} satisfies the following equality, for all $K \in \mesh$ and $1 \leq n \leq N$:
\begin{equation} \label{eq:rho2}
\dfrac{|K|}{2 \delta t} \Big ( (\rho^n_K)^2-(\rho_K^{n-1})^2 \Big ) 
+ \frac 1 2 \sum_{\edge \in\edges(K)} |\edge|\, (\rho_\edge^n)^2 \, (\bfu_\edge^n \cdot\bfn_{K,\edge}) + R_K^n = 0, 
\end{equation}
where
\begin{equation} \label{eq:R_rho2}
R_K^n=\frac{|K|}{2\delta t} (\rho_K^n-\rho_K^{n-1})^2 
- \frac 1 2 \sum_{\edge \in\edges(K)} |\edge|\, (\rho_\edge^n-\rho_K^n)^2\, (\bfu_\edge^n \cdot\bfn_{K,\edge}) \geq 0.
\end{equation}
\end{lem}

\begin{proof}
Multiply \eqref{eq:sch_mass} by $|K|\, \rho_K^n$.
In the discrete time derivative term, use the identity $2\,(a^2-a b)= (a^2-b^2)+(a-b)^2$ with $a=\rho_K^n$ and $b=\rho_K^{n-1}$.
In the discrete convection term, use the identity $2\, a b = a^2+ b^2 - (a-b)^2$ with $a=\rho_K^n$ and $b=\rho_\edge^n$.
The quantity $\sum_{\edge \in\edges(K)} |\sigma| (\rho_K^n)^2 (\bfu_\edge^n \cdot\bfn_{K,\edge})$ vanishes because $(\dive\bfu)_K^n=0$.
\end{proof}

\begin{rmrk}
A similar result may be obtained for the partial differential equation satisfied by $\psi(\rho)$, where $\psi$ is any convex real function, and generalized to the case where the velocity field is not divergence-free (see the appendices of \cite{her-14-imp}).
\end{rmrk}

\medskip
We now prove a discrete equivalent of the kinetic energy balance.
Recall that in the continuous setting, this relation is formally obtained by taking the scalar product between $\bfu$ and the momentum balance equation \eqref{eq:mom}; this yields, using the mass balance equation \eqref{eq:mass} twice:
\[
\partial_t (\rho\frac {|\bfu|^2} 2) + \dive \bigl(\rho\frac{|\bfu|^2} 2 \bfu \bigr) - \lapi\bfu \cdot \bfu+ \gradi p \cdot \bfu = 0.
\]
In the discrete setting, this computation must be mimicked on the mesh used for the discretization of the momentum balance equation, namely the dual mesh.
This is the reason why we chose the density on the diamond cells and the mass fluxes on the faces of the dual mesh in such a way that the discrete mass balance equation \eqref{eq:mass_D} holds on the dual cells.
Thanks to this choice, we obtain the following identity.

\begin{lem}[Discrete kinetic energy balance] \label{lmm:kinetic_energy}
Any solution to the scheme \eqref{eq:impl_scheme} satisfies the following equality, for all $\edge \in \edgesint$ and $1 \leq n \leq N$:
\begin{equation} \label{eq:imp_kin}
\begin{aligned}
\dfrac{1}{2\delta t} \Bigl( \rho^n_\Ds\, |\bfu^n_\edge|^2-\rho^{n-1}_\Ds\, |\bfu_\edge^{n-1}|^2 \Bigr)
+ \dfrac{1}{2|D_\edge|} \ & \sum_{\substack{\edged \in \edgesd(D_\edge) \\ \edged=\edgeedgeprime }} \fluxd^n \ \bfu^n_\edge \cdot \bfu^n_{\edge'} \\[1ex]
& -(\lapi\bfu)^n_\edge \cdot \bfu_\edge^n
+ (\gradi p)^n_\edge \cdot \bfu_\edge^n+R_\edge^n=0, 
\end{aligned}
\end{equation}
where
\[
R_\edge^n= \frac 1 { 2 \delta t}\ \rho^{n-1}_\Ds\ | \bfu^n_\edge - \bfu^{n-1}_\edge |^2.
\]
\end{lem}
\begin{proof}
Let us take the scalar product of the discrete momentum balance equation \eqref{eq:sch_mom} by the corresponding velocity unknown $\bfu_\edge^n$, which gives the relation $T_1 - (\lapi \bfu)^n_\edge \cdot \bfu_\edge^n + (\gradi p)^n_\edge \cdot \bfu_\edge^n=0$, with:
\[
T_1 =
\Bigl(
\dfrac 1 {\delta t} \bigl(\rho^n_\Ds \bfu^n_\edge-\rho^{n-1}_\Ds \bfu_\edge^{n-1} \bigr)
+ \dfrac 1 {2|D_\edge|}\ \sum_{\substack{\edged \in \edgesd(D_\edge) \\ \edged=\edgeedgeprime}}  \fluxd^n\ (\bfu^n_\edge+\bfu^n_{\edge'})
\Bigr)\cdot \bfu_\edge^n.
\]
Now, using the identity $2\, (\rho|\bfa|^2-\rho^*\bfa\cdot\bfb) = \rho|\bfa|^2 - \rho^*|\bfb|^2 + \rho^* |\bfa-\bfb|^2 + (\rho-\rho^*)|\bfa|^2$ with $\rho=\rho_\Ds^n$, $\rho^*=\rho_\Ds^{n-1}$, $\bfa=\bfu_\edge^n$ and $\bfb=\bfu_\edge^{n-1}$, we obtain
\[\begin{aligned}
T_1
&
= \dfrac 1 {2\delta t} \Bigl( \rho^n_\Ds |\bfu^n_\edge|^2-\rho^{n-1}_\Ds |\bfu_\edge^{n-1}|^2 \Bigr)
+ \dfrac 1 {2\delta t}\ \rho^{n-1}_\Ds\ | \bfu^n_\edge - \bfu^{n-1}_\edge |^2
+ \dfrac 1 {2 |D_\edge|}\ \sum_{\edged=\edgeedgeprime} \fluxd^n\ \bfu^n_\edge \cdot \bfu^n_{\edge'}
\\ &
+ \Bigl(\frac 1 {\delta t} \ (\rho^n_\Ds-\rho^{n-1}_\Ds)
+ |D_\edge|^{-1}\sum_{\edged\in\edges(D_\edge)} \fluxd^n \Bigr)\, \frac{|\bfu_\edge^n|^2} 2.
\end{aligned}\]
The last term is equal to zero since the dual densities $\rho_\Ds$ and the dual fluxes $\fluxd$ are chosen so as to satisfy the discrete mass conservation equation \eqref{eq:mass_D} on the cells of the dual mesh.
This concludes the proof.
\end{proof}

Lemmas \ref{lmm:rho2} and \ref{lmm:kinetic_energy} allow to prove the following proposition, which gathers the "local in time" estimates satisfied by the discrete solutions.
The first three inequalities readily provide by induction uniform (\ie\ independent from the time and space steps) bounds for the solution.
On the opposite, the right hand side of inequality \eqref{eq:p_stab} blows up when the time step tends to zero, which is consistent with the fact that no estimate is known for the pressure in the continuous case; this bound is thus useful only for the proof of the existence of a solution to the scheme.

\begin{prop}[Estimates on the discrete solutions] \label{prop:bvfaible}
Let $\disc$ be a staggered discretization of $\Omega$ in the sense of Definition \ref{def:disc}, such that $\theta_\disc\leq \theta_0$ where $\theta_\disc$ is defined in \eqref{eq:reg}.
For $n\in\lbrace1,..,N\rbrace$, assume that the density $\rho^{n-1}$ is such that $0<\rhomin \leq \rho_K^{n-1}\leq \rhomax$ for all $K$ in $\mesh$.
Then, any solution $(\rho^n,\bfu^n,p^n)$ to the scheme \eqref{eq:impl_scheme} satisfies the following relations:
\begin{align} &
\label{eq:imp_maxprinc}
\rhomin \leq \rho_K^n \leq \rhomax, \quad  \forall K \in \mesh,
\\[0.5ex] & \label{eq:rho_bv} 
\displaystyle \frac 1 2 \sum_{K \in \mesh} |K|\, (\rho^n_K)^2 
+ \frac{\delta t} 2 \sum_{\substack{\edge \in \edgesint \\ \edge=K|L}} |\edge|\, (\rho_L^n-\rho_K^n)^2  \displaystyle |\bfu_\edge^n \cdot\bfn_{K,\edge}|
+ \mathcal{R}_\rho^n
= \frac 1 2 \sum_{K \in \mesh}|K|\, (\rho^{n-1}_K)^2 ,
\\ & \label{eq:imp_stab} \displaystyle
\dfrac 1 2 \sum_{\edge \in \edgesint} |D_\edge|\, \rho^n_\Ds\, |\bfu^n_\edge|^2
+ \delta t \norm{\bfu^n}_\brok^2 + \mathcal{R}^n_\bfu
= \frac 1 2 \sum_{\edge \in \edgesint} |D_\edge|\, \rho^{n-1}_\Ds\, |\bfu^{n-1}_\edge|^2, 
\\[0.5ex] & \label{eq:p_stab}
\norm{p^n}_{L^2(\Omega)} \leq C_p^{n-1},
\end{align}
where $C_p^{n-1}$ only depends on $\rho^{n-1}$, $\bfu^{n-1}$, $\theta_0$, $d$, $\Omega$, $\delta t$ and $h_\disc$.
The terms $\mathcal{R}^n_\rho$ and $\mathcal{R}^n_\bfu$ are the following non-negative remainders which depend on differences of time translates of the density and the velocity respectively:
\[
\mathcal{R}^n_\rho = \frac 1 2 \sum_{K \in \mesh} |K|\  (\rho^n_K-\rho^{n-1}_K)^2, \qquad
\mathcal{R}^n_\bfu = \frac 1 2 \sum_{\edge \in \edgesint} |D_\edge|\ \rho^{n-1}_\Ds\ |\bfu^n_\edge-\bfu^{n-1}_\edge|^2.
\]
\end{prop}

\begin{proof}
The maximum principle for the density \eqref{eq:imp_maxprinc} is a classical consequence of the upwind choice \eqref{eq:sch_mass} and the discrete divergence-free constraint \eqref{eq:sch_div}. 
Relation \eqref{eq:rho_bv} is obtained by summing \eqref{eq:rho2} over the cells of the mesh.
As usual, the convective terms (\ie\ the second term in \eqref{eq:rho2}) vanishes by conservativity; the second term in \eqref{eq:rho_bv} is obtained by summing over the cells the second term of the remainder \eqref{eq:R_rho2}, and using the definition of the upwind approximation of the density at the face.
Similarly, summing Equation \eqref{eq:imp_kin} over the faces $\edge \in \edgesint$ yields \eqref{eq:imp_stab}, since the discrete gradient and divergence operators are dual with respect to the $\xL^2$-inner product (see \eqref{eq:dual_divgrad2}) and the convection term vanishes in the summation once again by conservativity (assumption $(H2)$ of Definition \ref{def:dualfluxes}).

\medskip
Finally, we prove the estimate on the pressure.
The Rannacher-Turek finite element discretization satisfies an \textit{inf-sup} condition, which implies that for $p^n$ solution of \eqref{eq:impl_scheme} (observing that $\int_\Omega p^n(\bfx) \dx=0$), there exists $\bfv$ in $\xH_{\edges,0}(\Omega)$, with $\norm{\bfv}_\brok=1$ and a positive real number $\beta$, depending only on $\Omega$ and $\theta_0$, such that
\[
\int_\Omega p^n(\bfx)\, \dive_{\mesh} \bfv (\bfx) \dx \geq \beta\ \norm{p^n}_{L^2(\Omega)}.
\]
Hence, taking the scalar product of \eqref{eq:sch_mom} by $|D_\edge|\,\bfv_\edge$ and summing over $\edge$ in $\edgesint$, we get $\beta \ \norm{p^n}_{L^2(\Omega)} \leq T_1+T_2+T_3$ with
\[ \begin{aligned} &
T_1 =\dfrac 1 {\delta t}\sum_{\edge\in\edgesint}|D_\edge|(\rho^n_\Ds \bfu^n_\edge-\rho^{n-1}_\Ds \bfu_\edge^{n-1})\cdot\bfv_\edge,
\\ &
T_2 =\sum_{\edge\in\edgesint}\bfv_\edge \cdot \sum_{\edged \in\edgesd(D_\edge)} \fluxd^n \bfu^n_\edged,
\\ &
T_3 =- \sum_{\edge\in\edgesint}|D_\edge|(\lapi \bfu)_\edge^n\cdot\bfv_\edge.
\end{aligned} \]
We prove that each one of these terms is controlled by a constant depending only on $\rho^{n-1}$, $\bfu^{n-1}$, $\theta_0$, $d$, $\Omega$, $\delta t$ and $h_\disc$. 
For the first term we have, by the Cauchy-Schwarz inequality:
\[\begin{aligned}
|T_1| 
&
\leq \dfrac 1 {\delta t}\sum_{\edge \in \edgesint}|D_\edge|(\rho^n_\Ds |\bfu^n_\edge\cdot \bfv_\edge|+\rho^{n-1}_\Ds |\bfu_\edge^{n-1} \cdot \bfv_\edge| )
\\ &
\leq \dfrac{\rhomax^{1/2}}{\delta t} 
\Bigl( \sum_{\edge \in \edgesint}|D_\edge|\, (\rho^n_\Ds)^{1/2}\, |\bfu^n_\edge \cdot \bfv_\edge|
+ \sum_{\edge \in \edgesint}|D_\edge|\, (\rho^{n-1}_\Ds)^{1/2}\, |\bfu_\edge^{n-1} \cdot \bfv_\edge| \Bigr)
\\ &
\leq \dfrac{\rhomax^{1/2}}{\delta t}\ \norm{\bfv}_{L^2(\Omega)}
\Bigl(\bigl( \sum_{\edge \in \edgesint}|D_\edge|\, \rho^n_\Ds\, |\bfu^n_\edge|^2 \bigr)^{1/2}
+ \bigl(\sum_{\edge \in \edgesint}|D_\edge|\, \rho^{n-1}_\Ds\, |\bfu^{n-1}_\edge|^2 \bigr)^{1/2} \Bigr).
\end{aligned}\]
We thus obtain the expected control on $T_1$ by invoking \eqref{eq:imp_stab} and the Poincar\'e inequality of Lemma \ref{lmm:poincare_brok}, which yields $\norm{\bfv}_{L^2(\Omega)} \leq C\ \norm{\bfv}_\brok=C$, with $C$ only depending on $d$, $\Omega$ and $\theta_0$.
For the second term, we have $T_2= \mcal{Q}^{\rm mom}_\mesh(\rho^n, \bfu^n,\bfu^n,\bfv) + \mcal{Q}^{\rm mom}_\edges(\rho^n, \bfu^n,\bfu^n,\bfv) -\mcal{Q}^{\rm mom}_\mesh(\rho^n ,\bfu^n,\bfu^n,\bfv)$ which by Lemma \ref{lmm:convconv} (with $\eps =1$) and Lemma \ref{lmm:estimates_Bbar} yields:
\[
|T_2| \leq C\, \rhomax \norm{\bfu^n}_\brok^2 \norm{\bfv}_\brok
+ C'\rhomax \norm{\bfu^n}_\brok^2 \norm{\bfv}_\brok
\leq C'',
\]
where $C''$ depends on $\Omega$, $d$, $\delta t$, $\rho^{n-1}$, $\bfu^{n-1}$, $\theta_0$ and $h_\disc$.
Finally, to control the last term $T_3$, we observe that:
\[
T_3 =	\sum_{K \in \mesh} \int_K \gradi \tilde{\bfu}^n(\bfx): \gradi \tilde{\bfv} (\bfx) \dx.
\]
Using first the Cauchy-Schwarz inequality for the integration on the cell $K$ and then the discrete Cauchy-Schwarz inequality for the summation over $K$ in $\mesh$, we easily obtain $|T_3|\leq  \norm{\bfu^n}_\brok\norm{\bfv}_\brok$, and we conclude once again by \eqref{eq:imp_stab}. 
\end{proof} 
%
%
\subsection{Existence of a solution to the scheme} \label{subsec:existence}

The existence of a solution to the scheme \eqref{eq:impl_scheme}, which consists in an algebraic non-linear system, is obtained by a topological degree argument.
Its proof is based on an abstract theorem stated in Appendix \ref{thrm:degree}.

\begin{thrm}[Existence of a solution] \label{thrm:existence}
For $n\in\lbrace1,..,N\rbrace$, assume that the density $\rho^{n-1}$ is such that $0 < \rhomin \leq \rho_K^{n-1}\leq \rhomax$ for all $K$ in $\mesh$. Then the non-linear system \eqref{eq:impl_scheme} admits at least one solution $(\rho^n, \bfu^n, p^n)$ in $\xH_\mesh(\Omega) \times \xH_{\edges,0}(\Omega)^d \times \xH_\mesh(\Omega)$, and any possible solution satisfies the estimates of Proposition \ref{prop:bvfaible}.
\end{thrm}

\begin{proof}
This proof makes use of Theorem \ref{thrm:degree}.
Let $N_\mesh={\rm card}(\mesh)$ and $N_\edges=d\ {\rm card}(\edgesint)$; we identify $\xH_\mesh(\Omega)$ with $\xR^{N_\mesh}$ and $\xH_{\edges,0}(\Omega)^d$ with $\xR^{N_\edges}$.
Let $V=\xR^{N_\mesh} \times \xR^{N_\edges} \times \xR^{N_\mesh}$.
We consider the function $F:\ V\times [0,1] \rightarrow V$ given by:
\[
F(\rho,\bfu, p,\lambda)=
\left| \begin{array}{ll} \displaystyle
\frac 1 {\delta t}(\rho_K-\rho_K^{n-1}) + \lambda \frac 1 {|K|} \sum_{\edge \in\edges(K)} \fluxK,
&
K \in \mesh
\\[4ex] \displaystyle
\frac 1 {\delta t}(\rho_\Ds \bfu_\edge-\rho^{n-1}_\Ds \bfu_\edge^{n-1})
+ \lambda \frac 1 {|D_\edge|} \sum_{\edged \in\edgesd(D_\edge)} \fluxd \bfu_\edged
\\[4ex] \hspace{5cm} - (\lapi\bfu)_\edge
+ (\gradi p)_\edge,
 &
\edge \in \edgesint
\\[2ex] \displaystyle
-(\dive\bfu)_K + \frac 1 {|K|} \sum_{L \in \mesh} |L|\ p_L,
&
K \in \mesh.
\end{array} \right.
\]
The function $F$ is continuous from $V\times[0,1]$ to $V$ and the problem $F(\rho,\bfu,p,1)=0$ is equivalent to system \eqref{eq:impl_scheme}.
Indeed, multiplying the third line in the above formula by $|K|$, summing and using the fact that $\bfu_\edge=0$ for $\edge\in\edgesext$ yields $ \sum_{L \in \mesh} |L|\ p_L=0$ and therefore $(\dive\bfu)_K=0$ for all $K\in\mesh$.
Moreover, an easy verification shows that the problem $F(\rho,u,p,\lambda)=0$ for $\lambda$ in $[0,1]$, satisfies the same estimates as stated in Proposition \ref{prop:bvfaible} uniformly  in $\lambda$.
Hence, defining
\[
{\mathcal O}= \Bigl\lbrace (\rho,\bfu,p) \in V \mbox{ s.t. } \frac \rhomin 2 < \rho < 2\,\rhomax,\ \norm{\bfu}_\brok< C \mbox{ and } \norm{p}_{L^2(\Omega)}< C \Bigr\rbrace,
\]
with $C$ (strictly) larger than the right-hand sides of \eqref{eq:imp_stab} and \eqref{eq:p_stab}, the second hypothesis of Theorem \ref{thrm:degree} is also satisfied.
Therefore, in order to prove the existence of at least one solution to the scheme \eqref{eq:impl_scheme}, it remains to show that the topological degree of $F(\rho,\bfu,p,0)$ with respect to $0_V$ and ${\mathcal O}$ is non-zero.
The function  $G:(\rho,\bfu,p) \mapsto F(\rho,\bfu,p,0)$ is clearly differentiable on ${\mathcal O}$, and its jacobian matrix is given by
\[
{\rm Jac}~ G(\rho,\bfu,p)= 
\left [ \begin{array}{c|c}
\ds \frac 1 {\delta t} \text{Id}_{\xR^{N_\mesh\times N_\mesh}} 	& 0
\\[1.5ex] \hline 
\ds A	& \multirow{2}{*}{$S(\bfu,p)$}
\\
0  & 
\end{array} \right ], 
\]
where $A$ is some matrix in $\xR^{N_\edges \times N_\mesh}$ and $S(\bfu,p)\in\xR^{(N_\edges+N_\mesh)\times(N_\edges+ N_\mesh)}$ is the jacobian matrix associated with the \emph{inf-sup} stable Rannacher-Turek finite element discretization of the following Stokes problem:

\emph{find $(\bfu,p)$ such that $\ds \int_\Omega p(\bfx) \dx =0$ and}
\begin{equation}\label{eq:gen_stokes}
\begin{array}{ll} \ds
\frac 1 {\delta t} \rho(\bfx)\, \bfu - \lapi \bfu + \gradi p = 0, & \text{in $\Omega$,}
\\[2ex]
\dive \bfu = 0, & \text{in $\Omega$,}
\\[1ex]
\bfu = 0, & \text{on $\dv \Omega$.}
\end{array}
\end{equation}
With $\lambda=0$, the density unknowns $(\rho_K)_{K\in \mesh}$ are set to $(\rho^{n-1}_K)_{K\in \mesh}$ by the first block of equations, so the values of the density at the faces (which is computed over each dual cell $D_\edge$ as a linear combination of the density in the neighbour cells) are also known, and positive.
The generalized Stokes problem \eqref{eq:gen_stokes} thus also has one solution and only one, and there exists one and only one point of ${\mathcal O}$ such that $F(\rho,\bfu,p,0)= 0_V$.
Since the Jacobian matrix at this point is invertible (since $\text{Id}_{\xR^{N_\mesh\times N_\mesh}}$ and $S(\bfu,p)$ are invertible), this implies that the topological degree of $F(\rho,\bfu,p,0)$ with respect to ${\mathcal O}$ and $0_V$ is non-zero.
Therefore, by Theorem \ref{thrm:degree}, there exists at least one solution $(\rho,\bfu,p)$ to the equation $F(\rho,\bfu,p,1)=0$, \emph{i.e.} to the scheme \eqref{eq:impl_scheme}.
\end{proof}
%
%
\section{Proof of the convergence result}\label{sec:proof}

We begin by proving discrete analogues to the classical estimates satisfied by the exact solutions of problem \eqref{eq:pb}  (Section \ref{subsec:uniform_estimates}).
These are uniform estimates in the sense that they only depend on the parameters of the problem and on the upper bound $\theta_0$ on the discretization regularity.
Then, in Sections \ref{subsec:compactness_u} and \ref{subsec:compactness_rho}, we prove strong compactness for the discrete velocity and weak compactness for the discrete density.
We then conclude the proof by passing successively to the limit in the mass and momentum balance equations (Sections \ref{subsec:convergence_rho} and \ref{subsec:convergence_u} respectively).
In this last step, we actually need strong convergence for the sequence of discrete densities, which is proved in Section \ref{subsec:strong_convergence_rho}.
%
%
\subsection{Uniform estimates} \label{subsec:uniform_estimates}
We define $E_\disc(\Omega)=\lbrace \bfu \in \xH_{\edges,0}(\Omega)^d,~ \dive_\mesh \bfu = 0 \rbrace$ and we endow $E_\disc(\Omega)$ with the norm $\norm{.}_\brok$.
For $q$ in $[1 ,\infty)$, we define on the space of discrete velocity functions (see expression \eqref{eq:def_rho_u_e_p}) the following 
$\xL^q((0,T);E_\disc(\Omega))$ norm:
\[
\norm{\bfu}_{\xL^q((0,T);E_\disc(\Omega))}^q=\sum_{n=1}^N \delta t\ \norm{\bfu^n}^q_\brok.
\]

\begin{prop}[Uniform estimates for discrete solutions] \label{prop:estimates} \ \\
Let $\disc$ be a given staggered discretization such that $\theta_\disc \leq \theta_0$ for some positive real number $\theta_0$ and let $\delta t$ be a given time step.
Let $(\rho,\bfu,p)$ be the corresponding discrete solution given by the scheme \eqref{eq:impl_scheme}, as defined in \eqref{eq:def_rho_u_e_p}.
Then the following estimates hold:
\begin{eqnarray} \label{eq:estimate1}
& (i) & \qquad
\rhomin \leq \rho(\bfx,t) \leq \rhomax, \quad \text{for all $t\geq 0$ and for a.e. $\bfx$ in $\Omega$},
\\ \label{eq:estimate2}
& (ii) & \qquad \sum_{n=1}^N \delta t  \sum_{\substack{\edge \in \edgesint \\ \edge=K|L}} |\edge| (\rho_L^n -\rho_K^n)^2  | \bfu_\edge^n\cdot \bfn_{K,\edge}|\leq C_0,
\\ \label{eq:estimate3}
& (iii) & \qquad
\norm{\bfu }_{\xL^\infty((0,T);\xL^2(\Omega))}
+ \norm{\bfu }_{\xL^2((0,T);E_\disc(\Omega))} \leq C_1. 
\end{eqnarray}
where $\rhomin$ and $\rhomax$ stand for the minimum and maximum values of the initial density $\rho_0$, as defined in assumption \eqref{eq:H_rho} of Section \ref{sec:cont}, and $C_0$, and $C_1$ are  positive real numbers depending only on $T$, $\Omega$, $d$, $\rho_0$, $\bfu_0$ and $\theta_0$.
\end{prop}

\begin{proof}
Estimate \eqref{eq:estimate1} is a direct consequence of equation \eqref{eq:imp_maxprinc} in Proposition \ref{prop:bvfaible}.
Inequality \eqref{eq:estimate2} is obtained from \eqref{eq:rho_bv} after summing over $n$.
In the same way, the estimates on $\norm{\bfu }_{\xL^\infty((0,T);\xL^2(\Omega))}$ and $\norm{\bfu }_{\xL^2((0,T);E_\disc(\Omega))}$ are obtained through \eqref{eq:imp_stab} after summing over $n$ and using the fact that $\rho$ is positive and bounded from below by $\rhomin>0$.
\end{proof}

%
%
\medskip

\subsection{Compactness of the sequence of discrete velocities} \label{subsec:compactness_u}
In this section and in the following one,  $(\rho\m,\bfu\m)_{m\in \xN}$ are the discrete densities and velocities solutions of the scheme \eqref{eq:inicond}-\eqref{eq:impl_scheme} associated with $(\disc\m,\delta t\m)_{m\in\xN}$ a regular sequence of staggered discretizations and time steps. In this section, we prove the following compactness result on the sequence of velocities $(\bfu\m)_{m\in \xN}$:

\begin{prop} \label{prop:compactness_u}
Under the assumptions of Theorem \ref{thrm:convergence}, there exists
\[
\bar \bfu \in \xL^{\infty}((0,T);\xL^2(\Omega)^d)\cap \xL^2((0,T);\xH_0^1(\Omega)^d), \quad \text{with }\dive (\bar \bfu) =0,
\]
and a subsequence of $(\bfu\m)_{m\in\xN}$, still denoted $(\bfu\m)_{m\in\xN}$, which converges to $\bar \bfu$ strongly in $\xL^q((0,T);\xL^2(\Omega)^d)$ for all $q\in[1,\infty)$.
\end{prop}

The proof of Proposition \ref{prop:compactness_u} relies on estimates of the time translations of the velocity. The following lemma provides an estimate on the $\xL^2(\xL^2)$-norm of the time translations of the discrete velocity $\bfu$ for a given discretization $\disc$, which leads to strong compactness in $\xL^2((0,T);\xL^2(\Omega)^d)$  through Kolomogorov's compactness Theorem stated in Appendix \ref{sec:kolmogorov}.

\begin{lem}[Time translations of the discrete velocity] \label{lmm:time-trans} \ \\
Let $\disc$ be a given staggered discretization such that $\theta_\disc \leq \theta_0$ for some positive real number $\theta_0$ and let $\delta t$ be a given time step satisfying $\delta t \leq 1$. Let $(\rho,\bfu,p)$ be the corresponding discrete solution given by the scheme \eqref{eq:impl_scheme}, as defined in \eqref{eq:def_rho_u_e_p}. The discrete velocity $\bfu$ satisfies:
\begin{equation} 
\label{eq:time-trans}
 \int_{\tau}^{T}\norm{\bfu(.,t)-\bfu(.,t-\tau)}^2_{\xL^2(\Omega)^d}\dt \leq C_2 \left ( \tau^{\frac 14} + \delta t^{\frac 14} + h_\disc^{\frac 12} \right ), \qquad \text{for all} \quad \tau \in(0,\min(T,1)),
\end{equation}
where $C_2$ is a positive constant, depending only on $T$, $\Omega$, $d$, $\rho_0$, $\bfu_0$ and $\theta_0$.
\end{lem}

\begin{proof}
Let $\tau$ be a given real number in $(0,\min(T,1))$. For every $t\in(\tau,T)$, let $n,k$ be the two integers defined by $n=\Ent{\frac{t}{\delta t}}$ and $n-k=\Ent{\frac{t-\tau}{\delta t}}$. It is easily seen that $0\leq n-k \leq n \leq N$ and $k \delta t \leq \tau+\delta t$. We have for all $\edge$ in $\edgesint$: 
\[\begin{aligned}
\rho^{n-k}_\Ds ( \bfu^n_\edge- \bfu_\edge^{n-k}) 
&
= (\rho^n_\Ds \bfu^n_\edge-\rho^{n-k}_\Ds \bfu_\edge^{n-k}) -   \bfu^n_\edge(\rho^n_\Ds -\rho^{n-k}_\Ds)
\\[1ex]
&
= \sum_{p=n-k+1}^{n}(\rho^{p}_\Ds \bfu^{p}_\edge-\rho^{p-1}_\Ds \bfu_\edge^{p-1}) -  \bfu^n_\edge \sum_{p=n-k+1}^{n}(\rho^p_\Ds -\rho^{p-1}_\Ds)
\\[1ex]
&
=  \sum_{p=n-k+1}^{n} \delta t \Big [ -{|D_\edge|}^{-1} \sum_{\edged \in\edgesd(D_\edge)} \fluxd^p \bfu^p_\edged
+ (\lapi \bfu)_\edge^p
- (\gradi p)_\edge^p \Big ]
\\[1ex]
& \hspace{5cm}
+ \bfu^n_\edge \sum_{p=n-k+1}^{n} \delta t {|D_\edge|}^{-1}\sum_{\edged \in\edgesd(D_\edge)} \fluxd^p,
\end{aligned}
\]
by the discrete momentum balance equations \eqref{eq:sch_mom} and the discrete mass balance over the dual cells \eqref{eq:mass_D}.
Let $\bfv(.,t)$ be a time-dependent element of $E_\disc(\Omega)$ which we denote $\bfv(\bfx,t)=\sum_{\edge\in\edges}\bfv_\edge(t)\chi_\Ds(\bfx)$ and denote $\tilde{\rho}(\bfx,t-\tau)=\sum_{\edge\in\edges}\rho_\Ds^{n-k}\chi_\Ds(\bfx)$. Taking the scalar product of the above equation with $|D_\edge|\,\bfv_\edge(t)$ and summing over $\edge \in \edges$ (recall that $\bfv_\edge(t)=0$ for $\edge\in\edgesext$), we obtain: 
\begin{multline*}
\int_{\Omega}\tilde{\rho}(.,t-\tau)(\bfu(.,t)-\bfu(.,t-\tau)) \cdot \bfv(.,t) =\sum_{\edge \in\edgesint}  |D_\edge|\, \rho^{n-k}_\Ds (\bfu^n_\edge - \bfu_\edge^{n-k}) \cdot \bfv_\edge(t) 
\\ = T_1(t) +T_2(t) + T_3(t) + T_4(t),
\end{multline*}
where:
\[
 \begin{array}{ll}
  \ds T_1(t) = -\sum_{p=n-k+1}^{n} \delta t \sum_{\edge \in\edgesint} |D_\edge|\, (\gradi p)_\edge^p \cdot \bfv_\edge(t),
  & \quad \ds T_2(t) = \sum_{p=n-k+1}^{n} \delta t \sum_{\edge \in\edgesint} |D_\edge|\, (\lapi \bfu)_\edge^p \cdot \bfv_\edge(t),\\
  \ds T_3(t) = - \sum_{p=n-k+1}^{n} \delta t \sum_{\edge \in\edgesint}\bfv_\edge(t) \cdot\sum_{\edged \in\edgesd(D_\edge)} \fluxd^p \bfu^p_\edged, 
  & \quad \ds T_4(t) =  \sum_{p=n-k+1}^{n} \delta t \sum_{\edge \in\edgesint}(\bfu^n_\edge \cdot \bfv_\edge(t)) \sum_{\edged \in\edgesd(D_\edge)} \fluxd^p.
 \end{array}
\]

Since $\dive_\mesh \bfv(.,t)=0$, by the discrete gradient-divergence duality (see Lemma \ref{lmm:dual_divgrad}), one has $T_1(t)=0$.

The second term is controlled as follows:
\[
\begin{aligned}
|T_2(t)| 
= \Big |  \sum_{p=n-k+1}^{n} \delta t \sum_{K\in\mesh} \int_K \gradi \tilde \bfu^p(\bfx) : \gradi \tilde \bfv(t)(\bfx) \dx \Big | 
&\leq  \sum_{p=n-k+1}^{n} \delta t \norm{\bfu^p}_{\brok}\norm{\bfv(.,t)}_{\brok}
\\
&\leq \Big ( \sum_{p=n-k+1}^{n} \delta t \norm{\bfu^p}_{\brok}^2 \Big)^{\frac 12} ( k \delta t  )^{\frac 12} \norm{\bfv(.,t)}_{\brok},
\end{aligned} 
\]
by the Cauchy-Schwarz inequality. Hence $|T_2(t)|\leq C_1 (\tau + \delta t)^{\frac 12} \norm{\bfv(.,t)}_{\brok}$ by estimate \eqref{eq:estimate3}, which gives $|T_2(t)|\leq C_1 (\tau^{\frac 14}+\delta t^{\frac 14}) \norm{\bfv(.,t)}_{\brok}$ since $\tau,\,\delta t<1$.

For the third term, we remark that
$
T_3(t)= -\sum_{p=n-k+1}^{n} \delta t \, \mcal{Q}^{\rm mom}_\edges(\rho^p,\bfu^p,\bfu^p,\bfv(.,t))  = T_{3,1}(t)+T_{3,2}(t)
$
with:
\[
 \begin{aligned}
T_{3,1}(t) &= -\sum_{p=n-k+1}^{n} \delta t \, \big (\mcal{Q}^{\rm mom}_\edges(\rho^p,\bfu^p,\bfu^p,\bfv(.,t))  - \mcal{Q}^{\rm mom}_\mesh(\rho^p,\bfu^p,\bfu^p,\bfv(.,t)) \big ), \\
T_{3,2}(t) &= -\sum_{p=n-k+1}^{n} \delta t \, \mcal{Q}^{\rm mom}_\mesh(\rho^p,\bfu^p,\bfu^p,\bfv(.,t)).
 \end{aligned}
\]
By Lemma  \ref{lmm:convconv}, and estimates \eqref{eq:estimate1} and \eqref{eq:estimate3}, the first term is controlled as follows:
\[
|T_{3,1}(t)| \leq  C  \sum_{p=n-k+1}^{n} \delta t \ \norm{\rho^p}_{\xL^{\infty}(\Omega)}\, \norm{\bfu^p}_\brok^2\, \norm{\bfv(.,t)}_\brok h_\disc^{1-\eps}  \leq C \norm{\rho_0}_{\xL^{\infty}(\Omega)} C_1^2\norm{\bfv(.,t)}_\brok h_\disc^{1-\eps}, 
\]
where $C$ is a positive constant depending only on $\Omega$, $d$ and $\theta_0$. One may choose $\eps=\frac 12$ for both dimensions $d=2$ and $d=3$. By Lemma \ref{lmm:estimates_Bbar}, the second term $T_{3,2}(t)$ is controlled as follows:
\[
\begin{aligned}
|T_{3,2}(t)| 
  &\leq \sum_{p=n-k+1}^{n} \delta t  \norm{\rho^p}_{\xL^{\infty}(\Omega)^d}\, \norm{\bfu^p}_{\xL^4(\Omega)}^2\, \norm{\bfv(.,t)}_\brok 
  \\
  &\leq  \norm{\rho_0}_{\xL^{\infty}(\Omega)} \norm{\bfv(.,t)}_\brok (k\delta t)^{\frac 14} 
         \Big ( \sum_{p=n-k+1}^{n} \delta t  \norm{\bfu^p}_{\xL^4(\Omega)}^{\frac 83} \Big )^{\frac 34}, 
\end{aligned}
\]
thanks to  H\"older's inequality with powers $4$ and $4/3$. We then remark that 
\[
\norm{\bfu^p}_{\xL^4(\Omega)}\leq \norm{\bfu^p}_{\xL^2(\Omega)}^{\frac 14} \norm{\bfu^p}_{\xL^6(\Omega)}^{\frac 34} \leq C_1^{\frac 14}\norm{\bfu^p}_{\xL^6(\Omega)}^{\frac 34} \leq  C_1^{\frac 14} C(d,\Omega,\theta_0) \norm{\bfu^p}_{\brok}^{\frac 34}
\]
by the Cauchy-Schwarz inequality, estimate \eqref{eq:estimate3} and the discrete Sobolev embedding given in Lemmas \ref{lmm:injsobolev_br} and \ref{lmm:H1ns}. Injecting this in the above inequality on $|T_{3,2}(t)|$, and invoking once again estimate \eqref{eq:estimate3} one gets:
\[
 \begin{aligned}
  |T_{3,2}(t)| 
  &\leq  C_1^{\frac 12}C(d,\Omega,\theta_0)^2 \norm{\rho_0}_{\xL^{\infty}(\Omega)} \norm{\bfv(.,t)}_\brok (k\delta t)^{\frac 14} 
         \Big ( \sum_{p=n-k+1}^{n} \delta t  \norm{\bfu^p}_{\brok}^{2} \Big )^{\frac 34} 
  \\ &\leq C_1^{2}C(d,\Omega,\theta_0)^2 \norm{\rho_0}_{\xL^{\infty}(\Omega)}(\tau+\delta t)^{\frac 14} \norm{\bfv(.,t)}_\brok.
 \end{aligned}
\]
Hence, $|T_{3}(t)|\leq C \big (\tau^{\frac 14}+\delta t^{\frac 14}  + h_\disc^{\frac 12} \big )\norm{\bfv(.,t)}_\brok$ for some constant $C$ depending only on the initial data and on $\Omega$, $T$, $d$ and $\theta_0$.

The fourth term can be estimated as follows. We first remark that 
\[
T_4(t)=  \sum_{p=n-k+1}^{n} \delta t \mcal{Q}^{\rm mass}_\edges(\rho^p,\bfu^p,\bfu^n,\bfv(.,t)).
\]
where, by Lemma \ref{lmm:estimates_Qmass}, on has
\[
\begin{aligned}
|\mcal{Q}^{\rm mass}_\edges(\rho^p,\bfu^p,\bfu^n,\bfv(.,t))| 
&\leq 
C\norm{\rho^p}_{\xL^{\infty}(\Omega)} \norm{\bfu^p}_{\xL^6(\Omega)^d} \left ( \norm{\bfu^n}_{\xL^3(\Omega)^d} \norm{\bfv(.,t)}_{\brok} + \norm{\bfv(.,t)}_{\xL^3(\Omega)^d} \norm{\bfu^n}_{\brok}\right ).
\end{aligned}
\]
Using estimate \eqref{eq:estimate1} on the density, the fact that $\bfu^n=\bfu(.,t)$, and  the now familiar continuous embeddings $\xH_{\edges,0}(\Omega)^d \subset \xL^6(\Omega)^d \subset \xL^3(\Omega)^d$, one gets:
\[
\begin{aligned}
 |T_4(t)| & \leq C' \sum_{p=n-k+1}^{n} \delta t \norm{\bfu^p}_{\brok} \norm{\bfu(.,t)}_{\brok} \norm{\bfv(.,t)}_{\brok}
	  \\ &\leq C' (k\delta t)^{\frac 12} \Big ( \sum_{p=n-k+1}^{n} \delta t  \norm{\bfu^p}_{\brok}^{2} \Big )^{\frac 12}  \norm{\bfu(.,t)}_{\brok} \norm{\bfv(.,t)}_{\brok}.
\end{aligned}
\]

Using once again estimate \eqref{eq:estimate3}, and the fact that $\tau,\,\delta t<1$ yields  $|T_4(t)| \leq C (\tau^{\frac 14}+\delta t^{\frac 14}) \norm{\bfu(.,t)}_{\brok} \norm{\bfv(.,t)}_{\brok}$ for some constant $C$ depending only on the initial data and on $\Omega$, $T$, $d$ and $\theta_0$.

Collecting the estimates on $T_1(t)$, $T_2(t)$, $T_3(t)$ and $T_4(t)$, one gets:
\[
 \int_{\Omega}\tilde{\rho}(.,t-\tau)(\bfu(.,t)-\bfu(.,t-\tau)) \cdot \bfv(.,t)\leq C \big (\tau^{\frac 14}+\delta t^{\frac 14} + h_\disc^{\frac 12} \big ) \big ( 1+  \norm{\bfv(.,t)}_{\brok}^2 + \norm{\bfu(.,t)}_{\brok}^2 \big ).
\]
Finally, selecting $\bfv(.,t):=\bfu(.,t)-\bfu(.,t-\tau)$ and recalling that $\widetilde{\rho}(.,t)\geq\rho_{\rm min}>0$ almost everywhere for all $t\in(0,T)$ yields:
\[
 \norm{\bfu(.,t)-\bfu(.,t-\tau)}^2_{\xL^2(\Omega)^d} \leq C \big (\tau^{\frac 14}+\delta t^{\frac 14} + h_\disc^{\frac 12} \big )\big ( 1+  2\norm{\bfu(.,t-\tau)}_{\brok}^2 + 3\norm{\bfu(.,t)}_{\brok}^2 \big ).
\]
Integrating for $t$ in $(\tau,T)$ and invoking once again estimate \eqref{eq:estimate3} yields the expected result \eqref{eq:time-trans}.
\end{proof}

We may now give the proof of Proposition  \ref{prop:compactness_u}.
\begin{proof}[Proof of Proposition \ref{prop:compactness_u}]
We proceed in four steps.
\begin{list}{-}{\itemsep=2ex \topsep=1ex \leftmargin=0.5cm \labelwidth=0.7cm \labelsep=0.3cm \itemindent=0.cm}
\item \textbf{Step 1: compactness in $\xL^2(\xL^2)$.}
The first step consists in extracting a strongly converging subsequence from $(\bfu\m)_{m\in\xN}$ thanks to Kolmogorov's compactness Theorem.
This result is recalled in appendix \ref{sec:kolmogorov} (Theorem \ref{thrm:kolmogorov}).
In our setting, the Banach space $B$ of Theorem \ref{thrm:kolmogorov} is $\xL^2(\Omega)^d$, $p=2$, and the subset $A$ is defined by $\ds A=\cup_{m\in\xN}\lbrace \bfu\m\rbrace$. Any $\bfu\m\in A$ satisfies $\norm{\bfu\m}_{\xL^2((0,T);\xL^2(\Omega)^d)}\leq C\norm{\bfu\m}_{\xL^2((0,T);E_{\disc\m})}\leq CC_1$ by the Poincar\'e inequality of Lemma \eqref{lmm:poincare_brok} and estimate \eqref{eq:estimate3} so that $A\subset\xL^2((0,T);\xL^2(\Omega)^d)$.
We now check the three assumptions $(h1)$-$(h3)$ of the Theorem.
\begin{list}{-}{\itemsep=0.5ex \topsep=0.5ex \leftmargin=1.cm \labelwidth=0.7cm \labelsep=0.3cm \itemindent=0.cm}
\item $(h1)$ -- The operator $P$ is defined by extending any function $\bfu\in A$ by zero outside the interval $(0,T)$. Clearly, one has $\norm{P\bfu}_{\xL^2(\xR;\xL^2(\Omega)^d)}= \norm{\bfu}_{\xL^2((0,T);\xL^2(\Omega)^d)}\leq C_1$ for all $\bfu\in A$.
\item $(h2)$ -- For all $\phi \in \mcal{C}_c^\infty(\xR,\xR)$, and $\bfu\m\in A$, the quantity $\int_{\xR}(P\bfu\m)\phi \dt$ is an element of $\xH_{\edges\m,0}(\Omega)^d$ with: 
\[\hspace{1.8cm}
 \norm{\int_\xR(P\bfu\m)\phi \dt}_{\edges\m,{\rm b}}\leq\int_0^T \norm{\bfu\m(.,t)}_{\edges\m,{\rm b}} |\phi(t)| \dt \leq C_1 \norm{\phi}_{\xL^2((0,T);\xR)},
\]
by the Cauchy-Scharz inequality and estimate \eqref{eq:estimate3}. Hence, by the discrete Rellich Theorem \ref{lmm:compact_br}, the family $\lbrace \int_{\xR}(P\bfu)\phi \dt, \, \bfu\in A\rbrace$ is relatively compact in $\xL^2(\Omega)^d$.
\item $(h3)$ -- It remains to prove that $\parallel P\bfu-P\bfu(\square-\tau) \parallel_{\xL^2(\xR,\xL^2(\Omega))^d} \to 0$ as $\tau\to 0^+$ uniformly with respect to $\bfu\in A$, where the square stands for the time variable $t$. For all $\bfu \in A$ and $\tau\in(0,T)$ we have:
\begin{multline*}
 \hspace{2cm} \parallel P\bfu-P\bfu(\square-\tau) \parallel^2_{\xL^2(\xR,\xL^2(\Omega)^d)} 
 \\ 
 = \int_0^\tau \norm{\bfu(.,t)}_{\xL^2(\Omega)^d}^2\dt + \int_\tau^T \norm{\bfu(.,t)-\bfu(.,t-\tau)}_{\xL^2(\Omega)^d}^2\dt + \int^T_{T-\tau} \norm{\bfu(.,t)}_{\xL^2(\Omega)^d}^2\dt.
\end{multline*}
%
Thanks to the $\xL^\infty(\xL^2)$-estimate \eqref{eq:estimate3}, the first and third terms are each controlled by $C_1^2\,\tau$, and the second term is controlled thanks to Lemma \ref{lmm:time-trans} on the time translations of the velocity. Let $\eps >0$ be a small real number (smaller than $1$). There exists $M\in\xN$ such that, $(\delta t\m)^{\frac 14} \leq \eps$ and $(h\m)^{\frac12}\leq \eps$ for all $m\geq M$. Let $\bar \tau$ be a real number in $(0,\min(T,1))$ such that $\bar \tau^{\frac 14}\leq \eps$. Thanks to Lemma \ref{lmm:time-trans}, we have, for all $m\geq M$:
\[
\parallel P\bfu\m-P\bfu\m(\square-\tau) \parallel^2_{\xL^2(\xR,\xL^2(\Omega)^d)}  \leq  (3\,C_2\,+2\,C_1^2)\, \eps, \qquad  \forall \tau\in (0,\bar\tau).
\]
Now, for a fixed $m < M$, we have, by a classical result that $\parallel P\bfu\m-P\bfu\m(\square-\tau) \parallel^2_{\xL^2(\xR,\xL^2(\Omega)^d)}$ tends to zero as $\tau\to 0^+$. Hence, for all $m< M$, there exists $\tau\m>0$ such that
\[
 \parallel P\bfu\m-P\bfu\m(\square-\tau) \parallel^2_{\xL^2(\xR,\xL^2(\Omega)^d)} \leq \eps, \qquad  \forall \tau\in (0,\tau\m).
\]
Defining $\tilde \tau = \min (\bar \tau, \min \lbrace \tau\m, \, 0\leq m < M \rbrace )$ and $C=\max(1,3\,C_2\,+2\,C_1^2)$, two numbers which are independent of $\bfu\in A$, we have proven that:
\[
  0 < \tau \leq \tilde \tau \qquad  \Longrightarrow  \qquad \parallel P\bfu-P\bfu(\square-\tau) \parallel^2_{\xL^2(\xR,\xL^2(\Omega)^d)} \leq C \eps, \quad  \forall \bfu \in A,
\]
which expresses that $\parallel P\bfu-P\bfu(\square-\tau) \parallel_{\xL^2(\xR,\xL^2(\Omega))^d} \to 0$ as $\tau\to 0^+$ uniformly for $\bfu\in A$.
\end{list}
Hence, Kolmogorov's Theorem \ref{thrm:kolmogorov} applies and there exists $\bar \bfu \in \xL^2((0,T);\xL^2(\Omega)^d)$ and a subsequence of $(\bfu\m)_{m\in\xN}$, still denoted $(\bfu\m)_{m\in\xN}$, which converges towards $\bar \bfu$ in $\xL^2((0,T);\xL^2(\Omega)^d)$ as m tends to infinity.
\item \textbf{Step 2: convergence in $\xL^q(\xL^2)$.}
As $\bfu\m\to\bar u$ in $\xL^2((0,T);\xL^2(\Omega)^d)$ we clearly have convergence in $\xL^q((0,T);\xL^2(\Omega)^d)$ for $1\leq q \leq 2$. Thanks to \eqref{eq:estimate3}, we have $\norm{\bfu\m}_{\xL^\infty((0,T);L^2(\Omega)^d)} \leq C$, for all $m\in\xN$.
Hence, there exists $\hat \bfu \in \xL^\infty((0,T);\xL^2(\Omega)^d)$ and a subsequence $(\bfu^{\phi(m)})_{m\in\xN}$ such that $\bfu^{\phi(m)} \weak^* \hat \bfu$ in $\xL^\infty((0,T);\xL^2(\Omega)^d)$.
As $\bfu^{\phi(m)} \to \bar \bfu$ in $\xL^1((0,T);\xL^2(\Omega)^d)$, the uniqueness of the limit in the sense of distributions implies that $\bar \bfu=\hat \bfu$, which means that $\bar \bfu \in \xL^\infty((0,T);\xL^2(\Omega)^d)$.
Now, using a classical interpolation result on $\xL^p(0,T)$ spaces, we have for all $q\in [1,\infty)$
\[
\norm{\bar \bfu - \bfu^{(m)}}_{\xL^q((0,T);\xL^2(\Omega)^d)}
\leq \norm{\bar \bfu - \bfu^{(m)}}_{\xL^1((0,T);\xL^2(\Omega)^d)}^{\frac 1 q} \norm{\bar \bfu - \bfu^{(m)}}_{\xL^\infty((0,T);\xL^2(\Omega)^d)}^{1-\frac 1 q},
\]
which implies that $\bfu^{(m)}$ converges towards $\bar \bfu$ in $\xL^q((0,T);L^2(\Omega)^d)$  for all $q\in [1,\infty)$ as $m$ tends to infinity.

\item \textbf{Step 3: regularity of the limit.}
According to \eqref{eq:estimate3}, $\norm{\bfu\m}_{\xL^2((0,T);E_{\disc\m}(\Omega))} \leq C_1$, for all $m\in\xN$.
Integrating equation \eqref{eq:controle_sauts_br} of Lemma \ref{lmm:controle_sauts_br} along the time variable, we get
\[
\norm{{\bfu^\sharp}\m(.+\eta,\square)-{\bfu^\sharp}\m(.,\square)}_{\xL^2(\xR^d\times\xR)^d}^2 \leq
C\, C_1^2\, |\eta|\ (|\eta|+h\m),
\quad \forall \eta \in \xR^d,~ \forall m\in\xN,
\]
where the dot stands for the space variable $\bfx$, the square stands for the time variable $t$, $C$ is independent of $m$, and ${\bfu^\sharp}\m$ is the extension of $\bfu\m$ to $\xR^d\times\xR$ by setting ${\bfu^\sharp}\m=0$ on $(\xR^d\times\xR) \setminus (\Omega\times(0,T))$.
Let $\bar \bfu^\sharp$ be the extension by zero of $\bar \bfu$ outside $\Omega\times(0,T)$.
Since ${\bfu^\sharp}\m \to \bar\bfu^\sharp$ in $L^2(\xR^d\times\xR)^d$ and $h\m\to0$ as $m\to\infty$, passing to the limit in the above inequality yields
\[
\norm{\bar\bfu^\sharp(.+\eta,\square)-\bar\bfu^\sharp(.,\square)}_{\xL^2(\xR^d\times\xR)^d}^2 \leq C' |\eta|^2, \quad \forall \eta \in \xR^d.
\]
This implies that $\gradi \bar \bfu \in \xL^2( \Omega\times(0,T))$ and that $\bar \bfu=0$ on $\dv\Omega$ since $\bar\bfu^\sharp=0$ outside $\Omega\times(0,T)$.
Hence, $\bar \bfu \in \xL^2((0,T);H_0^1(\Omega)^d)$.

\item  \textbf{Step 4: the limit is solenoidal.}
It remains to prove that $\dive \, \bar \bfu (\bfx,t)=0$ for \textit{a.e.} $(\bfx,t)$ in $\Omega\times(0,T)$.
We have $\bfu\m \to \bar \bfu$ in $\xL^2(\Omega\times(0,T))^d$.
By the partial reciprocal of the dominated convergence theorem, there exists a subsequence still denoted $(\bfu\m)_{m\in\xN}$ such that for \textit{a.e.} $t>0$, $\bfu\m(.,t) \to \bar \bfu(.,t)$ in $\xL^2(\Omega)^d$.
Now let $\phi\in \xC^\infty_c( \Omega\times(0,T))$ and for $t>0$, let $\phi_K(t)=|K|^{-1}\int_K\phi(\bfx,t) \dx$ for every $K$ in $\mesh\m$.
For a fixed $m\in\xN$, denote by $\phi\m(t)$ the function of $H_{\mesh\m}(\Omega)$ defined by  $\phi\m(t)(\bfx)=\phi_K(t)$ for all $\bfx$ in $K$, $K$ in $\mesh\m$.
Since $\phi$ is smooth, we easily prove that for \textit{a.e.} $t>0$, $\phi\m(t) \to \phi(.,t)$ strongly in $\xL^2(\Omega)^d$ as $m\to+\infty$ and $\norm{\gradi_{\edges\m} \phi\m(t)}_{\xL^2(\Omega)^d}\leq C \norm{\gradi \phi(.,t)}_{\xL^\infty(\Omega)^d} $ for all $m$ in $\xN$, where $C$ depends only on $d$, $\Omega$ and $\theta_0$.
By Lemma \ref{lmm:compact_fv}, for \textit{a.e.} $t>0$, $\gradi_{\edges\m} \phi\m(t)$ converges weakly in $\xL^2(\Omega)^d$ towards $\gradi \phi(.,t)$ as $m$ tends to infinity. Finally, we may write for \textit{a.e.} $t>0$:
\[
\begin{aligned}
0 &= \int_\Omega  \phi\m(t)(\bfx) \dive_{\mesh\m} \bfu\m(\bfx,t) \dx \\
  &= - \int_\Omega \bfu\m(\bfx,t) \cdot \gradi_{\edges\m} \phi\m(t)(\bfx) \dx \\
  &\to - \int_\Omega \bar \bfu(\bfx,t) \cdot \gradi \phi(\bfx,t) \dx \quad \text{as $m\to+\infty$.}
\end{aligned}
\]
Since $\phi$ is arbitrarily chosen, this proves that, for \textit{a.e.} $t>0$, $\dive\, \bar \bfu(.,t)=0$ as a function of $\xL^2(\Omega)^d$ and therefore $\dive \bar \bfu=0$ \textit{a.e.} in $\in \Omega\times (0,T)$.
\end{list}
\end{proof}

%
%
\medskip
\subsection{Compactness of the sequence of discrete densities} \label{subsec:compactness_rho}

\begin{prop} \label{prop:compactness_rho}
Under the assumptions of Theorem \ref{thrm:convergence}, there exists $\bar \rho$ in $\xL^\infty(\Omega \times (0,T))$ with $\rhomin \leq \bar \rho \leq \rhomax$ a.e. in $\Omega \times(0,T)$, and a subsequence of $(\rho\m)_{m\in\xN}$, still denoted $(\rho\m)_{m\in\xN}$, which converges towards $\bar \rho$  in  $\xL^\infty(\Omega\times (0,T))$ weak-$\star$.
\end{prop}

\begin{proof}
By \eqref{eq:estimate1}, we have $\norm{\rho\m}_{\xL^\infty(\Omega\times(0,T))} \leq \rhomax$ for all $m$ in $\xN$, which implies the weak-star convergence of a subsequence of $(\rho\m)_{m\in\xN}$, still denoted $(\rho\m)_{m\in\xN}$, towards some function $\bar \rho$ in $\xL^\infty(\Omega\times(0,T))$, \ie\ for all $\phi$ in $\xL^1(\Omega \times (0,T))$:
\begin{equation} \label{eq:rho_weakconv}
\lim_{m\to\infty} \int_0^T \int_\Omega \rho\m(\bfx,t)\, \phi(\bfx,t) \dx \dt
= \int_0^T \int_\Omega \bar \rho(\bfx,t)\, \phi(\bfx,t)\dx\dt. 
\end{equation}
Furthermore, an easy consequence of \eqref{eq:estimate1} and \eqref{eq:rho_weakconv} is the non-negativity of the integrals
\[
\int_0^T \int_\Omega (\bar \rho(\bfx,t)-\rhomin)\, \mathcal{X}_A(\bfx,t)\dx\dt \quad \mbox{and} \quad
\int_0^T \int_\Omega (\rhomax -\bar \rho(\bfx,t))\, \mathcal{X}_A(\bfx,t)\dx\dt
\]
for any borelian set $A$ of $ \Omega\times(0,T)$, which is equivalent to $\rhomin \leq \bar \rho (\bfx,t) \leq \rhomax $ a.e. in $\Omega \times (0,T)$.
\end{proof}
%
%
\medskip
\subsection{Passing to the limit in the mass balance equation} \label{subsec:convergence_rho}

\begin{prop} \label{prop:convergence_rho}
Under the assumptions of Theorem \ref{thrm:convergence}, the weak star limit in $\xL^\infty(\Omega\times(0,T))$  of $(\rho\m)_{m\in\xN}$, $\bar \rho$, and the strong limit of $(\bfu\m)_{m\in\xN}$ in $\xL^2( \Omega\times(0,T))^d$, $\bar \bfu$, satisfy:
\[
- \int_0^T  \int_\Omega \bar\rho(\bfx,t) 
\ \Bigl(\dv_t \phi(\bfx,t) + \bar\bfu(\bfx,t) \cdot  \gradi\phi(\bfx,t) \Bigr) \dx \dt = \int_\Omega \rho_0(\bfx) \phi(\bfx,0) \dx,
\]
for all $\phi$ in $ \xC^{\infty}_c(\Omega \times [0,T))$.
\end{prop}

Before proving this proposition, we first state the following lemma, the proof of which is easy and relies on Taylor's inequalities for functions with at least $C^2$-regularity.
\begin{lem} \label{lmm:test_mass}
Let $\disc$ be a given staggered discretization such that $\theta_\disc \leq \theta_0$ for some positive real number $\theta_0$ and let $\delta t$ be a given time step.
Let $\phi\in \xC^\infty_c(\Omega \times [0,T))$ and, for $n$ in $\lbrace 0,..,N\rbrace$, let us define:
\begin{list}{-}{\itemsep=0.5ex \topsep=0.5ex \leftmargin=1.5cm \labelwidth=0.7cm \labelsep=0.3cm \itemindent=0.cm}
\item $\phi_K^n=\phi(\bfx_K,t_n)$ for all $K$ in $\mesh$, with $\bfx_K$ the mass center of $K$.
\item $\displaystyle \phi_\edge^n=|\edge|^{-1}\int_\edge \phi(\bfx,t_n) \dedge(\bfx)$ for all $\edge \in \edges$.
\end{list}
We denote by $\phi^0_\mesh$ the function $\phi^0_\mesh(\bfx) = \sum_{K \in \mesh} \phi_K^0\,  \mathcal{X}_K(\bfx)$, and we define the discrete time derivative and gradient of $\phi$ by:
\[\begin{aligned} & 
\partd_t \phi_\mesh(\bfx,t) = \sum_{n=1}^N \sum_{K \in \mesh}
\frac 1 {\delta t} \left ( \phi_K^n - \phi_K^{n-1}  \right )\,  \mathcal{X}_K(\bfx) \mathcal{X}_{(n-1,n]}(t),
\\ & 
\gradi \phi_\mesh(\bfx,t) = \sum_{n=1}^N \sum_{K \in \mesh} (\gradi \phi)^n_K \, \mathcal{X}_K(\bfx) \mathcal{X}_{(n-1,n]}(t), 
\\ &
\mbox{with }
(\gradi \phi)^n_K = \frac 1 {|K|} \sum_{\edge\in\edges(K)} |\edge|\,  \phi_\edge^{n-1}\, \bfn_{K,\edge}.
\end{aligned}\]
Then for all $q$ in $[1,\infty]$, 
\begin{align} \label{eq:interpolate_t} & 
\norm{\partd_t \phi_\mesh-\dv_t\phi}_{\xL^q( \Omega\times(0,T))}
+ \norm{\gradi \phi_\mesh-\gradi\phi}_{\xL^q( \Omega\times(0,T))^d} \leq C_1 (\delta t +h_\disc),
\\[1ex] \label{eq:interpolate_s} &
\norm{\phi^0_\mesh-\phi(.,0)}_{\xL^q(\Omega)} \leq C_2 h_\disc.
\end{align}
where $C_1$ and $C_2$ are two positive constants only depending on $T$, $\Omega$, $\theta_0$, $q$ and $\phi$.
\end{lem}

We now prove Proposition \ref{prop:convergence_rho}.
\begin{proof}
Let $m \in \xN$ and let $\phi_K^{n-1}=\phi(\bfx_K,t_{n-1})$, for $K \in \mesh\m$ and $n\in\lbrace 1,..,N\m\rbrace$.
Multiplying the discrete mass balance equation \eqref{eq:sch_mass} by $\delta t\, |K|\, \phi_K^{n-1}$ and summing over $K\in\mesh$ and $n\in\lbrace 1,..,N\m\rbrace$, we get $T_1\m+T_2\m=0$ with
\[
T_1\m = \sum_{n=1}^N \sum_{K\in\mesh} |K|\, (\rho^n_K-\rho_K^{n-1}) \phi_K^{n-1},\qquad
T_2\m = \sum_{n=1}^N \delta t \sum_{K\in\mesh}  \phi_K^{n-1} \sum_{\edge \in\edges(K)} \fluxK^n,
\]
where we have dropped for short the superscript $\m$ for the number of time steps and the mesh.
We observe that, for $m$ large enough, $\phi^{N-1}_K=0$ for all $K\in\mesh$ since $\phi$ has a compact support in $\Omega \times [0,T)$; we suppose throughout this proof that we are in this case.
Performing a discrete integration by parts in $T_1\m$, we get:
\[ \begin{aligned}
 T_1\m  &= - \sum_{n=1}^N \sum_{K\in\mesh} |K|\,\rho_K^n (\phi^n_K-\phi_K^{n-1})  - \sum_{K\in\mesh} |K|\, \rho_K^0 \phi^0_K
\\
 &= - \int_0^T \int_\Omega \rho\m(\bfx,t)\, \partd_t \phi_{\mesh\m}(\bfx,t) \dx \dt
 - \int_\Omega \rho_0(\bfx)\, \phi_{\mesh\m}^0(\bfx) \dx. 
\end{aligned}
\]
By \eqref{eq:interpolate_t}, since $\delta t \m\to 0$ and $h\m\to 0$ as  $m\to \infty$, we obtain that $\partd_t \phi_{\mesh\m}$ strongly converges towards $\dv_t \phi$ in $\xL^1( \Omega\times(0,T))$ as $m\to \infty$; in addition, by \eqref{eq:interpolate_s}, $\phi^0_{\mesh\m}$ strongly converges to $\phi(\cdot, 0)$ in $\xL^1(\Omega)$.
Thus:
\[
\lim_{m\to\infty} T_1\m = - \int_0^T \int_\Omega \bar \rho(\bfx,t)\, \dv_t \phi(\bfx,t) \dx \dt
- \int_\Omega \rho_0(\bfx)\, \phi(\bfx,0) \dx.
\]
Let us now turn to the second term $T_2\m$.
Rearranging the terms, we get:
\begin{equation} \label{eq:first_term} 
\begin{aligned}
T_2\m &= -\sum_{n=1}^N \delta t \sum_{\substack{\edge \in \edgesint \\ \edge=K|L}}
|\edge|\, \rho_\edge^n \, (\phi_L^{n-1}-\phi_K^{n-1})\,\bfu_\edge^n \cdot \bfn_{K,\edge}
\\
&= -\sum_{n=1}^N \delta t \sum_{\substack{\edge \in \edgesint \\ \edge=K|L}}
|\edge|\, \frac{\rho_K^n+\rho_L^n} 2\, (\phi_L^{n-1}-\phi_K^{n-1})\, \bfu_\edge^n \cdot \bfn_{K,\edge} + R_{2,1}\m,
\end{aligned}
\end{equation}
where:
\[
R_{2,1}\m=\sum_{n=1}^N \delta t\sum_{\substack{\edge \in \edgesint \\ \edge=K|L}}
|\edge|\ \Bigl( \frac{\rho_K^n+\rho_L^n} 2 -\rho_\edge^n \Bigr)\ (\phi_L^{n-1}-\phi_K^{n-1})\ \bfu_\edge^n\cdot \bfn_{K,\edge}.
\]
Reordering the sum in the first term of \eqref{eq:first_term}, we get:
\[
T_2\m
=  - \frac 1 2 \sum_{n=1}^N \delta t \sum_{K\in\mesh} \rho_K^n \sum_{\substack{\edge \in \edges(K) \\ \edge=K|L}}
|\edge|\, (\phi_L^{n-1}-\phi_K^{n-1})\,\bfu_\edge^n \cdot \bfn_{K,\edge}
+ R_{2,1}\m.
\]
Since $(\dive \bfu)_K^n=0$ for $K \in \mesh$ and $1 \leq n \leq N$, the quantity $ \sum_{\edge \in \edges(K)} |\edge|\, \rho_K^n \,\phi_K^{n-1}\, \bfu_\edge^n \cdot \bfn_{K,\edge}$ vanishes.
Let us introduce the notation $\hat \phi_\edge^{n-1}=(\phi_K^{n-1}+\phi_L^{n-1})/2$, for $\edge \in \edgesint$, $\edge=K|L$, and for $1 \leq n \leq N$.
We thus get:
\[
T_2\m
=  - \sum_{n=1}^N \delta t \sum_{K\in\mesh} \rho_K^n \sum_{\substack{\edge \in \edges(K) \\ \edge=K|L}}
|\edge|\, \hat \phi_\edge^{n-1}\,\bfu_\edge^n \cdot \bfn_{K,\edge}
+ R_{2,1}\m
\]
and the term $T_2\m$ can be written as follows:
\[
T_2\m
=-\sum_{n=1}^N \delta t \sum_{K\in\mesh} \rho_K^n \bfu_K^n \cdot \sum_{\edge\in\edges(K)} |\edge|\, \phi_\edge^{n-1} \bfn_{K,\edge}
+ R_{2,1}\m+ R_{2,2}\m+ R_{2,3}\m,
\]
where $\displaystyle \bfu_K^n=\sum_{\edge\in\edges(K)}\xi_K^\edge \bfu_\edge^n$ with $\xi_K^\edge$ defined in \eqref{eq:xiksigma}, and $R_{2,2}\m$ and $R_{2,3}\m$ are defined by
\[\begin{aligned} &
R_{2,2}\m=  \sum_{n=1}^N \delta t \sum_{K\in\mesh} \rho_K^n
\sum_{\edge\in\edges(K)} |\edge|\, ( \phi_\edge^{n-1} - \hat \phi_\edge^{n-1})\, \bfu_\edge^n \cdot \bfn_{K,\edge},
\\ &
R_{2,3}\m=  \sum_{n=1}^N \delta t \sum_{K\in\mesh} \rho_K^n
\sum_{\edge\in\edges(K)} |\edge|\, \phi_\edge^{n-1}\, (\bfu_K^n-\bfu_\edge^n) \cdot \bfn_{K,\edge}.
\end{aligned}\]
Let us assume for now that $R_{2,1}\m + R_{2,2}\m+ R_{2,3}\m = {\mathcal O}((h\m)^{1/2})$ as $m$ tends to infinity.
We may then write
\begin{equation} \label{eq:T2lim}
\begin{aligned}
T_2\m &= - \sum_{n=1}^N \delta t \sum_{K\in\mesh}
|K|\, \rho_K^n\, \bfu_K^n \cdot \frac 1 {|K|} \sum_{\edge\in\edges(K)} |\edge|\,\phi_\edge^{n-1}\, \bfn_{K,\edge} + {\mathcal O}((h\m)^{1/2})
\\
&= - \int_0^T \int_\Omega \rho\m(\bfx,t)\, \tilde\bfu\m(\bfx,t) \cdot \gradi\phi_{\mesh\m}(\bfx,t) \dx \dt + {\mathcal O}((h\m)^{1/2}), 
\end{aligned}
\end{equation}
where $\tilde\bfu\m(\bfx,t)= \sum_{n=1}^N \sum_{K\in\mesh} \bfu_K^n\ \mathcal{X}_K(\bfx)\ \mathcal{X}_{(n-1,n]}(t)$.
The function $\tilde \bfu\m$ converges towards $\bar \bfu$ strongly in $\xL^2( \Omega\times(0,T))^d$ as $m$ tends to infinity.
Indeed,  
\[
\begin{aligned}
\norm{\tilde \bfu\m - \bfu\m}_{\xL^2( \Omega\times(0,T))^d}^2 
&=  \sum_{n=1}^N \delta t \sum_{K\in\mesh} \sum_{\edge\in\edges(K)} |D_{K,\edge}||\bfu_K^n -\bfu_\edge^n|^2 \\
&\leq {h\m}^2 \sum_{n=1}^N \delta t \sum_{K\in\mesh} h_K^{d-2}\sum_{\edge\in\edges(K)} |\bfu_K^n -\bfu_\edge^n|^2 \\
& \leq {h\m}^2 \sum_{n=1}^N \delta t\ \norm{\bfu(.,t_n)}_\fv^2,
\end{aligned}
\]
since $\bfu_K^n$ is a convex combination of  $(\bfu_\edge^n)_{\edge\in\edges(K)}$.
Hence, by Lemma \ref{lmm:H1ns} and the uniform estimate \eqref{eq:estimate3}, the difference $\tilde \bfu\m - \bfu\m$ converges to zero in $\xL^2( \Omega\times(0,T))^d$.
Since $\bfu\m\to\bar\bfu$ in $\xL^2( \Omega\times(0,T))^d$, so does $\tilde \bfu\m$.
Moreover, $\rho\m \weak^* \bar \rho$  in $\xL^\infty( \Omega\times(0,T))$ and, by \eqref{eq:interpolate_t}, $\gradi\phi_{\mesh\m} \to \gradi\phi$ strongly in $\xL^q( \Omega\times(0,T))$ for all $q$ in $[1,\infty]$.
Hence, \eqref{eq:T2lim} implies that
\[
\lim_{m\to\infty} T_2\m=  - \int_0^T \int_\Omega \bar \rho(\bfx,t)\, \bar\bfu(\bfx,t) \cdot \gradi \phi (\bfx,t) \dx \dt.
\]
Let us now prove that  $R_{2,1}\m + R_{2,2}\m+ R_{2,3}\m = {\mathcal O}((h\m)^{1/2})$ as $m$ tends to infinity.
For the first term $R_{2,1}\m$, we have
\begin{equation}\label{eq:R21}
|R_{2,1}\m| \leq \frac 1 2 \sum_{n=1}^N \delta t \sum_{\substack{\edge \in \edgesint \\ \edge=K|L}}
|\edge|\ | \rho_L^n -\rho_K^n|\ |\phi_L^{n-1}-\phi_K^{n-1}|\ | \bfu_\edge^n\cdot \bfn_{K,\edge}|,
\end{equation}
by the upwind definition \eqref{eq:rho_upwind} of $\rho_\edge^n$.
Hence, applying the Cauchy-Schwarz inequality, we get 
\[\begin{aligned}
|R_{2,1}\m| & \leq \frac 12 \Bigl(\sum_{n=1}^N \delta t \sum_{\substack{\edge \in \edgesint \\ \edge=K|L}}
|\edge|\, (\rho_L^n -\rho_K^n)^2  |\bfu_\edge^n\cdot \bfn_{K,\edge}|\Bigr)^{1/2} 
\\
& \hspace{20ex} \times \Bigl(\sum_{n=1}^N \delta t \sum_{\substack{\edge \in \edgesint \\ \edge=K|L}}
|\edge|\, (\phi_L^{n-1} -\phi_K^{n-1})^2  | \bfu_\edge^n \cdot \bfn_{K,\edge}|\Bigr)^{1/2}
\\
& \leq \frac 1 2\, C_0^{1/2}  \Bigl(\sum_{n=1}^N \delta t \sum_{\substack{\edge \in \edgesint \\ \edge=K|L}}
|\edge|\, (\phi_L^{n-1} -\phi_K^{n-1})^2  |\bfu_\edge^n\cdot \bfn_{K,\edge}|\Bigr)^{1/2}, 
\end{aligned}\]
by the estimate \eqref{eq:estimate2}.
By Taylor's inequality applied to $\phi_L^{n-1}-\phi_K^{n-1}$, there exists $C_4$ only depending on $T$, $\Omega$, $d$, $\theta_0$ and $\phi$ such that
\[
 |R_{2,1}\m| 
\leq C_4\ (h\m)^{1/2} \Bigl( \sum_{n=1}^N \delta t \sum_{\substack{\edge \in \edgesint \\ \edge=K|L}}
|D_\edge|\ |\bfu_\edge^n|\Bigr)^{1/2}
\leq  C_4\ (h\m)^{1/2} \Bigl( \sum_{n=1}^N \delta t\ \norm{\bfu(.,t_n)}_{\xL^1(\Omega)}\Bigr)^{1/2}.
\]
In addition, by the $\xL^p$-$\xL^q$ inequalities and the discrete Sobolev inequality of Lemma \ref{lmm:injsobolev_br}, we have for all $n$ in $\lbrace 1,..,N \rbrace$, $\norm{\bfu(.,t_n)}_{\xL^1(\Omega)^d}\leq C_5 \norm{\bfu(.,t_n)}_{\brok}$ where $C_5$ only depends on $\Omega$, $d$ and $\theta_0$.
Hence,
\[
|R_{2,1}\m| \leq C_6 \, (h\m)^{1/2}\, T^{1/4} \Bigl(\sum_{n=1}^N \delta t \norm{\bfu(.,t_n)}^2_{\brok}\Bigr)^{1/4}
\leq C_7\, T^{1/4}\ (h\m)^{1/2},
\]
by the estimate \eqref{eq:estimate3}, where $C_6$ and $C_7$ are independent of $m$.
The second term $R_{2,2}\m$ can be rearranged as follows
\[ \begin{aligned}
|R_{2,2}\m| &= \Bigl| \sum_{n=1}^N \delta t \sum_{\substack{\edge \in \edgesint \\ \edge=K|L}}
|\edge|\, (\rho_K^n-\rho_L^n)\, (\phi_\edge^{n-1} -\hat\phi_\edge^{n-1})\, \bfu_\edge^n \cdot \bfn_{K,\edge}  \Bigr|
\\
&\leq \sum_{n=1}^N \delta t \sum_{\substack{\edge \in \edgesint \\ \edge=K|L}}
|\edge|\, |\rho_K^n-\rho_L^n|\, |\phi_\edge^{n-1} -\hat \phi_\edge^{n-1}|\, |\bfu_\edge^n \cdot \bfn_{K,\edge}|.
\end{aligned}
\]
This last expression is now similar to the left-hand side of \eqref{eq:R21}, with $\phi_\edge^{n-1} -\hat \phi_\edge^{n-1}$ instead of $\phi_L^{n-1} - \phi_K^{n-1}$, but these terms both vary as $h\m$ by Taylor's inequality; following the same lines, we thus conclude that $R_{2,2}$ behaves as $R_{2,1}$, \ie\ $R_{2,2}\leq C_8 \,(h\m)^{1/2}$ for some $C_8$ independent of $m$.
The third term $R_{2,3}\m$ may be recast as follows:
\[
R_{2,3}\m = \sum_{n=1}^N \delta t \sum_{K\in\mesh} \rho_K^n 
\sum_{\edge\in\edges(K)} |\edge|\,(\phi_\edge^{n-1}-\phi_K^{n-1})\, (\bfu_K^n-\bfu_\edge^n)\cdot\bfn_{K,\edge},
\]
since $(\dive \bfu)_K^n=0$ and $\sum_{\edge\in\edges(K)} |\edge|\, \bfn_{K,\edge}=0$.
Hence,
\[\begin{aligned}
|R_{2,3}\m|
&
\leq  \sum_{n=1}^N \delta t \sum_{K\in\mesh} \rho_K^n \sum_{\edge\in\edges(K)}
|\edge|\ |\phi_\edge^{n-1}-\phi_K^{n-1}|\ |\bfu_K^n-\bfu_\edge^n|
\\ &
\leq \norm{\rho^n}_{\xL^\infty(\Omega)}
\Bigl(\sum_{n=1}^N \delta t\sum_{K\in\mesh}\ \sum_{\edge\in\edges(K)} |\edge|\ |\phi_\edge^{n-1}-\phi_K^{n-1}|^2 \Bigr)^{1/2} \\
& \hspace{20ex}
\times \Bigl(\sum_{n=1}^N \delta t\sum_{K\in\mesh}\ \sum_{\edge\in\edges(K)} |\edge|\ |\bfu_K^n-\bfu_\edge^n|^2 \Bigr)^{1/2} .
\end{aligned}\]
There exists $C_9(\theta_0)$ such that $ |\edge|\,|\phi_\edge^{n-1}-\phi_K^{n-1}|^2\leq C_9(\theta_0)\, |D_\edge|\ \norm{\gradi\phi}_{\xL^\infty(\Omega\times[0,T))^d}^2 h\m$ and  $|\edge|\leq h\m h_K^{d-2}$.
Moreover, since  $\norm{\rho^n}_{\xL^\infty(\Omega)} \leq \norm{\rho_0}_{\xL^\infty(\Omega)}$, we obtain that there exists $C_{10}$ independent of $m$ such that:
\[
|R_{2,3}\m| 
\leq C_{10}\ h\m  \Bigl( \sum_{n=1}^N \delta t\sum_{K\in\mesh}
h_K^{d-2}  \sum_{\edge\in\edges(K)} |\bfu_K^n-\bfu_\edge^n|^2 \Bigr)^{1/2}
\leq C_{10}\ h\m  \Bigl( \sum_{n=1}^N \delta t\ \norm{\bfu(.,t_n)}_\fv^2 \Bigr)^{1/2}, 
\]
since $\bfu_K^n$ is a convex combination of  $(\bfu_\edge^n)_{\edge\in\edges(K)}$.
Hence, by Lemma \ref{lmm:H1ns} and the uniform estimate \eqref{eq:estimate3}, there exists $C_{11}$ independent of $m$ such that $|R_{2,3}\m|\leq C_{11}\, h\m$, which concludes the proof of Proposition \ref{prop:convergence_rho}.
\end{proof}
%
%
\subsection{Strong convergence of the approximate densities} \label{subsec:strong_convergence_rho}

\begin{prop} \label{prop:strong_rho}
Under the assumptions of Theorem \ref{thrm:convergence}, the sequence $(\rho\m)_{m\in\xN}$ strongly converges in $\xL^q(\Omega\times(0,T))$, for all $q$ in $[1,\infty)$, towards its weak star limit in $\xL^\infty(\Omega\times(0,T))$, $\bar \rho$.
\end{prop}

\begin{proof}
As $(\rho\m)_{m\in\xN}$ is bounded $\xL^\infty(\Omega\times(0,T))$, it is sufficient, by interpolation, to prove the strong convergence of $\rho\m$ towards $\bar \rho$ in  $\xL^2(\Omega\times(0,T))$. As $\rho\m \weak^* \bar \rho$  in $\xL^\infty( \Omega\times(0,T))$, we also have $\rho\m \weak \bar \rho$  in $\xL^2( \Omega\times(0,T))$, which implies that $\norm{\bar \rho}_{\xL^2(\Omega\times(0,T))} \leq \liminf_{m \to \infty} \norm{\rho\m}_{\xL^2(\Omega\times(0,T))}$.
By estimate \eqref{eq:rho_bv} of Proposition \ref{prop:bvfaible}, we have for all $n$ in $\lbrace 1,..,N \rbrace$:
\[
\sum_{K \in \mesh}|K|  (\rho^n_K)^2 \leq  \sum_{K \in \mesh}|K|  (\rho^0_K)^2 \leq  \norm{\rho_0}_{\xL^2(\Omega)}^2,
\]
which yields $\norm{\rho\m(.,t)}_{\xL^2(\Omega)}^2 \leq \norm{\rho_0}_{\xL^2(\Omega)}^2$ for all $t\in(0,T)$ and all $m$ in $\xN$.
Moreover, $\bar \rho$ is a weak solution of the transport equation :
\[
\dv_t \bar \rho + \bar \bfu \cdot \gradi\bar \rho =0,
\]
where $\bar \bfu$ is a divergence-free function in $\xL^2((0,T),\xH^1_0(\Omega)^d)$. Thanks to the theory of renormalized solutions \cite{dip-89-ord}, it is possible to prove that $\bar \rho$ is unique (once $\bar \bfu$ and $\rho_0$ are given), and belongs to the space $C^0((0,T);\xL^2(\Omega))$. In addition, for any scalar function $\beta:\xR\to\xR$ in $\xC^\infty(\xR)$, the function $\beta(\bar \rho)$ is also a weak solution of the transport equation, which yields (taking $\beta(x)=x^2$, and using the boundary conditions on $\bar \bfu$) that  $\norm{\bar \rho(.,t)}_{\xL^2(\Omega)}=\norm{\rho_0}_{\xL^2(\Omega)}$ for all $t\in(0,T)$.
Therefore, we have $\norm{\rho\m(.,t)}_{\xL^2(\Omega)}^2 \leq \norm{\bar \rho(.,t)}_{\xL^2(\Omega)}^2$ for all $t\in[0,T)$ and all $m$ in $\xN$.
Integrating this last inequality for $t\in[0,T)$, we obtain $\norm{\rho\m}_{\xL^2(\Omega\times(0,T))}^2 \leq \norm{\bar \rho}_{\xL^2(\Omega\times(0,T))}^2$ for all $m$ in $\xN$, and passing to the limit as $m$ goes to infinity yields:
\[
\limsup_{m \to \infty} \norm{\rho\m}_{\xL^2(\Omega\times(0,T))} \leq \norm{\bar \rho}_{\xL^2(\Omega\times(0,T))}.
\]
This proves that $\lim_{m \to \infty} \norm{\rho\m}_{\xL^2(\Omega\times(0,T))} =\norm{\bar \rho}_{\xL^2(\Omega\times(0,T))}$ and so that $\rho\m$ converges strongly to $\bar \rho$ in $\xL^2(\Omega\times(0,T))$ as $m$ tends to infinity.
\end{proof}

\begin{rmrk}
 In the pioneering work by Di Perna and Lions \cite{dip-89-ord}, the authors deal with more general convection fields satisfying only $\dive \, \bar\bfu$ in $\xL^1((0,T);\xL^\infty)$, with some restriction on the class of functions $\beta$. In the present case of the transport equation with divergence-free velocity, the result holds for any smooth scalar function $\beta$. We refer to \cite[Theorem VI.1.6]{boy-13-math} for a detailed proof.  
\end{rmrk}

%
%
\subsection{Passing to the limit in the momentum balance equation} \label{subsec:convergence_u}

\begin{prop}\label{prop:convergence_u}
Under the assumptions of Theorem \ref{thrm:convergence}, the limit in $\xL^q( \Omega\times(0,T))^d$, $q$ in  $[1, \infty)$, of the sequence $(\rho\m)_{m\in\xN}$, $\bar \rho$, and the limit in $\xL^2( \Omega\times(0,T))^d$ of the sequence $(\bfu\m)_{m\in\xN}$, $\bar \bfu$, satisfy
\begin{multline*}
\int_0^T  \int_\Omega \Bigl( - \bar\rho(\bfx,t) \bar\bfu(\bfx,t) \cdot \dv_t \bfv(\bfx,t)
- (\bar\rho(\bfx,t) \bar\bfu(\bfx,t) \otimes \bar\bfu(\bfx,t)) : \gradi \bfv(\bfx,t)
\\
+ \gradi \bar\bfu(\bfx,t) : \gradi \bfv(\bfx,t) \Bigr) \dx \dt
= \int_\Omega \rho_0(\bfx) \bfu_0(\bfx) \cdot \bfv(\bfx,0) \dx,
\end{multline*}
for all $\bfv$ in $ \xC^{\infty}_c(\Omega \times [0,T))^d$ such that $\dive\bfv=0$.
\end{prop}

Before proving this proposition, we first state the following preliminary Lemma.

\begin{lem} \label{lmm:test_mom}
Let $\disc$ be a given staggered discretization such that $\theta_\disc \leq \theta_0$ for some positive real number $\theta_0$ and let $\delta t$ be a given time step.
Let $\bfv\in \xC^\infty_c(\Omega\times[0,T))^d$ with $\dive\bfv=0$ and let us define, for all $\edge$ in $\edges$, $K \in \mesh$ and $n$ in $\lbrace 0,..,N\rbrace$:
\[
\bfv_\edge^n=\frac 1 {|\edge|}\ \int_\edge \bfv(\bfx,t_n) \dedge(\bfx), \qquad
\bfv_K^n=\sum_{\edge\in\edges(K)}\xi_K^\edge \bfv_\edge^n,
\]
with $\xi_K^\edge$ defined in \eqref{eq:xiksigma}.
We denote $\bfv^0_\edges$ the piecewise constant function defined by 
\[
\bfv^0_\edges(\bfx) = \sum_{\edge \in \edges} \bfv_\edge^0\, \mathcal{X}_{D_\edge}(\bfx), 
\]
and we define the following discrete (or time partially discrete) interpolates of $\dv_t \bfv$, $\gradi \bfv$ and $\lapi \bfv$:
\[ \begin{aligned} &
\partd_t \bfv_\edges(\bfx,t) = \sum_{n=1}^N \sum_{\edge \in \edges} \frac 1 {\delta t}
\bigl( \bfv_\edge^n - \bfv_\edge^{n-1}  \bigr) \ \mathcal{X}_{D_\edge}(\bfx) \mathcal{X}_{(n-1,n]}(t),
\\ & 
\gradi \bfv_\edges(\bfx,t) = \sum_{n=1}^N \sum_{\substack{\edge \in \edgesint \\ \edge=K|L}}
\frac{|\edge|}{|D_\edge|} (\bfv_L^{n-1}-\bfv_K^{n-1}) \otimes \bfn_{K,\edge}\, \mathcal{X}_{D_\edge}(\bfx) \mathcal{X}_{(n-1,n]}(t),
\\ &
\lapi \bfv_{\delta t}(\bfx,t) = \sum_{n=1}^N \lapi \bfv(\bfx,t_{n-1})\mathcal{X}_{(n-1,n]}(t).
\end{aligned}\]
Then for all $q$ in $[1,\infty]$, there exists three positive real numbers $C_1$, $C_2$ and $C_3$ only depending on $T$, $\Omega$, $\theta_0$, $q$ and $\bfv$, such that
\begin{align} \label{eq:interpolate_tbis} &
\norm{\partd_t \bfv_\edges-\dv_t\bfv}_{\xL^q( \Omega\times(0,T))^d} \leq C_1 (\delta t +h_\disc),
\\[1ex] \label{eq:interpolate_sbis} &
\norm{\bfv^0_\edges-\bfv(.,0)}_{\xL^q(\Omega)^d} \leq C_2 h_\disc,
\\[1ex] \label{eq:interpolate_ster} &
\norm{\lapi\bfv_{\delta t}-\lapi \bfv}_{\xL^q( \Omega\times(0,T))^d} \leq C_3 \delta t.
\end{align}
Moreover, if $(\disc\m,\delta t\m)_{m\in\xN}$ is a regular sequence of staggered discretizations, for $q\in(1,\infty)$ (\emph{resp.} $q=\infty$) $\gradi\bfv_{\edges\m}$ converges weakly (\emph{resp.} weak-$\star$) towards $\gradi \bfv$ in $\xL^q( \Omega\times(0,T))^{d\times d}$.
\end{lem}

\begin{rmrk}
The proof of the weak (or weak star) convergence of the discrete gradient $\gradi\bfv_{\edges\m}$ towards $\gradi \bfv$  follows similar steps as the proof of Lemma \ref{lmm:compact_fv}.
\end{rmrk}

We are now in position to prove Proposition \ref{prop:convergence_u}.

\begin{proof}
Let $m \in \xN$ and let $\bfv_\edge^{n-1}$ be the mean value of $\bfv(\cdot,t_{n-1})$ over $\edge$, for $\edge \in \edges\m$ and $n\in\lbrace 1,..,N\m\rbrace$.
Taking the scalar product of the discrete momentum balance equation \eqref{eq:sch_mom} by $\delta t\, |D_\edge|\, \bfv_\edge^{n-1}$ and summing over $\edge\in\edgesint\m$ and $n\in\lbrace 1,..,N\m\rbrace$ we get $T_1\m+T_2\m+T_3\m+T_4\m=0$ with
\[\begin{aligned} &
T_1\m = \sum_{n=1}^N \sum_{\edge \in \edgesint} |D_\edge| (\rho_\Ds^n \bfu_\edge^n
- \rho_\Ds^{n-1} \bfu_\edge^{n-1}) \cdot \bfv_\edge^{n-1},
\\ &
T_2\m = \sum_{n=1}^N \delta t\sum_{\edge \in \edgesint} \sum_{\edged \in\edgesd(D_\edge)} \fluxd^n \bfu^n_\edged \cdot \bfv_\edge^{n-1},
\\ &
T_3\m =-  \sum_{n=1}^N \delta t \sum_{\edge \in \edgesint} |D_\edge| (\lapi \bfu)_\edge^n \cdot \bfv_\edge^{n-1},
\\ &
T_4\m = \sum_{n=1}^N \delta t \sum_{\edge \in \edgesint} |D_\edge|(\gradi p)_\edge^n \cdot \bfv_\edge^{n-1},
\end{aligned}\]
where we have dropped for short the superscript $\m$ on the number of time steps and the set of internal faces.
Since $\dive\bfv=0$ and by Relations \eqref{eq:div_stab} and \eqref{eq:dual_divgrad1} in Lemma \ref{lmm:dual_divgrad}, we get that $T_4\m=0$.
The three other terms, $T_1\m$, $T_2\m$ and $T_3\m$, stand respectively for the time derivative, the convection and the diffusion term.

\medskip
\noindent {\bf The time derivative term} --
Since the support of $\bfv$ is compact in $\Omega\times [0,T)$, $\bfv^{N-1}_\edge=0$ for all $\edge\in\edgesint$, at least for $m$ large enough; we suppose that we are in this case.
Performing a discrete integration by parts in $T_1\m$, we get:
\[ \begin{aligned}
T_1\m 
= & -  \sum_{n=1}^N \sum_{\edge\in\edgesint} |D_\edge|\, \rho_\Ds^n \bfu_\edge^n \cdot (\bfv^n_\edge-\bfv_\edge^{n-1})  - \sum_{\edge\in\edgesint} |D_\edge|\, \rho_\Ds^0 \bfu_\edge^0 \cdot \bfv^0_\edge
\\ 
= & -  \sum_{n=1}^N \sum_{\substack{\edge \in \edgesint \\ \edge=K|L}}
(|D_{K,\edge}|\, \rho_K^n +|D_{L,\edge}|\, \rho_L^n)\ \bfu_\edge^n \cdot (\bfv^n_\edge-\bfv_\edge^{n-1})
- \sum_{\substack{\edge \in \edgesint \\ \edge=K|L}} (|D_{K,\edge}|\, \rho_K^0 +|D_{L,\edge}|\, \rho_L^0)\ \bfu_\edge^0 \cdot \bfv^0_\edge, 
\end{aligned}\]
by the definition \eqref{eq:pd2} of $\rho_\Ds$.
Hence,
\[
T_1\m = - \int_0^T \int_\Omega \rho\m(\bfx,t)\, \bfu\m(\bfx,t) \cdot \partd_t \bfv_{\edges \m}(\bfx,t) \dx \dt 
- \int_\Omega (\rho\m)^0(\bfx)\ (\bfu\m)^0(\bfx) \cdot \bfv_{\edges\m}^0(\bfx) \dx.
\]
We have $\bfu\m\to\bar \bfu$ in $\xL^2( \Omega\times(0,T))^d$ and, by \eqref{eq:interpolate_tbis}, since $\delta t\m\to 0$ and $h\m\to 0 $ as $m\to\infty$, $\partd_t \bfv_{\edges\m}$ converges to $\dv_t\bfv$ strongly in $\xL^2( \Omega\times(0,T))^d$.
By the Cauchy-Schwarz inequality, $\bfu\m \cdot \partd_t \bfv_{\edges\m}$ thus converges towards $\bar \bfu \cdot \dv_t \bfv$ strongly in $\xL^1( \Omega\times(0,T))$.
Since $\rho\m\weak\bar\rho$ in $\xL^\infty( \Omega\times(0,T))$ weak-$\star$, we obtain the convergence of the first term.
In addition, from the initialization \eqref{eq:inicond} of the scheme and the assumed regularity of the initial data (\ie\ $\rho_0 \in \xL^\infty(\Omega)$ and $\bfu_0 \in \xL^2(\Omega)^d$), $(\rho\m)^0$ converges to $\rho_0$ in $\xL^q(\Omega)$ for all $q$ in $[1,\infty)$ and $(\bfu\m)^0$ converges to $\bfu_0$ in $\xL^q(\Omega)^d$ for all $q$ in $[1,2]$.
Finally, from Inequality \eqref{eq:interpolate_sbis}, $\bfv_{\edges\m}^0$ converges to $\bfv_0$ in $\xL^q(\Omega)^d$ for all $q$ in $[1,\infty]$.
Hence,
\[
\lim_{m\to \infty} T_1\m = - \int_0^T \int_\Omega \bar\rho(\bfx,t)\ \bar\bfu(\bfx,t) \cdot \dv_t \bfv(\bfx,t) \dx \dt
- \int_\Omega  \rho_0(\bfx)\ \bfu_0(\bfx) \cdot \bfv(\bfx,0) \dx \dt.
\]

\medskip
\noindent {\bf The convection term} --
Denoting $\bfv_K^{n-1}=\sum_{\edge\in\edges(K)}\xi_K^\edge \bfv_\edge^{n-1}$ where $\xi_K^\edge$ is defined in \eqref{eq:xiksigma}, thanks to Lemma \ref{lmm:convconv}, the second term $T_2\m$ may be written as
\begin{equation} \label{eq:T2_a}
T_2\m = \sum_{n=1}^N \delta t \sum_{K\in \mesh}
\bfv_K^{n-1} \cdot \sum_{\edge \in \edges(K)} |\edge|\, \rho_\edge^n\, (\bfu_\edge^n \cdot\bfn_{K,\edge})\, \bfu_\edge^n + R_{2,1}\m, 
\end{equation}
where the remainder term $R_{2,1}\m$ reads:
\[
 R_{2,1}\m = \sum_{n=1}^N \delta t
 \ \bigl(\mathcal{Q}_\mesh(\rho^n,\bfu^n,\bfu^n,\bfv^{n-1})-\mathcal{Q}_\edges(\rho^n,\bfu^n,\bfu^n,\bfv^{n-1})\bigr),
\]
with the notation for fixed-time discrete functions introduced in Definition \ref{def:disc_space}.
Rearranging the sum in \eqref{eq:T2_a}, we get:
\[\begin{aligned}
T_2\m
= & - \sum_{n=1}^N \delta t \sum_{\substack{\edge \in \edgesint \\ \edge=K|L}}
|\edge|\,\rho_\edge^n\, (\bfu_\edge^n\cdot\bfn_{K,\edge})\ \bfu_\edge^n \cdot(\bfv_L^{n-1}-\bfv_K^{n-1}) + R_{2,1}\m
\\  
= & -\sum_{n=1}^N \delta t \sum_{\substack{\edge \in \edgesint \\ \edge=K|L}}
|\edge|\, \Bigl( \frac{|D_{K,\edge}|}{|D_\edge|} \rho_K^n+\frac{|D_{L,\edge}|}{|D_\edge|} \rho_L^n \Bigr)
(\bfu_\edge^n\cdot\bfn_{K,\edge})\ \bfu_\edge^n \cdot (\bfv_L^{n-1}-\bfv_K^{n-1})
+ R_{2,2}\m + R_{2,1}\m,
\end{aligned}\]
where 
\[
R_{2,2}\m
= \sum_{n=1}^N \delta t \sum_{\substack{\edge \in \edgesint \\ \edge=K|L}}
|\edge|\, \Bigl(\frac{|D_{K,\edge}|}{|D_\edge|} \rho_K^n+\frac{|D_{L,\edge}|}{|D_\edge|} \rho_L^n - \rho_\edge^n \Bigr)  (\bfu_\edge^n\cdot\bfn_{K,\edge})\ \bfu_\edge^n \cdot(\bfv_L^{n-1}-\bfv_K^{n-1}).
\]
Assuming that $R_{2,1}\m+R_{2,2}\m={\mathcal O}\bigl((h\m)^{1/2}\bigr)$ as $m$ tends to infinity, we obtain
\[\begin{aligned}
T_2\m 
= & -\sum_{n=1}^N \delta t \sum_{\substack{\edge \in \edgesint \\ \edge=K|L}}
( |D_{K,\edge}| \rho_K^n+|D_{L,\edge}|\rho_L^n )\ (\bfu_\edge^n\otimes\bfu_\edge^n):\frac{|\edge|}{|D_\edge|}(\bfv_L^{n-1}-\bfv_K^{n-1})\otimes\bfn_{K,\edge}
+{\mathcal O}\bigl((h\m)^{1/2}\bigr)
\\ 
= & -\int_0^T \int_\Omega \rho\m(\bfx,t)\, \bfu\m(\bfx,t)\otimes\bfu\m(\bfx,t):\gradi \bfv_{\edges\m}(\bfx,t) \dx \dt
+ {\mathcal O}\bigl((h\m)^{1/2}\bigr).
\end{aligned}\]
We have, by Lemma \ref{lmm:test_mom} that $\gradi \bfv_{\edges\m}$ converges towards $\gradi \bfv$ in $\xL^\infty( \Omega\times(0,T))^{d\times d}$ weak-$\star$. Hence, if we prove that $\rho\m\bfu\m\otimes\bfu\m$ strongly converges towards $\bar \rho \bar \bfu\otimes\bar\bfu$ in $\xL^1((0,T)\times\Omega)^{d\times d}$, then we obtain that
\[
\lim_{m\to\infty} T_2\m =
-\int_0^T \int_\Omega \bar \rho(\bfx,t) \bar \bfu(\bfx,t)\otimes\bar\bfu(\bfx,t):\gradi \bfv(\bfx,t) \dx \dt.
\]
Since $\bfu\m\to\bar\bfu$ in $\xL^1(0,T);\xL^2(\Omega)^d)$, we have for all $1\leq i,j \leq d$,  $u_i\m u_j\m\to\bar u_i \bar u_j$ in $\xL^1(0,T);\xL^1(\Omega))$ where $u_i$ is the $i^{\rm th}$ component of $\bfu$. Moreover, since $\bar \bfu \in \xL^{\infty}((0,T);\xL^2(\Omega)^d) \cap \xL^2((0,T);\xH_0^1(\Omega)^d)$, we have for all $1\leq i,j \leq d$, $\bar u_i \bar u_j \in \xL^{\infty}((0,T);\xL^1(\Omega)) \cap \xL^1((0,T);\xL^3(\Omega))$ by the Sobolev injection $\xH^1_0(\Omega) \subset \xL^6(\Omega)$.
Thanks to the following interpolation inequality (see \cite[Theorem II.5.5]{boy-13-math}):
\[
\norm{\bar u_i \bar u_j}_{\xL^{5/3}((0,T);\xL^{5/3}(\Omega))} \leq 
\norm{\bar u_i \bar u_j}_{\xL^{\infty}((0,T);\xL^1(\Omega))}^{2/5} 
\norm{\bar u_i \bar u_j}_{\xL^1((0,T);\xL^3(\Omega))}^{3/5}, 
\]
we get $\bar u_i \bar u_j \in  \xL^{5/3}((0,T)\times \Omega)$. We may now prove that $\rho\m u_i\m u_j\m$ strongly converges towards $\bar \rho \bar u_i\bar u_j$ in $\xL^1((0,T)\times\Omega)$. Indeed:
\[
\begin{aligned} &
\norm{\rho\m u_i\m u_j\m -\bar \rho \bar u_i\bar u_j}_{\xL^1((0,T)\times\Omega)}
\\ & \hspace{5ex}
\leq \norm{\rho\m (u_i\m u_j\m-\bar u_i\bar u_j)}_{\xL^1((0,T)\times\Omega)} 
+ \norm{\bar u_i \bar u_j(\rho\m-\bar \rho)}_{\xL^1((0,T)\times\Omega)}
\\ & \hspace{5ex}
\leq \norm{\rho\m}_{\xL^\infty((0,T)\times\Omega)} \norm {u_i\m u_j\m-\bar u_i\bar u_j}_{\xL^1((0,T)\times\Omega)}
+ \norm{\bar u_i \bar u_j}_{\xL^{5/3}((0,T)\times\Omega)} \norm{\rho\m-\bar \rho}_{\xL^{5/2}((0,T)\times\Omega)}.
\end{aligned}
\]
By the maximum principle on $\rho\m$ and the strong convergence of $\rho\m$ towards $\bar \rho$ in $\xL^{q}((0,T)\times\Omega)$ for all $q\in[1,\infty)$, we obtain the expected convergence of $\rho\m u_i\m u_j\m$. Note that the strong convergence of the sequence $(\rho\m)_{m\in\xN}$ is needed here; this is why Proposition \ref{prop:strong_rho} must be proved before Proposition \ref{prop:convergence_u}. 

\medskip
Let us now prove that $R_{2,1}\m+R_{2,2}\m={\mathcal O}\bigl((h\m)^{1/2}\bigr)$ as $m$ tends to infinity.
For the term $R_{2,1}\m$, we have
\[ \begin{aligned}
|R_{2,1}\m|
&
\leq \sum_{n=1}^N \delta t~|\mathcal{Q}_\mesh(\rho^n,\bfu^n,\bfu^n,\bfv^{n-1})-\mathcal{Q}_\edges(\rho^n,\bfu^n,\bfu^n,\bfv^{n-1})|
\\ &
\leq C\ ({h\m})^{1/2} \sum_{n=1}^N \delta t \norm{\rho^n}_{\xL^{\infty}(\Omega)}\ \norm{\bfu^n}_\brok^2\ \norm{\bfv^{n-1}}_\brok, 
\end{aligned}
\]
by the estimate \eqref{eq:convconv} in Lemma \ref{lmm:convconv}.
By \eqref{eq:estimate1}, $\norm{\rho^n}_{\xL^{\infty}(\Omega)} \leq \norm{\rho_0}_{\xL^{\infty}(\Omega)}$ for all $n$ in $\lbrace1,..,N\rbrace$; in addition, $\norm{\bfv^n}_\brok\leq C(\Omega,d,\theta_0) \norm{\gradi \bfv}_{\xL^{\infty}( \Omega\times(0,T))^{d\times d}}$ for all $n$ in $\lbrace1,..,N\rbrace$, so, by estimate \eqref{eq:estimate3}, we obtain that there exists $C$ independent of $m$ such that $|R_{2,1}\m|\leq C ({h\m})^{1/2}$.
For the second remainder term, $R_{2,2}\m$, we may write, by the definition of $\rho_\edge^n$,
\[\begin{aligned}
|R_{2,2}\m|
&
= \Bigl| \sum_{n=1}^N \delta t \sum_{\substack{\edge \in \edgesint \\ \edge=K|L}} |\edge|\, (\rho_L^n-\rho_K^n)
 \Bigl( \frac{|D_{L,\edge}|}{|D_\edge|}(\bfu_\edge^n\cdot\bfn_{K,\edge})^+ 
-\frac{|D_{K,\edge}|}{|D_\edge|}(\bfu_\edge^n\cdot\bfn_{K,\edge})^-  \Bigr) \,\bfu_\edge^n \cdot(\bfv_L^{n-1}-\bfv_K^{n-1}) \Bigr|
\\ &
\leq \sum_{n=1}^N \delta t \sum_{\substack{\edge \in \edgesint \\ \edge=K|L}}
|\edge|\ |\rho_L^n-\rho_K^n|\ |\bfu_\edge^n\cdot\bfn_{K,\edge}|\ |\bfu_\edge^n|\ |\bfv_L^{n-1}-\bfv_K^{n-1}|.
\end{aligned}\]
Applying the Cauchy-Schwarz inequality, we therefore obtain:
\[
|R_{2,2}\m|
\leq \Bigl( \sum_{n=1}^N \delta t \sum_{\substack{\edge \in \edgesint \\ \edge=K|L}}
|\edge| (\rho_L^n-\rho_K^n)^2 |\bfu_\edge^n\cdot\bfn_{K,\edge}|\Bigr)^{1/2}
\Bigl( \sum_{n=1}^N \delta t \sum_{\substack{\edge \in \edgesint \\ \edge=K|L}}
|\edge| |\bfu_\edge^n\cdot\bfn_{K,\edge}|\ |\bfu_\edge^n|^2 |\bfv_L^{n-1}-\bfv_K^{n-1}|^2 \Bigr)^{1/2}. 
\]
The first term of the product at the right-hand side is controlled by the uniform estimate \eqref{eq:estimate2}.
Since, for any cell $M\in\mesh$, $\bfv_M^{n-1}$ is a convex combination of $\bfv_\edge^{n-1}$ for $\edge$ in $\edges(M)$, we obtain that $|\edge|\, |\bfv_L^n-\bfv_K^n|^2\leq C\ |D_\edge|\ \norm{\gradi\bfv}_{\xL^\infty(\Omega\times[0,T))^{d\times d}}^2\, h\m$, where $C$ only depends on $\theta_0$.
Hence, there exists $C$ independent of $m$ such that
\[
|R_{2,2}\m| \leq 
C \Bigl( h\m \sum_{n=1}^N \delta t  \sum_{\substack{\edge \in \edgesint \\ \edge=K|L}} |D_\edge| |\bfu_\edge^n|^3 \Bigr)^{1/2}
\hspace{-2ex}= C (h\m)^{1/2} \Bigl( \norm{\bfu\m}_{\xL^3((0,T);\xL^3(\Omega)^d)}\Bigr)^{3/2}.
\]
Now, as a consequence of H\"older's inequality, we have
\[
\norm{\bfu\m}_{\xL^3((0,T);\xL^3(\Omega)^d)} \leq |\Omega|^{1/6}
\ \norm{\bfu\m}_{\xL^\infty((0,T);\xL^2(\Omega)^d)}\ \norm{\bfu\m}_{\xL^2((0,T);\xL^6(\Omega)^d)}.
\]
Since $\norm{ \bfu^{(m)}}_{\xL^2((0,T);\xL^6(\Omega)^d)}\leq\norm{ \bfu^{(m)}}_{\xL^2((0,T);E_{\disc\m}(\Omega))}\leq C_1$ by the discrete Sobolev inequality of Lemma \ref{lmm:injsobolev_br} and the uniform estimate \eqref{eq:estimate3}, we obtain that $|R_{2,2}|\leq C (h\m)^{1/2}$.

\bigskip
\noindent {\bf The diffusion term} --
Denoting $\tilde \bfu^n(\bfx)=\sum_{\edge\in\edgesint}\bfu_\edge^n\, \zeta_\edge(\bfx)$ and similarly $\tilde \bfv^{n-1}(\bfx) = \sum_{\edge\in\edgesint} \bfv_\edge^{n-1}\, \zeta_\edge(\bfx)$ (so the function $\tilde \bfv^{n-1}$ is defined by $\tilde \bfv^{n-1}(\bfx) = \widetilde{r_\edges \bfv}(\bfx,t_{n-1})$, where $r_\edges$ is the projection operator defined in \eqref{eq:rh}), we have:
\begin{equation} \label{eq:t333}
\begin{aligned}
T_3\m &= \sum_{n=1}^N \delta t \sum_{K\in\mesh} \int_K \gradi \tilde \bfu^n(\bfx): \gradi  \tilde \bfv^{n-1}(\bfx) \dx
\\ &
= \sum_{n=1}^N \delta t \sum_{K\in\mesh} \int_K \gradi \tilde \bfu^n(\bfx): \gradi \bfv(\bfx,t_{n-1}) \dx + R_{3,1}\m,
\end{aligned}
\end{equation}
where
\[
R_{3,1}\m   =  \sum_{n=1}^N \delta t \sum_{K\in\mesh}
\int_K \gradi \tilde \bfu^n (\bfx): \bigl(\gradi \tilde \bfv^{n-1}(\bfx) -\gradi  \bfv (\bfx,t_{n-1})\bigr) \dx.
\]
Performing an integration by parts over each element $K$ in \eqref{eq:t333}, we obtain
\[
T_3\m = -\sum_{n=1}^N \delta t \sum_{K\in\mesh} \int_K \tilde \bfu^n(\bfx) \cdot \lapi \bfv(\bfx,t_{n-1}) \dx + R_{3,2}\m+R_{3,1}\m,
\]
where
\[
R_{3,2}\m = \sum_{n=1}^N \delta t \sum_{K\in\mesh} \sum_{\edge\in\edges(K)}
\int_\edge [\tilde \bfu^n]_\edge(\bfx) \cdot \gradi \bfv(\bfx,t_{n-1}) \cdot \bfn_{K,\edge} \dedge(\bfx). 
\]
Let us now replace $\tilde \bfu^n$ by the piecewise constant function $\bfu^n$ over each diamond cell, which yields:
\[ \begin{aligned}
T_3\m  &=  -\sum_{n=1}^N \delta t \sum_{K\in\mesh}\sum_{\edge\in\edges(K)}
\int_{D_{K,\edge}} \tilde \bfu^n (\bfx) \cdot \lapi\bfv (\bfx,t_{n-1}) \dx + R_{3,2}\m+R_{3,1}\m
\\
&=-\sum_{n=1}^N \delta t \sum_{K\in\mesh}\sum_{\edge\in\edges(K)}
\int_{D_{K,\edge}} \bfu_\edge^n \cdot \lapi \bfv (\bfx,t_{n-1}) \dx + R_{3,3}\m+ R_{3,2}\m+R_{3,1}\m,
\end{aligned}
\]
where
\[
R_{3,3}\m = \sum_{n=1}^N \delta t \sum_{K\in\mesh}\sum_{\edge\in\edges(K)}
\int_{D_{K,\edge}} ( \bfu_\edge^n -\tilde \bfu^n(\bfx)) \cdot \lapi \bfv(\bfx,t_{n-1}) \dx.
\]
Assuming that $R_{3,3}\m+ R_{3,2}\m = {\mathcal O}(h\m)$ and $R_{3,1}\m = {\mathcal O}(h\m+\alpha\m)$ as $m$ tends to infinity, we may write
\[\begin{aligned}
 T_3\m 
 &
 =-\sum_{n=1}^N \delta t \sum_{K\in\mesh}\sum_{\edge\in\edges(K)} \int_{D_{K,\edge}} \bfu_\edge^n \cdot \lapi\bfv (\bfx,t_{n-1}) \dx
+{\mathcal O}({h\m}+\alpha\m)
\\ &
= -\int_0^T \int_\Omega \bfu\m(\bfx,t) \cdot \lapi \bfv_{\delta t \m}(\bfx,t) \dx \dt +{\mathcal O}({h\m}+\alpha\m).
\end{aligned}\]
As $\bfu\m$ strongly converges towards $\bar \bfu$ in $\xL^2(\Omega\times(0,T))^d$ and, by Inequality \eqref{eq:interpolate_ster}, $\lapi \bfv_{\delta t \m}$ strongly converges towards $\lapi\bfv$ in $\xL^q(\Omega\times(0,T))^d$ for all $q$ in $[1,\infty]$, we obtain
\[
\lim_{m\to\infty}  T_3\m = -\int_0^T \int_\Omega \bar \bfu(\bfx,t) \cdot \lapi \bfv(\bfx,t) \dx \dt 
= \int_0^T \int_\Omega \gradi \bar \bfu(\bfx,t) : \gradi \bfv(\bfx,t) \dx \dt,
\]
since $\bar \bfu \in \xL^2((0,T);\xH^1_0(\Omega)^d)$.
Let us now prove that $R_{3,1}\m = {\mathcal O}(h\m+\alpha\m)$ as $m$ tends to infinity.
Applying the Cauchy-Schwarz inequality for every primal cell $K$, $R_{3,1}\m$ can be controlled as follows:
\[\begin{aligned}
|R_{3,1}\m|
&
\leq  \sum_{n=1}^N \delta t \sum_{K\in\mesh} \norm{\gradi \tilde \bfu^n}_{\xL^2(K)^{d\times d}}
\norm{\gradi \tilde \bfv^{n-1} -\gradi  \bfv (.,t_{n-1})}_{\xL^2(K)^{d\times d}}
\\ &
\leq (h\m+\alpha\m) \sum_{n=1}^N \delta t \sum_{K\in\mesh} \norm{\gradi \tilde \bfu^n}_{\xL^2(K)^{d\times d}}
\snorm{\bfv (.,t_{n-1})}_{\xH^2(K)^d},
\end{aligned}\]
by the approximation properties of the Rannacher-Turek finite element space stated in Lemma \ref{lmm:RT}.
Using the Cauchy-Schwarz inequality again, we obtain, by definition of the broken Sobolev norm,
\[
\begin{aligned}
|R_{3,1}\m|
&\leq (h\m+\alpha\m)\ \norm{\bfu\m}_{\xL^2((0,T);E_{\disc\m}(\Omega))} \Bigl(\sum_{n=1}^N \delta t\, \snorm{\bfv (.,t_{n-1})}_{\xH^2(\Omega)^d} \Bigr)^{1/2} \\
& \leq C(T,\Omega,\bfv)\ C_1 \,(h\m+\alpha\m),
\end{aligned}
\]
where $C_1$ is given by the uniform estimate \eqref{eq:estimate3}.
Let us now consider the second term $R_{3,2}\m$.
By the weak continuity requirement for the discrete finite element velocity fields, we may write
\[
R_{3,2}\m = \sum_{n=1}^N \delta t \sum_{K\in\mesh} \sum_{\edge\in\edges(K)}
\int_\edge [\tilde \bfu^n]_\edge(\bfx) \cdot ( \gradi \bfv (\bfx,t_{n-1})-\gradi \bfv (\bfx_\edge,t_{n-1}))
\cdot \bfn_{K,\edge} \dedge(\bfx),
\]
since $\gradi \bfv (\bfx_\edge,t_n)). \bfn_{K,\edge}$ is independent of the integration variable $\bfx$ in $\edge$.
Using first the Cauchy-Schwarz inequality in $\xL^2(\edge)^d$ and then the discrete Cauchy-Schwarz inequality, we get, denoting $h_\edge = {\rm diam}(\edge)$:
\[\begin{aligned}
|R_{3,2}\m|
&
\leq \sum_{n=1}^N \delta t  \sum_{\edge \in \edges} \Bigl( \frac 1 {h_\edge}
\int_\edge [\tilde \bfu^n]^2_\edge(\bfx) \dedge(\bfx) \Bigr)^{1/2}
\Bigl(h_\edge \int_\edge | \gradi \bfv(\bfx,t_{n-1}) - \gradi \bfv(\bfx_\edge,t_{n-1}) |^2 \dedge(\bfx) \Bigr)^{1/2}
\\ &
\leq \Bigl(\sum_{n=1}^N \delta t \sum_{\edge \in \edges} \frac 1 {h_\edge}
\int_\edge [\tilde \bfu^n]^2_\edge(\bfx) \dedge(\bfx) \Bigr)^{1/2}
\Bigl(\sum_{n=1}^N \delta t \sum_{\edge \in \edges} h_\edge
\int_\edge | \gradi \bfv(\bfx,t_{n-1}) - \gradi \bfv(\bfx_\edge,t_{n-1}) |^2 \dedge(\bfx) \Bigr)^{1/2}.
\end{aligned}\]
By the regularity of $\bfv$, we have 
\[
|\gradi \bfv(\bfx,t_{n-1}) - \gradi \bfv(\bfx_\edge,t_{n-1})| \leq
h_\edge\ |\bfv|_{\xW^{2,\infty}(\Omega \times (0,T))^d},
\]
for all $\bfx \in  \edge,\ \edge \in \edges$ and $n \in \lbrace 1,..,N\rbrace$, which yields
\[
|R_{3,2}\m|
\leq C(\theta_0,T,\Omega,\bfv)\ \Bigl(\sum_{n=1}^N \delta t \sum_{\edge \in \edges} \frac 1 {h_\edge}
\int_\edge [\tilde \bfu^n]^2_\edge(\bfx) \dedge(\bfx) \Bigr)^{1/2}  h\m.
\]
Invoking \eqref{eq:controle_sauts_EF} in Lemma \ref{lmm:controle_sauts_EF} and the uniform estimate \eqref{eq:estimate3}, we obtain that $|R_{3,2}\m|\leq C h\m$ where $C$ does not depend on $m$.
We finally turn to the last term $R_{3,3}\m$.
First, we have
\[
|R_{3,3}\m| \leq \norm{\lapi\bfv}_{\xL^\infty( \Omega\times(0,T))^d} \sum_{n=1}^N \delta t
\sum_{K\in\mesh}\ \sum_{\edge\in\edges(K)}\int_{D_{K,\edge}} | \bfu_\edge^n -\tilde \bfu^n(\bfx)| \dx. 
\]
Moreover, observing that for each primal cell $K$ in $\mesh$, $\forall \bfx \in K$, $\sum_{\edge\in\edges(K)} \zeta_\edge(\bfx)=1$, and that the Rannacher-Turek shape functions are uniformly bounded by a real number only depending on $\theta_0$, we may write :
\[\begin{aligned}
\sum_{\edge\in\edges(K)} \int_{D_{K,\edge}} | \bfu_\edge^n -\tilde \bfu^n(\bfx)| \dx
&
= \sum_{\edge\in\edges(K)}\int_{D_{K,\edge}} \Big | \sum_{\edge'\in\edges(K)}
(\bfu_\edge^n -\bfu_{\edge'}^n)\ \zeta_{\edge'}(\bfx) \Big |\dx
\\ &
\leq  C(d)\ |K|\ \sum_{\edge,\edge'\in\edges(K)}
|\bfu_\edge^n -\bfu_{\edge'}^n|.
\end{aligned}\]
By the Cauchy-Schwarz inequality, we thus get:
\[\begin{aligned}
|R_{3,3}\m|
&
\leq C(d)\ \norm{\lapi\bfv}_{\xL^\infty( \Omega\times(0,T))^d}
\ \Bigl(\sum_{n=1}^N \delta t \ \sum_{K\in\mesh} h_K^{d-2}
\sum_{\edge,\edge'\in\edges(K)} |\bfu_\edge^n -\bfu_{\edge'}^n|^2 \Bigr)^{1/2}
\Bigl(\sum_{n=1}^N \delta t \ \sum_{K\in\mesh} h_K^2\ |K|\Bigr)^{1/2}
\\ &
\leq C(d)\ \norm{\lapi\bfv}_{\xL^\infty( \Omega\times(0,T))^d}\ |\Omega|^{1/2}\, T^{\frac 12} \, h\m
\Bigl(\sum_{n=1}^N \delta t  \sum_{K\in\mesh} h_K^{d-2}
\sum_{\edge,\edge'\in\edges(K)} |\bfu_\edge^n -\bfu_{\edge'}^n|^2 \Bigr)^{1/2}.
\end{aligned}\]
Since the finite volume $\xH^1$-norm is controlled by the finite element  $\xH^1$-norm (see Lemma \ref{lmm:H1ns}), we conclude by the uniform estimates \eqref{eq:estimate3} that there exists $C$ independent of $m$ such that $|R_{3,3}\m|\leq C h\m$, which ends the proof of Proposition \ref{prop:convergence_u}.
\end{proof}
%
%
\section{Extension to other discretizations and models}

\subsection{Other space discretizations}

A straightforward extension may be obtained by considering more general meshes composed of both simplices and quadrilaterals ($d=2$) or hexahedra ($d=3$). 
We refer to \cite{her-13-pre} for the precise definition of the dual mesh in this case. The discretization of the time derivatives, the mass and momentum convection fluxes, the free-divergence constraint as well as the pressure gradient remains unchanged. The only difference lies in the treatment of the diffusion term. Now, a local shape function associated with an edge of a simplex identifies with a Crouzeix-Raviart shape function.
The Crouzeix-Raviart element \cite{cro-73-con} enjoys similar stability and approximation properties to those of the Rannacher-Turek element stated in Lemmata \ref{lmm:RT} and \ref{lmm:controle_sauts_EF}. Moreover, the pair of approximation spaces also satisfies a discrete \emph{inf-sup condition} as in Lemma \ref{lmm:inf-sup}. Finally, it is also possible to prove the control of the finite element $\xH^1$-norm over the finite volume $\xH^1$-norm as stated in Lemma \ref{lmm:H1ns}. For these reasons, the main result of the paper, namely the convergence theorem, readily extends for this mixed space discretization.

\medskip
Another extension of the present result is possible for the MAC (Marker-And-Cell) space discretization.
A forthcoming paper proves a similar convergence result of the implicit staggered MAC scheme for the incompressible variable density Navier-Stokes equations.

\subsection{Other time discretizations}

The scheme \eqref{eq:impl_scheme} is fully implicit, and thus the implementation of the algorithm implies to find the solution of a fully non-linear coupled system.
Consequently, using this scheme appears to be difficult in a real computational context, mainly due to the computational cost and lack of robustness.
In the following, we describe three other possible time discretizations which yield efficient schemes, obtained thanks to a partial decoupling of the discrete equations. We also discuss the conditions under which these schemes satisfy stability estimates similar to those satisfied by the implicit scheme, which eventually enables to extend the convergence result.
We give the corresponding time algorithms, keeping the same staggered space discretization as previously.

\medskip
\subsubsection{Mass and momentum transport with an explicit convection field}
The first alternate time discretization is obtained through an explicit treatment of the convective velocity in the mass transport equation:
\begin{subequations} \label{expl_conv} 
\begin{align} 
& \displaystyle \dfrac 1 {\delta t}(\rho^n-\rho^{n-1}) + \dive(\rho^{n}\bfu^{n-1})=0, \label{expl_conv:mass}
\\
& \displaystyle \dfrac 1 {\delta t}(\rho^n \bfu^n-\rho^{n-1} \bfu^{n-1})
+ \divv(\rho^n \bfu^{n-1} \otimes \bfu^{n})
-\lapi \bfu^n
+ \gradi p^n =0, \label{expl_conv:mom}
\\[1ex]
& \dive \, \bfu^n  =0.
\end{align}
\end{subequations}
This scheme satisfies similar estimates for the density and the velocity as the fully implicit scheme, with no restriction on the time-step. Indeed, the density is controlled in $\xL^\infty$ provided an upwind discretization of the mass convection term, and the velocity is controlled through a discrete kinetic energy equation, obtained when taking the scalar product of \eqref{expl_conv:mom} with $\bfu^n$. Easy computations, similar to those performed in the proof of Lemma \ref{lmm:kinetic_energy} in the implicit case, yield the desired kinetic energy equation, given here in a semi-discrete form:
\begin{multline*}
\frac{1}{2 \delta t} \big ( \rho^n |\bfu^n|^2-\rho^{n-1} |\bfu^{n-1}|^2\Big) + \frac 12\dive\big(\rho^{n}|\bfu^n|^2\bfu^{n-1}\big)-\lapi \bfu^n \cdot \bfu^n + \gradi p^n \cdot \bfu^n 
\\+ \frac{1}{2 \delta t} \rho^{n-1}|\bfu^n-\bfu^{n-1}|^2+\underbrace{\Big( \dfrac 1 {\delta t}(\rho^n-\rho^{n-1}) + \dive(\rho^{n}\bfu^{n-1})\Big )}_{\text{$=0$ by \ref{expl_conv:mass}}}\frac{|\bfu^n|^2}{2} = 0.
\end{multline*}
Hence, the convergence analysis still holds in this case.
The benefit from such a discretization comes from the decoupling of the mass balance and hydrodynamics; from a computational point of view, the difficulty is now reduced to compute the solution of the linear system associated with (linearized) Navier-Stokes equations, for which some techniques are available (SIMPLE-like methods, Augmented Lagrangian algorithms, \dots).

\medskip
\subsubsection{An (as much as possible) explicit scheme}
An interesting scheme for low viscosity flows (typically a viscosity $ \mu$ less than the space step $h$) is the following scheme, where only the pressure is treated in an implicit way (which is mandatory for stability reasons): 
\begin{subequations} \label{expl} 
\begin{align} 
& \displaystyle \dfrac 1 {\delta t}(\rho^n-\rho^{n-1}) + \dive(\rho^{n}\bfu^{n-1})=0, 
\label{expl:mass} \\
& \displaystyle \dfrac 1 {\delta t}(\rho^n \bfu^n-\rho^{n-1} \bfu^{n-1})
+ \divv(\rho^n \bfu^{n-1} \otimes \bfu_{\rm upw}^{n-1})
- \mu \lapi \bfu^{n-1}
+ \gradi p^n =0, 
\label{expl:mom} \\ 
& \dive \, \bfu^n  =0.
\label{expl:div}
\end{align}
\end{subequations}
Here, the mass equation can still be solved independently of the momentum equation, and the solution of the latter is obtained by solving an elliptic problem on the pressure $p^n$, which is not difficult from a computational point of view.
Indeed, computing the (discrete) divergence of (the non-conservative form of) \eqref{expl:mom} and using \eqref{expl:div}, yields a relation of the form $-\lapi p^n = F^{n-1}$.
The $\xL^2$-stability of this scheme (\ie~a discrete kinetic energy balance equation) is ensured under a CFL restriction on the time step of the form $\delta t \leq c(h/|\bfu^{n-1}|+h^2/\mu)$ (see \cite{her-13-expl}), provided that an upwind space discretization be used in the convection term of the momentum balance equation (contrary to what is done in the present paper).
Here again, the convergence analysis is the same as for the implicit scheme, with a slight difference due to the upwind choice in the momentum equation.
The assessment of the computational efficiency of this scheme, with applications to the Large Eddy Simulation of turbulent flows, is planned in a near future.

\medskip
\subsubsection{Projection method}
Finally, an interesting scheme from the computational point of view is the following pressure-correction scheme. Here again, it is possible to solve the correction step \eqref{corr1}-\eqref{corr2} thanks to an elliptic problem on the pressure.
\begin{subequations}
\label{eq:corr_press_scheme} \begin{align} 
& \displaystyle \qquad
\dfrac 1 {\delta t}(\rho^n-\rho^{n-1}) + \dive(\rho^{n}\bfu^{n-1})=0.
\\[2ex] & \nonumber
\text{\textit{Velocity prediction step:} }
\\  \label{eq:pred} & \displaystyle \qquad
\dfrac{1}{\delta t}(\rho^{n} \tilde \bfu^n-\rho^{n-1}\bfu^{n-1})
+  \divv(\rho^n \bfu^{n-1} \otimes \tilde \bfu^n)
- \lapi \tilde \bfu^n
+ \left (\frac{\rho^{n}}{\rho^{n-1}} \right )^{1/2}\gradi p^{n-1} =0.
\\[2ex] & \nonumber
\text{\textit{Velocity and pressure correction step:} }
\\ \label{corr1} & \displaystyle \qquad
\dfrac{1}{\delta t} \rho^n \, (\bfu^{n}-\tilde \bfu^n) + \gradi p^{n}- \left (\frac{\rho^{n}}{\rho^{n-1}} \right )^{1/2}\gradi p^{n-1}  =0,  
\\ \label{corr2} & \qquad
\dive \, \bfu^n  =0. 
\end{align} 
\end{subequations}

\smallskip
In addition, this scheme is unconditionally $\xL^2$-stable, \ie~ without any restriction on the time-step.
More precisely, thanks to the scaling factor $(\rho^{n}/\rho^{n-1} )^{1/2}$ applied to the pressure gradient in \eqref{eq:pred}, it is possible to prove that the beginning and end-of-step velocities $(\bfu^n)_{0\leq n \leq N}$ are controlled in a discrete $\xL^\infty(\xL^2)$-norm, while the intermediate velocities $(\tilde \bfu^n)_{0\leq n \leq N}$ are controlled in a discrete $\xL^2(\xH_0^1)$-norm thanks to the diffusion term (see \cite{her-14-imp} for a similar computation in the case of the compressible Euler's equations).
A consequence of these different estimates for $\tilde \bfu^n$ and $\bfu^n$ is that the velocity convergence obtained in the fully implicit case does not hold anymore, and the scheme convergence is still an open issue.
We are currently investigating this for the (simpler) constant density case. Note that still in the constant density case, error estimates, obtained when assuming that the solution to the continuous problem is smooth, have been the topic of a wide literature (see \eg\ \cite{she-92-one,she-92-oneII,gue-96-som}).

\medskip
\subsection{Extension to a density-dependent viscosity}
The model with density-dependent viscosity, for which a convergence analysis is carried out in \cite{liu-07-conv} for a discontinuous Galerkin method, reads as follows:
\[
\begin{aligned} &
\partial_t\rho +\dive (\rho\bfu)=0,
\\[0ex] &
\partial_t (\rho\bfu)+\dive (\rho\bfu \otimes \bfu)-\divv( \mu(\rho)\bfD(\bfu)) +\gradi p=0,
\\[0ex] &
\dive \, \bfu = 0,
\end{aligned}
\]
where $\bfD(\bfu)$ is the symmetric part of $\gradi \bfu $, and the viscosity $\mu$ is a continuous positive function of the density (typically $\mu(\rho)=\rho^{1/2}$ for an ideal gas). The convergence analysis performed here for the staggered implicit scheme is still valid, provided some stabilization term is added in the discretization of the diffusion term:
\[
 -(\divv \,\mu \bfD(\bfu))_\edge  = \frac 1 {|D_\edge|} \sum_{K \in \mesh} \int_K \mu(\rho_K)\sum_{\edge'\in \edges(K)}
\bfu_{\edge'} \, \bigl(\bfD( \zeta_{\edge'}) \cdot \gradi \zeta_\edge \bigr) \dx + \mbox{Stab}.
\]
The discrete kinetic energy, which is derived as in the constant viscosity case, yields a control on the $\xL^2((0,T);\xL^2)$-norm of $\bfD(\bfu)$ (actually of $D(\tilde \bfu)$ where $\tilde \bfu$ is the finite element approximation of $\bfu$). In order to infer from this an $\xL^2((0,T);\xH_0^1)$ control of the velocity, one needs a discrete Korn's inequality. In \cite{bre-03-kor}, the author proves discrete Korn's inequalities, which in our context may be written as follows:
\[
\norm{\bfu}_{\brok}^2 \leq C \left (  \norm{D(\tilde \bfu)}_{\xL^2(\Omega)}^2 +\norm{\tilde \bfu}_{\xL^2(\Omega)}^2+ \sum_{\edge\in\edges} \frac{1}{h_\edge} \norm{\pi_\edge [\tilde \bfu]_\edge}_{\xL^2(\edge)}^2\right),
\]
where $\pi_\edge$ is the orthogonal projection operator from $\xL^2(\edge)$ onto the space of vector polynomial functions on $\edge$ of degree less than one. Hence, the desired control on the velocity may be obtained, provided a stabilization term in the discrete diffusion that allows a control on the edge velocity jumps in the above discrete Korn's inequality. 
%
%
\appendix
\section{A topological degree result}\label{degree}

The following theorem follows from standard arguments of the topological degree theory (see \cite{deimling} for an overview of the theory and \eg\ \cite{eym-98-an, gal-08-unc} for other uses in the same objective as here, namely the proof of existence of a solution to a numerical scheme).

\begin{thrm} \label{thrm:degree} 
Let $N$ and $M$ be two positive integers and $V=\xR^N \times \xR^M \times \xR^N$. Let $b\in V$ and $f(\cdot)$ and $F(\cdot,\cdot)$ be two continuous functions respectively from $V$ and $V\times[0,1]$ to $V$ satisfying:
\begin{enumerate}
\item[(i)]$F(\cdot,\,1)=f(\cdot)$;
\item[(ii)] $\forall \alpha \in [0,1]$, if an element $v$ of $\bar {\mathcal O}$ (the closure of ${\mathcal O}$) is such that $F(v,\alpha)=b$, then $v \in {\mathcal O}$, where ${\mathcal O}$ is defined as follows:
$$
{\mathcal O}=\{ (x,y,z) \in V \mbox{ s.t. } C_0< x < C_1 \mbox{ and } \norm{y}_M< C_2 \mbox{ and } \norm{z}_N< C_3 \}
$$
where, for any real number $c$ and vector $x$, the notation $x>c$ means that each component of $x$ is larger than $c$; $C_0$, $C_1$, $C_2$ and $C_3$ are positive constants and $\norm{y}_M$ and $\norm{z}_N$ are two norms defined on $\xR^M$ and $\xR^N$ respectively;
\item[(iii)] the topological degree of $F(\cdot,0)$ with respect to $b$ and ${\mathcal O}$ is equal to $d_0 \neq 0$.
\end{enumerate}
Then the topological degree of $F(\cdot,1)$ with respect to $b$ and ${\mathcal O}$ is also equal to $d_0 \neq 0$; consequently, there exists at least one solution $v\in {\mathcal O}$ to the equation $f(v)=b$.
\end{thrm}
%
%
\section{Kolmogorov's Theorem}
\label{sec:kolmogorov}
The proof of the following theorem can be found in \cite{gal-12-comp}.

\begin{thrm}[Kolmogorov's Theorem] \label{thrm:kolmogorov}
Let $B$ be a Banach space, $p$ a real number such that $1\leq p <+\infty$ and $T>0$. Let $A\subset\xL^p((0,T);B)$. The subset $A$ is relatively compact in $\xL^p((0,T);B)$ if $A$ satisfies the three following conditions:
 
\begin{enumerate}
 \item [$(h1)$] For all $u\in A$, there exists $Pu\in\xL^p(\xR;B)$ such that $Pu=u$ almost everywhere in $(0,T)$ and $\norm{Pu}_{\xL^p(\xR;B)}\leq C$, where $C$ depends only on $A$.
 \item [$(h2)$] For all $\phi \in \mcal{C}_c^\infty(\xR,\xR)$, the family $\lbrace \int_{\xR}(Pu)\phi \dt, \, u\in A\rbrace$ is relatively compact in $B$.
 \item [$(h3)$] $\norm{Pu-Pu(.-\tau)}_{\xL^p(\xR;B)}\to 0$, as $\tau\to 0^+$, uniformly with respect to $u\in A$.
\end{enumerate}
 
\end{thrm}

\bibliographystyle{plain}
\bibliography{JCLKS_INSVD.bib}

\begin{thebibliography}{10}

\bibitem{ans-11-anl}
G.~Ansanay-Alex, F.~Babik, J.-C. Latch{\'e}, and D.~Vola.
\newblock An {L}$^2$-stable approximation of the {N}avier-{S}tokes convection
  operator for low-order non-conforming finite elements.
\newblock {\em International Journal for Numerical Methods in Fluids},
  66:555--580, 2011.

\bibitem{BLPS-fvca7}
F.~Babik, J.-C. Latch{\'e}, B.~Piar, and K.~Saleh.
\newblock {A Staggered Scheme with Non-conforming Refinement for the
  Navier-Stokes Equations}.
\newblock {\em Finite Volumes for Complex Applications VII-Methods and
  Theoretical Aspects: FVCA 7, Berlin, June 2014}, pages 87--95, 2014.

\bibitem{boy-14-sta}
F.~Boyer, F.~Dardalhon, C.~Lapuerta, and J.-C. Latch\'e.
\newblock Stability of a {C}rank-{N}icolson pressure correction scheme based on
  staggered discretizations.
\newblock {\em International Journal for Numerical Methods in Fluids},
  74:34--58, 2014.

\bibitem{boy-13-math}
F.~Boyer and P.~Fabrie.
\newblock {\em Mathematical Tools for the Study of the Incompressible
  Navier-Stokes Equations and Related Models}.
\newblock Springer, 2013.

\bibitem{bre-03-kor}
S.~C. Brenner.
\newblock {K}orn's inequalities for piecewise ${H}^1$ vector fields.
\newblock {\em Mathematics of Computation}, 73:1067--1087, 2003.

\bibitem{cro-73-con}
M.~Crouzeix and P.A. Raviart.
\newblock Conforming and nonconforming finite element methods for solving the
  stationary {S}tokes equations.
\newblock {\em RAIRO S\'erie Rouge}, 7:33--75, 1973.

\bibitem{deimling}
K.~Deimling.
\newblock {\em Nonlinear Functional Analysis}.
\newblock Springer, 1980.

\bibitem{des-08-hig}
O.~Desjardins, G.~Blanquart, G.~Balarac, and H.~Pitsch.
\newblock High order conservative finite difference scheme for variable density
  low {M}ach number turbulent flows.
\newblock {\em Journal of Computational Physics}, 227:7125--7159, 2008.

\bibitem{dip-89-ord}
R.~J. Di~Perna and P.L. Lions.
\newblock Ordinary differential equations, transport theory and {S}obolev
  spaces.
\newblock {\em Inventiones mathematicae}, 98:511--547, 1989.

\bibitem{eym-98-an}
R.~Eymard, T.~Gallou{\"e}t, M.~Ghilani, and R.~Herbin.
\newblock Error estimates for the approximate solutions of a nonlinear
  hyperbolic equation given by finite volume schemes.
\newblock {\em IMA Journal of Numerical Analysis}, 18:563--594, 1998.

\bibitem{eym-00-book}
R.~Eymard, T.~Gallou{\"e}t, and R.~Herbin.
\newblock {\em The finite volume method}.
\newblock Handbook for Numerical Analysis, Ph. Ciarlet and J.L. Lions Editors,
  North Holland, 2000.

\bibitem{eym-10-sushi}
R.~Eymard, T.~Gallou{\"e}t, and R.~Herbin.
\newblock Discretisation of heterogeneous and anisotropic diffusion problems on
  general non-conforming meshes. {SUSHI}: a scheme using stabilisation and
  hybrid interfaces.
\newblock {\em IMA Journal of Numerical Analysis}, 30:1009--1043, 2010.

\bibitem{gal-08-unc}
T.~Gallou{\"e}t, L.~Gastaldo, R.~Herbin, and J.-C. Latch{\'e}.
\newblock An unconditionally stable pressure correction scheme for compressible
  barotropic {N}avier-{S}tokes equations.
\newblock {\em Mathematical Modelling and Numerical Analysis}, 42:303--331,
  2008.

\bibitem{gal-12-comp}
T.~Gallou{\"e}t and J.-C. Latch{\'e}.
\newblock Compactness of discrete approximate solutions to parabolic {PDE}s-
  {A}pplication to a turbulence model.
\newblock {\em Communications in Pure and Applied Analysis}, 11:2371--2391,
  2012.

\bibitem{gas-10-unc}
L.~Gastaldo, R.~Herbin, and J.-C. Latch{\'e}.
\newblock An unconditionally stable finite element-finite volume pressure
  correction scheme for the drift-flux model.
\newblock {\em Mathematical Modelling and Numerical Analysis}, 44:251--287,
  2010.

\bibitem{gas-11-dis}
L.~Gastaldo, R.~Herbin, and J.-C. Latch{\'e}.
\newblock A discretization of phase mass balance in fractional step algorithms
  for the drift-flux model.
\newblock {\em IMA Journal of Numerical Analysis}, 31:116--146, 2011.

\bibitem{gir-86-fin}
V.~Girault and P.A. Raviart.
\newblock {\em Finite element methods for {Navier-Stokes} equations. {T}heory
  and algorithms}.
\newblock Springer-Verlag, 1986.

\bibitem{gra-16-unc}
D.~Grapsas, R.~Herbin, W.~Kheriji, and J.-C. Latch{\'e}.
\newblock An unconditionnally stable pressure correction scheme for the
  compressible {N}avier-{S}tokes equations.
\newblock {\em SMAI Journal of Computational Mathematics}, 2:51--97, 2016.

\bibitem{gue-96-som}
J.-L. Guermond.
\newblock Some implementations of projection methods for {N}avier-{S}tokes
  equations.
\newblock {\em Mathematical Modelling and Numerical Analysis}, 30:637--667,
  1996.

\bibitem{har-68-num}
F.H. Harlow and A.A. Amsden.
\newblock Numerical calculation of almost incompressible flow.
\newblock {\em Journal of Computational Physics}, 3:80--93, 1968.

\bibitem{har-71-num}
F.H. Harlow and A.A. Amsden.
\newblock A numerical fluid dynamics calculation method for all flow speeds.
\newblock {\em Journal of Computational Physics}, 8:197--213, 1971.

\bibitem{har-65-num}
F.H. Harlow and J.E. Welsh.
\newblock Numerical calculation of time-dependent viscous incompressible flow
  of fluid with free surface.
\newblock {\em Physics of Fluids}, 8:2182--2189, 1965.

\bibitem{her-13-pre}
R.~Herbin, W.~Kheriji, and J.-C. Latch\'e.
\newblock Pressure correction staggered schemes for barotropic one-phase and
  two-phase flows.
\newblock {\em Computers and Fluids}, 88:524--542, 2013.

\bibitem{her-14-imp}
R.~Herbin, W.~Kheriji, and J.-C. Latch\'e.
\newblock On some implicit and semi-implicit staggered schemes for the shallow
  water and {E}uler equations.
\newblock {\em Mathematical Modelling and Numerical Analysis}, 48:1807--1857,
  2014.

\bibitem{her-10-kin}
R.~Herbin and J.-C. Latch\'e.
\newblock Kinetic energy control in the {MAC} discretization of the
  compressible {N}avier-{S}tokes equations.
\newblock {\em International Journal of Finites Volumes}, 7, 2010.

\bibitem{her-13-expl}
{Herbin, R.}, {Latch{\'e}, J.-C.}, and {Nguyen, T.T.}
\newblock {E}xplicit staggered schemes for the compressible {E}uler equations.
\newblock {\em ESAIM: Proc.}, 40:83--102, 2013.

\bibitem{isis}
ISIS.
\newblock A {CFD} computer code for the simulation of reactive turbulent flows.
\newblock \\ \texttt{https://gforge.irsn.fr/gf/project/isis}.

\bibitem{lar-91-how}
B.~Larrouturou.
\newblock How to preserve the mass fractions positivity when computing
  compressible multi-component flows.
\newblock {\em Journal of Computational Physics}, 95:59--84, 1991.

\bibitem{lio-96-mat}
P.-L. Lions.
\newblock Mathematical topics in fluid mechanics -- {V}olume 1 --
  {I}ncompressible models.
\newblock volume~3 of {\em Oxford Lecture Series in Mathematics and its
  Applications}. Oxford University Press, 1996.

\bibitem{liu-07-conv}
C.~Liu and N.J. Walkington.
\newblock Convergence of numerical approximations of the incompressible
  {N}avier-{S}tokes equations with variable density and viscosity.
\newblock {\em SIAM Journal on Numerical Analysis}, 45:1287--1304, 2007.

\bibitem{mor-10-ske}
Y.~Morinishi.
\newblock Skew-symmetric form of convective terms and fully conservative finite
  difference schemes for variable density low-{M}ach number flows.
\newblock {\em Journal of Computational Physics}, 229:276--300, 2010.

\bibitem{nic-00-con}
F.~Nicoud.
\newblock Conservative high-order finite-difference schemes for low-{M}ach
  number flows.
\newblock {\em Journal of Computational Physics}, 158:71--97, 2000.

\bibitem{ran-92-sim}
R.~Rannacher and S.~Turek.
\newblock Simple nonconforming quadrilateral {S}tokes element.
\newblock {\em Numerical Methods for Partial Differential Equations},
  8:97--111, 1992.

\bibitem{sch-89-non}
F.~Schieweck and L.~Tobiska.
\newblock A nonconforming finite element method of upstream type applied to the
  stationary {N}avier-{S}tokes equation.
\newblock {\em Mathematical Modelling and Numerical Analysis}, 23:627--647,
  1989.

\bibitem{sch-96-ano}
F.~Schieweck and L.~Tobiska.
\newblock An optimal order error estimate for an upwind discretization of the
  {N}avier-{S}tokes equations.
\newblock {\em Numerical Methods for Partial Differential Equations},
  12:407--421, 1996.

\bibitem{she-92-one}
J.~Shen.
\newblock On error estimates of projection methods for {N}avier-{S}tokes
  equations: First-order schemes.
\newblock {\em SIAM Journal on Numerical Analysis}, 29(1):57--77, 1992.

\bibitem{she-92-oneII}
Jie Shen.
\newblock On error estimates of some higher order projection and
  penalty-projection methods for {N}avier-{S}tokes equations.
\newblock {\em Numerische Mathematik}, 62:49--73, 1992.

\end{thebibliography}
\end{document}